\documentclass[a4paper]{article}

\usepackage[backref=page]{hyperref}
\usepackage{geometry}
\usepackage{mathtools}
\usepackage{amssymb}
\usepackage{amsthm}
\usepackage{empheq}

\usepackage[nodayofweek]{datetime}

\usepackage{bbm}
\usepackage[]{dsfont}
\usepackage{authblk}
\usepackage{setspace}
\usepackage{enumerate}

\usepackage{upgreek}
\usepackage{centernot}

\usepackage{thmtools}
\usepackage{thm-restate}

\usepackage{graphicx}

\usepackage[toc,page]{appendix}

\usepackage[linesnumbered,commentsnumbered,ruled,vlined]{algorithm2e}
\usepackage{cleveref}

\usepackage{xcolor}

\newcommand{\norm}[1]{\left\lVert#1\right\rVert}

\newcommand{\chibar}{\bar{\chi}_A}

\newcommand\sbullet[1][.5]{\mathbin{\vcenter{\hbox{\scalebox{#1}{$\bullet$}}}}}

\newcommand{\restatetext}{Restatement}
\newcommand{\prooftext}{Proof on p.~}

\renewcommand{\b}[1]{{#1}}

\usepackage[noadjust]{cite}
\usepackage{booktabs}

\usepackage[inline]{enumitem}
\newcommand{\mylabel}[2]{#2\def\@currentlabel{#2}\label{#1}}

\usepackage{tikz}
\usepackage{pgfplotstable}

\usepackage{subcaption}
\usepackage{microtype}

\renewcommand{\ker}{\operatorname{Ker}}

\newcommand{\diag}{\operatorname{Diag}}

\newcommand{\Le}{\kappa}

\hypersetup{
    colorlinks,
    linkcolor={red!50!black},
    citecolor={blue!50!black},
    urlcolor={blue!80!black}
}

\newcounter{clmcounter}

\newtheorem{theorem}{Theorem}[section]
\newtheorem{definition}[theorem]{Definition}
\newtheorem{lemma}[theorem]{Lemma}
\newtheorem{claim}[theorem]{Claim}
\newtheorem{restatementclaim}[clmcounter]{Claim}
\newtheorem{proposition}[theorem]{Proposition}
\newtheorem{remark}[theorem]{Remark}
\newtheorem{corollary}[theorem]{Corollary}
\newtheorem{example}[theorem]{Example}

\newenvironment{claimproof}[1][\proofname]
{
    \proof[Proof of #1]
}
{
    \endproof

}

\newtheoremstyle{case}{}{}{\itshape}{1em}{}{:}{ }{}
\theoremstyle{case}

\definecolor{darkred}{RGB}{180, 0, 0}
\definecolor{darkgreen}{rgb}{0, 0.6, 0}
\definecolor{darkblue}{RGB}{51,51,178}
\definecolor{lightgray}{RGB}{231,231,231}
\definecolor{lightblue}{RGB}{180,180,254}
\definecolor{lightred}{HTML}{FEB4B4}
\definecolor{darkcyan}{HTML}{7FBFBF}

\DeclareMathOperator*{\argmax}{arg\,max}
\DeclareMathOperator*{\argmin}{arg\,min}

\renewcommand{\phi}{\varphi}
\renewcommand{\rho}{\varrho}
\newcommand{\T}{\top}
\newcommand{\R}{\mathbb{R}}
\newcommand{\eps}{\varepsilon}
\newcommand{\supp}{\mathrm{supp}}
\newcommand{\set}[1]{\left\{ #1 \right\}}
\newcommand{\pr}[2]{\langle #1, #2 \rangle}

\newcommand{\potentialshorthand}{\Gamma}

\author[1]{Daniel Dadush}
\author[2]{Sophie Huiberts}
\author[3]{Bento Natura}
\author[4]{L{\'{a}}szl{\'{o}} A. V{\'{e}}gh}
\affil[1]{Centrum Wiskunde \& Informatica, The Netherlands}
\affil[2]{Columbia University, USA \thanks{This work was done while SH was at Centrum Wiskunde \& Informatica, The Netherlands.}}
\affil[3]{Brown University, USA and Georgia Tech, USA \thanks{This work was done while BN was at the London School of Economics and Political Science, United Kingdom.}}
\affil[4]{Department of Mathematics, London School of
  Economics and Political Science, United Kingdom}

\usepackage{accents}
\newcommand{\ubar}[1]{\underaccent{\bar}{#1}}

\newcommand{\Rx}{\mathit{Rx}}
\newcommand{\Rs}{\mathit{Rs}}

\newcommand{\barx}{\bar{x}}	
\newcommand{\ubarx}{\ubar{x}}

\newcommand{\bars}{\bar{s}}
\newcommand{\ubars}{\ubar{s}}

\newcommand{\rk}{\mathrm{rk}}
\newcommand{\lal}{\mathrm{ll}}

\newcommand{\as}{\mathrm{a}}

\newcommand{\cs}{\mathrm{c}}

\renewcommand{\epsilon}{\varepsilon}

\newcommand{\circuits}{\mathcal{C}}

\usepackage{bm}
\newcommand{\llspartition}[1]{\bm{#1}}

\title{A scaling-invariant algorithm for linear programming whose
  running time depends only on the constraint matrix\thanks{
This project has received funding from the European Research Council (ERC) under the European Union's Horizon 2020 research and innovation programme: DD and SH from grant agreement no.  805241-QIP,  BN and LAV from grant agreements no. 757481-ScaleOpt.  A preliminary version of this paper has appeared in the proceedings of the 52nd Annual ACM Symposium on Theory of Computing (STOC) \cite{DHNV20}.}}
\date{}

\makeatletter
\let\original@algocf@latexcaption\algocf@latexcaption
\long\def\algocf@latexcaption#1[#2]{\@ifundefined{NR@gettitle}{\def\@currentlabelname{#2}}{\NR@gettitle{#2}}\original@algocf@latexcaption{#1}[{#2}]}
\makeatother

\begin{document}
\maketitle

\begin{abstract}
Following the breakthrough work of Tardos (Oper. Res. '86) in the bit-complexity
model, Vavasis and Ye (Math. Prog. '96) gave the first exact algorithm for
linear programming in the real model of computation with running time depending
only on the constraint matrix. For solving a linear program (LP)
$\max\, c^\T x,\:  Ax = b,\:  x \geq 0,\:  A \in \R^{m \times n}$, Vavasis and Ye
developed a primal-dual interior point method using a \emph{`layered least
squares'} (LLS) step, and showed that $O(n^{3.5} \log
(\bar{\chi}_A+n))$ iterations suffice to solve (LP) exactly, where $\bar{\chi}_A$ is a condition measure
controlling the size of solutions to linear systems related to $A$. 

Monteiro and Tsuchiya (SIAM J. Optim. '03), noting that the central path is
invariant under rescalings of the columns of $A$ and $c$, asked whether there
exists an LP algorithm depending instead on the measure $\bar{\chi}^*_A$,
defined as the minimum $\bar{\chi}_{AD}$ value achievable by a column rescaling $AD$ of
$A$, and gave strong evidence that this should be the case. We resolve
this open question affirmatively. 

Our first main contribution is an $O(m^2 n^2 + n^3)$ time algorithm which works
on the linear matroid of $A$ to compute a nearly optimal diagonal
rescaling $D$ satisfying $\bar{\chi}_{AD} \leq n(\bar{\chi}_A^*)^3$. This
algorithm also allows us to approximate the value of $\bar{\chi}_A$ up to a
factor $n (\bar{\chi}_A^*)^2$. This result is in surprising contrast to that of
Tun{\c{c}}el (Math.  Prog. '99), who showed NP-hardness for approximating
$\bar{\chi}_A$ to within $2^{{\rm poly}({\rm rank}(A))}$. The key insight for
our algorithm is to work with ratios $g_i/g_j$ of circuits of $A$---i.e.,~minimal linear dependencies $Ag=0$---which allow us to approximate the
value of $\bar{\chi}_A^*$ by a maximum geometric mean cycle computation in what we
call the \emph{`circuit ratio digraph'} of $A$.   

While this resolves Monteiro and Tsuchiya's question by appropriate
preprocessing, it falls short of providing either a truly scaling invariant
algorithm or an improvement upon the base LLS analysis. In this vein, as our
second main contribution we develop a \emph{scaling invariant} LLS algorithm,
which uses and dynamically maintains improving estimates of the circuit ratio
digraph, together with a refined potential function based analysis for LLS
algorithms in general. With this analysis, we derive an improved $O(n^{2.5} \log
(n)\log (\bar{\chi}^*_A+n))$ iteration bound for optimally solving (LP) using
our algorithm. The same argument also yields a factor $n/\log n$ improvement on the
iteration complexity bound of the original Vavasis-Ye algorithm.
\end{abstract}

\section{Introduction}
The linear programming (LP) problem in primal-dual form is to solve
\begin{equation}
\label{LP_primal_dual} \tag{LP}
\begin{aligned}
\min \; &c^\top x \quad \\
Ax& =b \\
x &\geq 0, \\
\end{aligned}
\quad\quad\quad
\begin{aligned}
\max \; & y^\top b \\
A^\top y + s &= c \\
s & \geq 0, \\
\end{aligned}
\end{equation}
where $A\in \R^{m\times n}$, ${\rm rank}(A) = m$, $b\in \R^m$, $c\in
\R^n$ are given in the input, and $x,s\in\R^n$, $y\in \R^m$ are the variables.
The program in $x$ will be referred to as the primal problem and the program in $(y,s)$ as the
 dual problem.

Khachiyan~\cite{Khachiyan79} used the ellipsoid method to give the first
polynomial time LP algorithm in the bit-complexity model, that is, polynomial in the
bit description length of $(A,b,c)$. \b{An outstanding open question is the existence of a \emph{strongly polynomial}
algorithm for LP, listed by Smale as one of the most prominent mathematical challenges for the 21st century \cite{Smale98}. Such an algorithm amounts to solving LP using ${\rm poly}(n,m)$
basic arithmetic operations in the real model of computation.\footnote{In the bit-complexity model, a further requirement is that the algorithm must be in PSPACE.}}
Known strongly polynomially
solvable LP problems classes include: feasibility for two variable per inequality
systems~\cite{Megiddo83}, the minimum-cost circulation problem~\cite{Tardos85},
the maximum generalized flow problem~\cite{Vegh17,OV17}, and
discounted Markov decision problems~\cite{Ye2005,Ye2011}.

Towards this goal, the principal line of attack has been to develop LP algorithms whose running time is bounded in terms of natural \emph{condition measures}. Such condition
measures attempt to  measure the ``intrinsic complexity'' of LPs.
 An
important line of work in this area has been to parametrize LPs by the
``niceness'' of their solutions (e.g.~the depth of the most interior point),
where relevant examples include the Goffin measure~\cite{Goffin80} for conic
systems and Renegar's distance to ill-posedness for general
LPs~\cite{Renegar94,Renegar95}, and bounded ratios between the nonzero
entries in basic feasible solutions \cite{Chubanov-s-poly,KitaharaMizuno2013}.

\paragraph{Parametrizing by the constraint matrix}
A second line of
research, and the main focus of this work,
focuses on the complexity of the constraint
matrix $A$.
The first breakthrough in this area was given by Tardos~\cite{Tardos86}, who
showed that if $A$ has integer entries and all square submatrices of $A$ have
determinant at most $\Delta$ in absolute value, then \eqref{LP_primal_dual} can
be solved in \b{poly$(n,m,\log \Delta)$ arithmetic operations, independent of the encoding length of the vectors $b$ and $c$}.
This is achieved by finding the exact solutions to $O(nm)$ rounded LPs derived from the original LP, with the right hand side vector and cost function being integers of absolute
value bounded in terms of $n$ and $\Delta$. From $m$ such rounded problem instances, one can
infer, via proximity results, that  $x_i=0$ must hold
for every optimal solution for some index $i$.  The
process continues by induction until the optimal primal face is identified.

\paragraph{Path-following methods and  the Vavasis--Ye algorithm} In a seminal work, Vavasis and
Ye~\cite{Vavasis1996} introduced a new type of interior-point method
that optimally solves  \eqref{LP_primal_dual} within $O(n^{3.5} \log
(\bar{\chi}_A+n))$ iterations, where the condition number $\bar{\chi}_A$ controls the size of solutions to certain linear systems related
to the kernel of $A$ (see Section~\ref{sec:rescale} for the formal
  definition).

Before detailing the Vavasis--Ye (henceforth VY) algorithm, we recall the basics
of path following interior-point methods.
If both the primal and dual problems in \eqref{LP_primal_dual} are strictly feasible, the {\em central path} for \eqref{LP_primal_dual} is the curve $((x(\mu),y(\mu),s(\mu)): \mu >
0)$ defined by
\begin{equation}\tag{CP}\label{eq:central-path}
\begin{aligned}
x({\mu})_i s({\mu})_i &= \mu, \quad \forall i \in [n]\\ Ax(\mu) &= b, ~x(\mu) > 0, \\
A^\T y(\mu) + s(\mu) &= c, ~s(\mu) > 0,
\end{aligned}
\end{equation}
which converges to complementary optimal primal and dual solutions $(x^*,y^*,s^*)$ as
$\mu \rightarrow 0$,
 recalling that
the duality gap at time $\mu$ is exactly $x(\mu)^\top s(\mu) = n \mu$. We thus
refer to $\mu$ as the normalized \b{duality} gap. Methods
that ``follow the path'' generate iterates that stay in a certain neighborhood
around it while trying to achieve rapid multiplicative progress w.r.t.~to $\mu$,
where given $(x,y,s)$ `close' to the path, we define the normalized duality gap as $\mu(x,y,s) =
\sum_{i=1}^n x_i s_i/n$.
Given a target
parameter $\mu'$ and starting point close to the path at parameter $\mu$,
standard path following methods~\cite{Gonzaga92} can compute a point at
parameter below $\mu'$ in at most $O(\sqrt{n} \log(\mu/\mu'))$ iterations, and
hence the quantity $\log(\mu/\mu')$ can be usefully interpreted as the length of
the corresponding segment of the central path.

\paragraph{Crossover events and layered least squares steps}
At a very high level, Vavasis and Ye show that the central path can be decomposed into at
most $\binom{n}{2}$ short but curved segments,
possibly joined by long (apriori unbounded) but very straight segments. At the
end of each curved segment, they show that a new ordering relation $x_i(\mu) >
x_j(\mu)$---called a \emph{`crossover event'}---is implicitly learned. This
inequality did not hold at the start of the segment, but is guaranteed to
hold at every point from the end of the segment onwards. These
$\binom{n}{2}$  relations give a combinatorial way to measure
progress along the central path. In contrast to Tardos's algorithm,
where the main progress is setting variables to zero explicitly, the variables
participating in  crossover events cannot be identified; the analysis only shows their
existence.

At a technical level, the VY-algorithm is a variant of the
Mizuno--Todd--Ye~\cite{MTY} predictor-corrector method (MTY P-C). In predictor-corrector
methods, corrector steps bring an iterate closer to the path, i.e., improve
centrality, and predictor steps ``shoot down'' the path, i.e., reduce $\mu$
without losing too much centrality. Vavasis and Ye's main algorithmic innovation was the
introduction of a new predictor step, called the \emph{`layered least squares'}
(LLS) step, which crucially allowed them to cross each aforementioned
``straight'' segment of the central path in \emph{a single step}, recalling that
these straight segments may be arbitrarily long. To traverse the short and curved
segments of the path, the standard predictor step, known as \emph{affine
scaling} (AS), in fact suffices.

To compute the LLS direction, the variables are
decomposed into `layers' $J_1\cup J_2\cup\ldots\cup J_p=[n]$. The goal
of such a decomposition is to eventually learn a refinement of the
optimal partition of the variables
$ B^* \cup N^*=[n]$, where $B^* := \{i \in [n]: x^*_i
> 0\}$ and $N^* := \{i \in [n]: s^*_i > 0\}$ for the limit optimal
solution $(x^*,y^*,s^*)$.

 The primal affine scaling direction can be equivalently described by solving a
weighted least squares problem in $\ker(A)$, with respect to a weighting
defined according to the current iterate. The primal LLS direction is obtained
by solving a series of weighted least squares problems, starting with focusing
only on the final layer $J_p$. This solution is gradually extended to the higher
layers (i.e., layers with lower indices). The dual directions have analogous interpretations, with the solutions
on the layers obtained in the opposite direction, starting with $J_1$. If we use
the two-level layering $J_1=B^*$, $J_2=N^*$, and are sufficiently
close to the limit $(x^*,y^*,s^*)$ of the
central path, then the LLS step reaches an exact optimal solution in a
single step. We note that standard AS steps generically never find an exact optimal solution, and thus some form of ``LLS rounding'' \b{in the final iteration} is
always necessary to achieve finite termination \b{with an exact optimal solution.}

Of course, guessing $B^*$ and $N^*$ correctly is just as hard as
solving \eqref{LP_primal_dual}. Still, if we work with a ``good'' layerings,
these will reveal new information about the ``optimal order'' of the
variables, where $B^*$ is placed on higher layers than $N^*$.
 The crossover events correspond to swapping two wrongly ordered variables
into the correct ordering. Namely, a variable  $i\in B^*$ and $j\in
N^*$ are currently ordered on the same layer, or $j$ is in a higher layer than $i$. After the crossover event, $i$ will always be
placed on a  higher layer than $j$.

\paragraph{Computing good layerings and the $\bar{\chi}_A$ condition measure}
Given the above discussion, the obvious question is how to come up with ``good''
layerings?
The philosophy
behind LLS can be stated as saying that if modifying a set of variables $x_I$ barely
affects the variables in $x_{[n] \setminus I}$ (recalling that movement is
constrained to $\Delta x \in \ker(A)$), then one should optimize over $x_I$ without
regard to the effect on $x_{[n] \setminus I}$; hence $x_I$ should be
placed on lower layers.

VY's strategy for computing such layerings was to directly use the size of the
coordinates of the current iterate $x$ (where $(x,y,s)$ is a point
near the central path).
In particular, assuming $x_1\ge x_2\ge \ldots \ge x_n$, the layering
$J_1 \cup J_2\cup
\ldots \cup J_p = [n]$ corresponds to consecutive intervals constructed in
decreasing order of $x_i$ values. The break between $J_i$ and $J_{i+1}$
occurs if the gap $x_r/x_{r+1}
> g$, where $r$ is the rightmost element of $J_i$ and $g > 0$ is a threshold
parameter. Thus, the expectation is that if $x_i> g x_j$, then a
small multiplicative change to $x_j$, subject to moving in $\ker(A)$, should induce a small multiplicative change to $x_i$.
By proximity to the central path, the dual ordering is reversed as mentioned above.

The threshold $g$ for which this was justified in the VY-algorithm is a function of the
$\bar{\chi}_A$ condition measure. We now provide a convenient
definition that immediately
yields this justification (see Proposition~\ref{prop:subspace-chi}).
Letting $W = \ker(A)$ and $\pi_I(W) = \{x_I : x \in W\}$, we define
$\bar{\chi}_A := \bar{\chi}_W$ as the minimum number $M \geq 1$ such that for
any $\emptyset \neq I \subseteq [n]$ and  $z \in \pi_I(W)$, there
exists $y \in W$ with $y_I = z$ and $\|y\| \leq M \|z\|$. Thus,
a change of norm $\epsilon$ in the 
variables in $I$ can be lifted to a change of norm at most
$\bar{\chi}_A\epsilon$ in the variables in $[n]\setminus I$.
Crucially, $\bar{\chi}$ is a ``self-dual'' quantity. That is,
$\bar{\chi}_W = \bar{\chi}_{W^\perp}$, where $W^\perp = {\rm range}(A^\T)$ is the movement subspace for the
dual problem, justifying the reversed layering for the dual (see
Sections~\ref{sec:rescale} for more
details).

\paragraph{The question of scale invariance and  $\bar{\chi}^*_A$}
While the VY layering procedure is powerful, its properties are somewhat mismatched with those of the central path. In
particular, variable ordering information has \emph{no intrinsic meaning}
on the central path, as the path itself is \emph{scaling
  invariant}. Namely, the central path point $(x(\mu),y(\mu),s(\mu))$
w.r.t.\ the problem instance $(A,b,c)$ is
in bijective correspondence with the central path point $(D^{-1} x(\mu), D
y(\mu), D
s(\mu)))$ w.r.t.\ the problem instance $(AD,Dc,b)$
for any positive diagonal matrix $D$.
The standard path following algorithms are
also scaling invariant in this sense.

This lead Monteiro and Tsuchiya~\cite{Monteiro2003} to
ask whether a scaling invariant LLS algorithm exists. They noted that any such
algorithm would then depend on the potentially much smaller parameter
\begin{equation}
\label{chibar_star}
\bar{\chi}^*_A := \inf_D \bar{\chi}_{AD}\, ,
\end{equation}
where the infimum is taken over the set of $n \times n$ positive diagonal
matrices. Thus, Monteiro and Tsuchiya's question can be rephrased as to whether
there exists an exact LP algorithm with running time poly$(n,m,\log\bar{\chi}^*_A)$.

Substantial progress on this question was made in the followup
works~\cite{MonteiroT05,LMT09}. The paper \cite{MonteiroT05} showed that the
number of iterations of the MTY predictor-corrector
algorithm~\cite{MTY} can get from $\mu_0>0$ to
$\eta>0$ on the central path in $$O\left(n^{3.5}\log
  \bar\chi^*_A+\min\{n^2\log\log(\mu^0/\eta), \log(\mu^0/\eta)\}\right)$$
iterations.
This is attained by showing that the standard AS steps are reasonably close to the LLS steps. This proximity can be used to
show that the AS steps can traverse the ``curved'' parts of the central
path in the same iteration complexity bound as the VY
algorithm. Moreover, on the ``straight'' parts of the path, the rate
of progress amplifies geometrically, thus attaining a
 $\log \log$ convergence on these parts.
Subsequently, \cite{LMT09} developed an affine invariant
\emph{trust region} step, which traverses the full path in $O(n^{3.5}
\log(\bar{\chi}_A^*+n))$ iterations. However, \b{the running time of }each iteration is weakly polynomial
in $b$ and $c$. The question of developing an LP algorithm with
complexity bound
poly$(n,m,\log\bar{\chi}_A^*)$ thus remained open.

A related open problem to the above is whether it is possible to compute a
near-optimal rescaling $D$ for program~\eqref{chibar_star}? This would give an
alternate pathway to the desired LP algorithm by simply preprocessing the matrix
$A$. The related question of approximating $\bar{\chi}_A$ was already studied by
Tun{\c{c}}el~\cite{Tuncel1999}, who showed NP-hardness for approximating $\bar{\chi}_A$
to within a $2^{{\rm poly}({\rm rank}(A))}$ factor. Taken at face value, this
may seem to suggest that approximating the rescaling $D$ should be hard.

A further open question is whether Vavasis and Ye's cross-over analysis
can be improved. Ye showed in~\cite{Ye:2006}  that the
iteration complexity can be reduced to $O(n^{2.5} \log (\bar{\chi}_A+n))$ for
feasibility problems and further to $O(n^{1.5} \log(\bar{\chi}_A+n))$ for
homogeneous systems, though the $O(n^{3.5} \log(\bar{\chi}_A+n))$ bound for
optimization has not been improved since~\cite{Vavasis1996}.

\subsection{Our contributions}
In this work, we resolve all of the above questions in the affirmative. We
detail our contributions below.

\medskip

\noindent{\em 1. Finding an approximately optimal rescaling.} As our first contribution, we give an $O(m^2 n^2 + n^3)$ time algorithm that
works on the linear matroid of $A$ to compute a diagonal rescaling
matrix $D$ which achieves $\bar{\chi}_{AD} \leq n (\bar{\chi}_A^*)^3$, given any $m
\times n$ matrix $A$. Furthermore, this same algorithm allows us to approximate
$\bar{\chi}_A$ to within a factor $n(\bar{\chi}_A^*)^2$. The algorithm bypasses
Tun\c{c}el's hardness result by allowing the approximation factor to depend on
$A$ itself, namely on $\bar{\chi}_A^*$. This gives a simple first answer to
Monteiro and Tsuchiya's question: by applying the Vavasis-Ye algorithm directly
on the preprocessed $A$ matrix, we may solve any LP with constraint matrix $A$ using $O(n^{3.5}\log(
\bar{\chi}^*_A+n))$ iterations. Note that the approximation factor
$n(\bar{\chi}_A^*)^2$ increases the runtime only by a constant factor.

To achieve this result, we work  with the circuits of $A$, where a
circuit $C\subseteq [n]$ \b{corresponds to an inclusion-wise minimal set of linearly dependent columns.}
With each circuit, we can associate a vector
$g^C\in\ker(A)$ with $\supp(g^C)=C$ that is unique up to scaling.
By the {\em `circuit ratio'} $\kappa_{ij}$ associated with the pair of nodes $(i,j)$, we mean the  largest ratio $|g^C_j/g^C_i|$ taken over every
circuit $C$ of $A$ such that $i,j\in C$.
As our
first observation, we show that the maximum of all circuit ratios,
which we call the {\em `circuit imbalance measure'}, in fact characterizes $\bar{\chi}_A$ up to a
factor $n$. This measure was first studied by Vavasis~\cite{Vavasis1994}, who
showed that it lower bounds $\bar{\chi}_A$, though, as far as we are aware, our
upper bound is new. The circuit ratios of each pair $(i,j)$ induce a
weighted directed graph we call the \emph{`circuit ratio digraph'} of $A$. From
here, our main result is that $\bar{\chi}^*_A$ is up to a factor $n$ equal to
the maximum geometric mean cycle in the circuit ratio digraph. Our
algorithm populates the circuit ratio digraph with \b{approximations of the $\Le_{ij}$ ratios for each $i,j\in [n]$ using standard techniques from matroid theory}, and then computes a rescaling by solving the dual of
the maximum geometric mean ratio cycle on the \emph{`approximate circuit ratio digraph'}.

\medskip

\noindent{\em 2. Scaling invariant LLS algorithm.} While the above yields an LP
algorithm with  poly$(n,m,\log\bar{\chi}^*_A)$ running time, it does not
satisfactorily address Monteiro and Tsuchiya's question on a scaling invariant
algorithm. As our second contribution, we use the circuit ratio digraph directly
to give a natural scaling invariant LLS layering algorithm together with a
scaling invariant crossover analysis.

At a conceptual level, we show that the circuit ratios give a scale invariant
way to measure whether `$x_i >x_j$' and enable a natural layering algorithm. Assume for now that the circuit imbalance value
$\Le_{ij}$ is known for every pair $(i,j)$.
Given the circuit ratio graph induced by the $\Le_{ij}$'s and given a primal point $x$
near the path, our layering algorithm can be described as follows. We first rescale
the variables so that $x$ becomes the all ones vector, which rescales $\Le_{ij}$ to
$\Le_{ij} x_i/x_j$.  We then restrict the graph to its edges of length $\Le_{ij}x_i/x_j\geq 1/{\rm
poly}(n)$---the \emph{long edges} of the (rescaled) circuit ratio graph---and let the layering $J_1 \cup J_2\cup \ldots \cup J_p$ be a topological ordering of its
strongly connected components (SCC) with edges going from left to right.
Intuitively, variables that ``affect each other'' should be in the same layer,
which motivates the SCC definition.

We note that our layering algorithm does not have access to the true
circuit ratios $\Le_{ij}$; these are in fact NP-hard to compute. Getting a good
enough initial estimate for our purposes however is easy: we let $\hat\Le_{ij}$
be the ratio corresponding to an {\em arbitrary} circuit containing $i$ and $j$.
This already turns out to be within a factor $(\bar\chi^*_A)^2$ from the true value
$\Le_{ij}$---recall this is the maximum over all such circuits. Our layering
algorithm learns better circuit ratio estimates if the `lifting costs
of our SCC layering, i.e., how much it costs to lift changes from lower layer
variables to higher layers (as in the definition of $\bar{\chi}_A$), are larger
than we expected them to be based on the previous estimates.

We develop a scaling-invariant analogue of cross-over events as follows.
Before the crossover event, ${\rm poly}(n)(\bar\chi^*_A)^{n}>\Le_{ij} x_i/x_j$, and
after the crossover event, ${\rm poly}(n)(\bar\chi^*_A)^{n}<\Le_{ij} x_i/x_j$
for
all further central path points. Our analysis relies on
$\bar{\chi}_A^*$ in only a minimalistic way, and does not require an estimate on the value of $\bar{\chi}_A^*$.
  Namely, it is only used to show
that if $i,j\in J_q$, for a layer $q \in [p]$, then the rescaled circuit ratio $\Le_{ij} x_i/x_j$ is
in the range $({\rm poly}(n) \bar{\chi}_A^*)^{\pm O( |J_q|)}$. The argument to show this crucially
utilizes the maximum geometric mean cycle characterization. Furthermore, unlike
prior analyses \cite{Vavasis1996,Monteiro2003}, our definition of a ``good'' layering (i.e., `balanced'
layerings, see Section~\ref{sec:layering}), is completely independent of
$\bar{\chi}^*_A$.

\medskip

\noindent{\em 3. Improved potential analysis.} As our third contribution, we improve the Vavasis--Ye crossover analysis
using a new and simple potential function based approach.  When applied to our
new LLS algorithm, we derive an $O(n^{2.5} \log n \log(\bar{\chi}_A^*+n))$
iteration bound for path following, improving the polynomial term by an
$\Omega(n/\log n)$ factor compared to the VY analysis.

Our potential function can be seen as a fine-grained version of the
crossover events as described above. In case of such a crossover
event, it is guaranteed that in every subsequent iteration, $i$ is in
a layer before $j$. We analyze less radical changes instead: an
``event'' parametrized by $\tau$ means that $i$ and $j$ are currently together on a
layer of size $\le \tau$, and after the event,
$i$ is on a layer before
$j$, or if they are together on the same layer, then this layer must have  size $\ge2\tau$.
For every LLS step, we can find a parameter $\tau$ such that an event
of this type happens concurrently for at least
$\tau-1$ pairs within
the next
$O(\sqrt{n} \tau
\log(\bar\chi_A^*+n))$ iterations,

\medskip

Our improved analysis is also applicable to the original VY-algorithm. Let us now comment on the relation between the VY-algorithm and our new algorithm.
The VY-algorithm starts a new layer once $x_{\pi(i)}> g x_{\pi(i+1)}$ between
two consecutive variables where the permutation $\pi$ is a
non-increasing order of the $x_i$
variables, and $g={\rm
  poly}(n) \bar\chi_A$. Setting the initial `estimates' $\hat
\Le_{ij}=\bar\chi_A$ for a suitable polynomial, our algorithm runs the same way as the VY
algorithm. Using these estimates, the layering procedure becomes much
simpler: there is no need to verify `balancedness' as in our
algorithm.

However, using estimates $\hat\Le_{ij}=\bar\chi_A$ has drawbacks. Most importantly, it does not
give a lower bound on the true circuit ratio $\Le_{ij}$---to the contrary, $g$
will be an upper bound. In effect, this causes VY's layers to be  ``much
larger'' than ours, and for this reason, the connection to $\bar\chi^*_A$ is lost.
Nevertheless, our potential function analysis can still be adapted to the VY-algorithm to obtain the same $\Omega(n/\log n)$ improvement on the iteration
complexity bound; see Section~\ref{sec:VY} for more details.

\subsection{Related work}

Since the seminal works of Karmarkar~\cite{Karmarkar84} and
Renegar~\cite{Renegar1988}, there has been a tremendous amount of work on
speeding up and improving interior-point methods. In contrast to the present
work, the focus of these works has mostly been to improve complexity of
\emph{approximately solving} LPs. Progress has taken many forms, such as the
development of novel barrier methods, such as Vaidya's volumetric
barrier~\cite{Vaidya1989} and the recent entropic barrier of Bubeck and
Eldan~\cite{BE14} and the weighted log-barrier of Lee and
Sidford~\cite{LS14,LS19}, together with new path following techniques, such as
the predictor-corrector framework~\cite{Mehrotra1992,MTY}, as well as advances
in fast linear system solving~\cite{ST04,LS15}. For this last line, there has been
substantial progress in improving IPM by amortizing the cost of the iterative
updates, and working with approximate computations, \b{see
e.g.~\cite{Renegar1988,Vaidya1989} for classical results.} Recently, Cohen, Lee and
Song~\cite{CLS19} developed a new inverse maintenance scheme to get a randomized
$\tilde{O}(n^{\omega}\log(1/\eps))$-time algorithm for $\eps$-approximate LP,
which was derandomized by van den Brand~\cite{vdb20}; \b{here $\omega\approx 2.37$ is the matrix multiplication exponent}. \b{A very recent result by van den Brand et al. \cite{vdb20-tall-dense} obtained a randomized $\tilde O(nm+m^3)$ algorithm.
For special classes of LP
such as network flow and matching problems, even faster algorithms have been obtained using, among other techniques, fast
Laplacian solvers, see e.g.~\cite{Daitch2008,Madry2013,Brand2020,Brand2021}. }
Given the progress
above, we believe it to be an interesting problem to understand to what extent
these new numerical techniques can be applied to speed up LLS computations,
though we expect that such computations will require very high precision. We
note that no attempt has been made in the present work to optimize the
complexity of the linear algebra.

Subsequent to the conference version of this paper \cite{DHNV20}, some of the authors extended Tardos's framework to the real model of computation \cite{DadushNV20}, showing that  poly$(n,m,\log\bar{\chi}_A)$ running time can be achieved using approximate solvers in a black box manner. Combined with \cite{vdb20}, one obtains a deterministic $O(mn^{\omega+1} \log^{O(1)}(n) \log(\bar\chi_A))$ LP algorithm; using the initial rescaling subroutine from this paper, the dependence can be improved to ${\bar\chi}^*_A$ resulting in a running time of 
$O(mn^{\omega+1} \log^{O(1)}(n) \log(\bar\chi_A^* + n))$.
A weaker extension of Tardos's framework to the real model of computation was previously given by
Ho and Tun{\c{c}}el~\cite{ho2002}.

With regard to LLS algorithms, the original VY-algorithm required explicit
knowledge of $\bar{\chi}_A$ to implement their layering algorithm. The paper \cite{MMT98}
showed that this could be avoided by computing all LLS steps associated with $n$
candidate partitions and picking the best one. In particular, they showed that
all such LLS steps can be computed in $O(m^2 n)$ time. In \cite{Monteiro2003},
an alternate approach was presented to  compute an LLS partition directly from the
coefficients of the AS step. We note that these methods crucially rely on the
variable ordering, and hence are not scaling invariant. Kitahara and
Tsuchiya~\cite{KitaharaTsuchiya2013}, gave a 2-layer LLS step which achieves a
running time depending only on $\bar{\chi}_A^*$ and right-hand side $b$, but with no
dependence on the objective, assuming the primal feasible region is bounded.

A series of papers have studied the central path from a differential geometry perspective. Monteiro and
Tsuchiya~\cite{Monteiro2008} showed that a curvature integral of the central
path, first introduced by Sonnevend, Stoer, and Zhao~\cite{Sonnevend1991}, is in fact upper bounded by
$O(n^{3.5} \log(\bar{\chi}^*_A+n))$. This has been extended to SDP and symmetric cone programming \cite{kakihara2014}, and also studied in the context of information geometry \cite{kakihara2013}.

Circuits have appeared in several papers on linear and integer
optimization (see \cite{DKS19} and references within).
The idea of using circuits within the context of LP algorithms also appears
in~\cite{DHL15}. They develop a circuit augmentation framework for LP (as well ILP) and
show that simplex-like algorithms that take steps according to the ``best
circuit'' direction achieves linear convergence, though these steps are hard
to compute. Recently,  \cite{DKNV22} used circuit imbalance measures to obtain a circuit augmentation algorithm for LP with poly$(n,\log(\bar\chi_A))$ iterations.
We refer to \cite{ENV21} for an overview on circuit imbalances and their applications.

\medskip

Our algorithm makes progress towards strongly polynomial solvability of LP, by
improving the dependence poly$(n,m,\log\bar\chi_A)$ to poly$(n,m,\log\bar\chi^*_A)$.
However, in a remarkable recent paper, Allamigeon, Benchimol, Gaubert, and Joswig~\cite{allamigeon2018} have shown, using tools from tropical geometry, that
path-following methods for the standard logarithmic barrier {\em cannot} be
strongly polynomial. In particular, they give a parametrized family of
instances, where, for sufficiently large parameter values, any sequence of
iterations following the central path must be of exponential length---thus,
$\bar\chi^*_A$ will be doubly exponential. We note that very recently, Allamigeon, Gaubert, and Vandame \cite{AGV22} strengthened this result, showing that no interior point method using a self-concordant barrier function may be strongly polynomial.

As a further recent development, Allamigeon, Dadush, Loho, Natura, and V\'egh \cite{ADLNV22} complement these negative results by giving a weakly polynomial interior point method that always terminates in at most $O(2^n n^{1.5}\log n)$ iterations---even when $\log\bar\chi^*_A$ is unbounded. Moreover, their interior point method is `universal': it matches the number of iterations of any interior point method that uses a self-concordant barrier function up to a factor $O(n^{1.5} \log n)$. The `subspace LLS' step used in the paper is a generalization of the LLS step, using restricted movements in general subspaces, not only coordinate subspaces.

\subsection{Organization}
The rest of the paper is organized as follows. We conclude this section by introducing some notation. 
Section~\ref{sec:rescale} discusses our results on the circuit imbalance measure.
It starts with Section~\ref{sec:bar-chi}
on  the necessary background on the condition measures
$\bar{\chi}_A$ and $\bar{\chi}^*_A$.
Section~\ref{sec:circ} introduces the circuit imbalance measure, and formulates and explains all main results of Section~\ref{sec:rescale}. The proofs are given in the rest of the sections:
basic properties in Section~\ref{sec:basic-kappa-proof},
 the min-max characterization
in Section~\ref{sec:min-max}, the circuit finding algorithm in
Section~\ref{sec:matroid}, the algorithms for
approximating $\bar{\chi}^*_A$ and $\bar{\chi}_A$ in Section~\ref{sec:approx-alg}.

In Section~\ref{sec:scaling-invariant}, we develop our scaling invariant interior-point method. Interior-point preliminaries are given in
Section~\ref{sec:ipm-prelims}. Section~\ref{sec:aff-lls} introduces the affine scaling and layered-least-squares directions, and proves some basic properties.
Section~\ref{sec:lls-overview} provides a detailed overview of
the high level ideas and a roadmap to the analysis. Section~\ref{sec:linsys} further develops the theory of LLS directions and introduces partition lifting scores.
 Section~\ref{sec:layering} gives our scaling invariant layering procedure,
and our overall algorithm can be found
in Section~\ref{sec: the_algorithm}.

 In Section~\ref{sec:analysis}, we give the
potential function proof for the improved iteration bound, relying on technical
lemmas. The full proof of these lemmas is deferred to
Section~\ref{sec:main-lemmas}; however, Section~\ref{sec:analysis}
provides the high-level ideas to each proof. Section~\ref{sec:VY}
shows that our argument also leads to a factor $\Omega(n/\log n)$ improvement
in the iteration complexity bound of the VY-algorithm.

 In Section~\ref{sec:prox-proof},
we prove the technical properties of our LLS step, including its proximity to AS
and step length estimates. Finally, in Section~\ref{sec:initialization}, we discuss the
initialization of the interior-point method.

\medskip

Besides reading the paper linearly, we suggest two other possible ways of navigating the paper.
Readers mainly interested in the circuit imbalance measure and its approximation may focus only on Section~\ref{sec:rescale}; this part can be understood without any familiarity with interior point methods. Other readers, who wish to mainly focus on our interior point algorithm may read Section~\ref{sec:rescale} only up to Section~\ref{sec:circ}; this includes all concepts and statements necessary for the algorithm.

\subsection{Notation} 

Our notation will largely follow
\cite{Monteiro2003,MonteiroT05}. We let $\R_{++}$ denote the set of positive reals, and $\R_+$ the set of nonnegative reals.
 For $n\in \mathbb{N}$, we let
$[n]=\{1,2,\ldots,n\}$.
Let $e^i\in \R^n$ denote the $i$th unit
vector, and $e\in \R^n$ the all 1s vector. For a vector $x\in \R^n$, we let $\diag(x)\in \R^{n\times n}$
denote the diagonal matrix with $x$ on the diagonal. We let $\mathbf D$ denote the set of all positive $n\times n$
diagonal matrices and $\mathbf I_k$ denote the $k \times k$ identity matrix.
For $x,y\in
\R^n$, we use the notation $xy\in \R^n$ to denote
$xy=\diag(x)y=(x_iy_i)_{i\in [n]}$.
The \b{inner} product of the two vectors is denoted as $x^\top y$.
For $p\in \mathbb{Q}$, we also use the notation $x^{p}$ to denote the
vector $(x_i^{p})_{i\in [n]}$. Similarly, for $x,y\in \R^n$, we let $x/y$ denote
the vector $(x_i/y_i)_{i\in [n]}$.
We denote the support of a vector $x \in \R^n$ by $\supp(x) = \{i\in [n]: x_i \neq 0\}$.

For an index subset $I\subseteq [n]$, we use $\pi_I: \R^n \rightarrow \R^I$ for the coordinate
projection. That is, $\pi_I(x)=x_I$, and for a subset $S\subseteq
\R^n$, $\pi_I(S)=\{x_I:\, x\in S\}$.
We let $\R^n_I = \{x \in \R^n : x_{[n]\setminus I} = 0\}$.

For a matrix $B\in \R^{n\times k}$, $I\subset [n]$ and $J\subset [k]$ we let $B_{I,J}$ denote the submatrix of $B$ restricted to the set of rows in $I$ and columns in $J$. We also use $B_{I,\sbullet}=B_{I,[k]}$ and $B_J=B_{\sbullet,J}=B_{[n],J}$.
We let $B^{\dagger}\in \R^{k\times n}$ denote the
 pseudo-inverse of $B$.

We let $\ker(A)$ denote the kernel of the matrix
$A \subseteq \R^{m\times n}$. Throughout, we assume that the matrix $A$ in
\eqref{LP_primal_dual} has full row rank, and that $n\ge 3$.

\b{We use the real model of computation, allowing basic arithmetic operations $+$, $-$, $\times$, $/$,  comparisons, and square root computations. We keep (exact) square root computations for simplicity but we note that these could be avoided.}

\paragraph{Subspace formulation} Throughout the paper, we let $W=\ker(A)\subseteq \R^n$ denote the
kernel of the matrix $A$. Using this notation,
\eqref{LP_primal_dual} can be written in the form
\begin{align}\label{LP-subspace}
\begin{aligned}
\min \; & c^\top x \\
x &\in W + d \\
x &\geq 0,
\end{aligned}
\quad\quad
\begin{aligned}
\max \; &  d^\top(c-s) \\
s  &\in W^\perp+c \\
s &\geq 0,
\end{aligned}
\end{align}
where $d\in \R^n$ satisfies $Ad = b$.
One can e.g., choose $d$ as the minimum norm
solution  $d = \argmin \{\|x\|: Ax=b\} = A^\T (AA^\T)^{-1} b$. 
Note that $s \in W^\perp+c$ is equivalent to $\exists y \in \R^m$ such that
$A^\T y + c = s$. Hence, the 
 original variable $y$ is implicit in this formulation.

\begin{table}
\begin{tabular}{|l|l|l|}
    \hline
    {\bf Symbol} & {\bf Description} & {\bf Defined in} \\
    \hline
    $w = (x,y,s)$ & tuple of feasible solutions & \Cref{sec:ipm-prelims} \\
    \hline
    $\mu(w)$ & normalized duality gap & \Cref{sec:ipm-prelims} \\
    \hline
    $\bar\chi_W$ & subspace condition number & \Cref{sec:bar-chi} \\
    \hline
    $L_I^W$ & lifting map $\pi_I(W) \to W$ & \Cref{def:liftingmap} \\
    \hline
    $\Le^W, \Le^W_{ij}, \Le_{ij}^\delta$ & circuit imbalances & \Cref{def:circuit_imbalance} \\
    \hline
    $\ell^W(\cal J)$ & lifting score & \Cref{def:lifting-score} \\
    \hline
    $\mathcal{N}(\beta)$ & $\ell_2$-neighborhood of the central path & \Cref{sec:ipm-prelims} \\
    \hline
    $\Delta w^\as = (\Delta x^\as, \Delta y^\as, \Delta s^\as)$ & affine scaling direction & \Cref{sec:aff-lls} \\
    \hline
    $\Delta w^\cs = (\Delta x^\cs, \Delta y^\cs, \Delta s^\cs)$ & centrality direction & \Cref{sec:aff-lls} \\
    \hline
    $\Delta w^\lal = (\Delta x^\lal, \Delta y^\lal, \Delta s^\lal)$ & layered least-squares scaling direction & \Cref{sec:lls} \\
    \hline
    $\epsilon^\as_I(w)$ & norm of AS residuals on $I \subseteq [n]$ & \Cref{def:epsilon} \\
    \hline
    $w^+ = w + \alpha \Delta w$ & iterate after predictor step & \Cref{sec:aff-lls} \\
    \hline
    $\delta = \delta(w)$ & approximate rescaled dual $\delta = s^{1/2}x^{-1/2}$  & \Cref{def:delta} \\
    \hline
    $\Rx^\as, \Rs^\as, \Rx^\lal, \Rs^\lal$ & residuals & \Cref{eq:residuals} \\
    \hline
    $G_{w,\sigma} = ([n], E_{w,\sigma})$ & long edge graph & Page \pageref{longedgegraph} \\
    \hline
    $\mathcal J = (J_1,J_2,\dots,J_p)$ & ordered partition & Page \pageref{orderedpartition} \\
    \hline
    $W_{{\cal{J}}, k}$ & subspace & \Cref{sec:linsys} \\
    \hline
    $\gamma$ & parameter & \Cref{def:gamma} \\
    \hline
    $\hat G_{\delta,\sigma} = ([n], \hat E_{\delta,\sigma})$ & auxiliary graph & \Cref{def:aux-graph} \\
    \hline
    $\rho^\mu(i,j), \Psi^\mu(i,j), \Psi(\mu)$ & potential function & Page \pageref{sec:analysis} \\
    \hline
    $\xi^\lal_I(w)$ & norm of LLS residuals on $I \subseteq [n]$ & \Cref{eq:xi} \\
    \hline
    $\llspartition B, \llspartition N$ & Partition of variables based on $\Rx^\lal, \Rs^\lal$ & \Cref{eq:B-N} \\
    \hline
\end{tabular}
\caption{Recurring symbols that will be defined throughout the paper.}
\end{table}

\section{Finding an approximately optimal rescaling}
\label{sec:rescale}

\subsection[The condition number bar-chi]{The condition number {$\bar\chi$}}
\label{sec:bar-chi}
The condition number $\bar\chi_A$ is defined as
\begin{equation}
\begin{aligned}
\bar\chi_A&=\sup\left\{\|A^\top \left(A D A^\top\right)^{-1}AD\|\, : D\in
    {\mathbf D}\right\}\\
&=\sup\left\{\frac{\norm{A^\top y}}{\norm{p}} :
\text{$y$ minimizes $\norm{D^{1/2}(A^\top y - p)}$ for some $0 \neq p \in \R^n$ and $D \in \mathbf D$}\right\}.
\end{aligned}
\end{equation}
This condition number was first studied by Dikin \cite{dikin,dikin1974speed}, Stewart \cite{stewart}, and
Todd \cite{todd-90}, among others, and plays a key role in the analysis of the
Vavasis--Ye interior point method \cite{Vavasis1996}. There is an
extensive literature on the properties and applications of $\bar\chi_A$,
as well as its
relations to other condition numbers. We refer the reader to
the papers \cite{ho2002,Monteiro2003,Vavasis1996} for further results
and references.

It is important to note that $\bar\chi_A$ only depends on the subspace
$W=\ker(A)$. Hence, we can also write $\bar\chi_W$ for a subspace
$W\subseteq \R^n$, defined to be equal to $\bar\chi_A$ for some matrix
$A\in \R^{k\times n}$ with $W=\ker(A)$. We will use the notations $\bar\chi_A$ and
$\bar\chi_W$ interchangeably.

The next lemma summarizes some important known properties of $\bar\chi_A$.
\begin{proposition}\label{prop:chibar}~
Let $A\in \R^{m\times n}$ with full row rank and $W=\ker(A)$.
\begin{enumerate}[label=(\roman*)]
\item \label{i:bitcomp} If the entries of $A$ are all integers, then $\bar\chi_A$ is bounded by $2^{O(L_A)}$,
 where $L_A$ is the input bit length of $A$.
\item \label{i:submatrix_characterisation}  $\bar\chi_A = \max\{ \|B^{-1} A\| : B \text{ non-singular $m \times m$-submatrix of } A\} $.
\item \label{i:dagger} Let the columns of $B \in \R^{n \times (n-m)}$ form an orthonormal basis of $W$.
 Then
 \[
 \bar\chi_W = \max\left\{ \|B B_{I,\sbullet}^\dagger\| : \emptyset \neq I \subset [n]\right\}\, . \]
\item \label{i:chibar-dual} $\bar\chi_W=\bar\chi_{W^\perp}$.
\end{enumerate}
\end{proposition}
\begin{proof}~
Part \ref{i:bitcomp} was proved in \cite[Lemma
24]{Vavasis1996}. For part~\ref{i:submatrix_characterisation}, see
\cite[Theorem 1]{Todd2001}  and \cite[Lemma 3]{Vavasis1996}. In
part~\ref{i:dagger}, the
direction $\ge$ was proved in \cite{stewart}, and the direction $\le$
in \cite{OLeary1990}.
The duality statement \ref{i:chibar-dual} was shown in \cite{gonzaga_lara}.
\end{proof}
In \Cref{prop:self-dual}, we will also give another proof of \ref{i:chibar-dual}.
We now define the lifting map, a key operation in this paper, and
explain its connection to $\bar\chi_A$.
\begin{definition}\label{def:liftingmap}
Let us define the {\em lifting map}
$L_I^W : \pi_{I}(W) \to W$ by
\[
L_I^W(p) = \argmin\left\{\|z\| : z_I = p, z \in W\right\}.
\]
\end{definition}
Note that $L_I^W$ is the unique linear map from $\pi_{I}(W)$ to $W$
such that $\left(L_I^W(p)\right)_I = p$ and $L_I^W(p)$ is orthogonal to
$W \cap \R^n_{[n]\setminus I}$.

\begin{lemma} \label{lem:char-pseudo-inverse}
Let $W \subseteq \R^n$ be an $(n-m)$-dimensional linear subspace. Let
the columns of $B \in \R^{n \times (n-m)}$ denote an orthonormal basis of $W$. Then,
viewing $L_I^W$ as a matrix in $\R^{n\times |I|}$,
\[
L_I^W = B B_{I,\sbullet}^\dagger\, .
\]
\end{lemma}
\begin{proof}
If $p \in \pi_I(W)$, then $p = B_{I,\sbullet} y$ for some
$y \in \R^{n-m}$. By the well-known property of the pseudo-inverse we get $B_{I,\sbullet}^\dagger p = \argmin_{p = B_{I,\sbullet} y}\|y\|$.
This solution satisfies
$\pi_I(BB_{I,\sbullet}^\dagger p) = p$ and
$BB_{I,\sbullet}^\dagger p \in W$. Since the columns of $B$ form an orthonormal basis of $W$, we have $\|BB_{I,\sbullet}^\dagger p\|=\|B_{I,\sbullet}^\dagger p\|$. Consequently, $BB_{I,\sbullet}^\dagger p$ is the minimum-norm point with the above properties.
\end{proof}
The above lemma and \Cref{prop:chibar}\ref{i:dagger} yield the
following characterization. This will be the most suitable characterization of $\bar\chi_W$ for our purposes.
\begin{proposition}\label{prop:subspace-chi}
For a linear subspace $W  \subseteq \R^n$,
\[
\bar{\chi}_W =\max\left\{ \|L_I^W\|\, : {I\subseteq [n]}, I\neq\emptyset\right\}\, .
\]
\end{proposition}

The following notation will be convenient for our algorithm.
For a subspace $W\subseteq \R^n$ and an index set $I\subseteq [n]$, if $\pi_I(W) \neq \set{0}$ then
we define the \emph{lifting score}
\begin{equation}\label{def:lifting-score}
\ell^W(I):=\sqrt{\|L^W_{I}\|^2-1}\, .
\end{equation}
Otherwise, we define $\ell^W(I) = 0$.
This means that for any $z\in \pi_I(W)$ and $x = L_I^W(z)$, $\|x_{[n]\setminus I}\|\le \ell^W(I)\|z\|$.

\paragraph{The condition number $\bar\chi^*_A$}
For every $D\in {\mathbf D}$, we can consider the condition number
$\bar\chi_{DW}=\bar\chi_{AD^{-1}}$. We let
\[
\bar\chi^*_W=\bar\chi^*_A=\inf\{\bar\chi_{DW}\, : D\in {\mathbf D}\}\,
\]
denote the best possible value of $\bar\chi$ that can be attained by
rescaling the coordinates of $W$. The main result of this section is the
following theorem.
\begin{theorem}[name=\prooftext\pageref{proof:bar-chi-star}, {restate=[name=\restatetext]barchistar}]\label{thm:bar-chi-star}
There is an $O(n^2m^2 + n^3)$ time algorithm that for any matrix $A\in\R^{m\times n}$ computes an estimate $\xi$ of $\bar\chi_W$ such that
\[
\xi \leq \bar\chi_W \leq  n(\bar\chi_W^*)^2 \xi
\]
and a $D\in {\mathbf D}$ such that
\[
\bar\chi^*_W\le \bar\chi_{DW}\le n(\bar\chi_W^*)^3\, .
\]
\end{theorem}

\subsection{The circuit imbalance measure}\label{sec:circ}
\b{The key tool in proving \Cref{thm:bar-chi-star} is to study
a more combinatorial condition number, the {\em circuit imbalance measure} which turns out to give a good proxy to
$\bar\chi_A$.}
\begin{definition}
For a linear subspace $W \subseteq \R^n$ and a matrix $A$ such that
$W = \ker(A)$, a circuit is an inclusion-wise
minimal dependent set of columns of $A$.
Equivalently, a circuit is a set $C \subseteq [n]$ such
that $W \cap \R^n_C$ is one-dimensional and that no strict subset of $C$
has this property.
The set of circuits of $W$ is denoted by $\circuits_W$.
\end{definition}
Note that circuits defined above are the same as the circuits in the linear matroid
associated with  $A$.
Every circuit $C\in \circuits_W$ can be associated with a vector $g^C \in W$ such that $\supp(g^C) = C$; this vector is unique up to  scalar multiplication.

\begin{definition}\label{def:circuit_imbalance}
For a circuit $C \in \circuits_W$ and
$i,j \in C$, we let
\begin{equation}\label{eq:Le-C-def}
\Le^W_{ij}(C)=\frac{\left|g^C_j\right|}{\left|g^C_i\right|}\ .
\end{equation}
Note that since $g^C$ is unique up to scalar multiplication, this is independent of the choice of $g^C$.
For any $i,j\in [n]$, we define the \emph{circuit ratio} as the maximum of $\Le^W_{ij}(C)$ over all choices of
the circuit $C$:
\begin{equation}\label{eq:Le-def}
\Le^W_{ij}=\max\left\{\Le^W_{ij}(C):\, C\in\circuits_W, i,j\in C\right\}\ .
\end{equation}
By convention we set $\Le^W_{ij} = 0$ if there is no circuit supporting $i$ and $j$.
Further, we define the \emph{circuit imbalance measure} as
\[
\Le_W=\max\left\{\Le^W_{ij}:\, i, j\in [n]\right\}\, .
\]
Minimizing over all coordinate rescalings, we define
\[
\Le_W^* = \min\left\{ \Le_{DW}:\,  D \in \mathbf D\right\}\, .
\]
We omit the index $W$ whenever it is clear from context.
\b{Further, for a vector $d\in \R^n_{++}$,
we write $\Le_{ij}^d = \Le_{ij}^{\diag(d)W}$ and $\Le^d = \Le^d_W=\Le_{\diag(d)W}$.}
\end{definition}
We want to remark that a priori it is not clear that $\kappa_W^*$ is well-defined. \Cref{thm:low_bound_chi_bar_with_circuits} will show that the minimum of  $\{\kappa_{DW}:\, D\in \mathbf D\}$ is indeed attained.

\medskip

We next formulate the main statements on the circuit imbalance measure; proofs will be given in the subsequent subsections.
Crucially, we show that the circuit imbalance $\Le_W$ is a good proxy to  the condition number $\bar\chi_W$. The lower bound was already proven in \cite{Vavasis1994}, and the upper bound is from \cite{DadushNV20}. A slightly weaker upper bound
 $\sqrt{1 + (n\Le_W)^2}$ was previously given in the conference version of this paper \cite{DHNV20}.
\begin{theorem}[name=\prooftext\pageref{proof:thm:chi-approx}, {restate=[name=\restatetext]chiapprox},label=thm:chi-approx]
For a linear subspace $W\subseteq \R^n$,
\[
\sqrt{1 + (\Le_W)^2} \le \bar\chi_W\le n\Le_W.
\]
\end{theorem}

We now overview some basic properties of $\Le_W$.
\Cref{prop:subspace-chi} asserts that $\bar\chi_W$ is the maximum $\ell_2\to\ell_2$ operator norm of the mappings $L_I^W$ over $I\subseteq [n]$. In \cite{DadushNV20}, it was shown that $\Le_W$ is in contrast the maximum $\ell_1\to\ell_\infty$ operator norm of the same mappings; this easily implies the upper bound $\bar\chi_W\le n\Le_W$.

\begin{proposition}[{\cite{DadushNV20}}]\label{lem:kappa-lift}
For a linear subspace $W  \subseteq \R^n$,
\[
\kappa_W =\max\left\{ \frac{\|L_I^W(p)\|_\infty}{\|p\|_1}\, : {I\subseteq [n]}, I\neq\emptyset, p\in \pi_I(W)\setminus\{0\}\right\}\, .
\]
\end{proposition}

Similarly to $\bar\chi_W$, $\kappa_W$ is self-dual;  this holds for all individual $\Le_{ij}^W$ values as well.
\begin{lemma}[name=\prooftext\pageref{proof:lem:self-dual}, {restate=[name=\restatetext]selfdual}]\label{lem:self-dual}
For any subspace $W \subseteq \R^n$ and $i,j \in [n]$, $\Le_{ij}^W = \Le_{ji}^{W^\perp}$.
\end{lemma}

The next lemma provides a subroutine that efficienctly yields upper bounds on $\ell^W(I)$ or lower bounds on some circuit imbalance values. Recall the definition of the lifting score $\ell^W(I)$ from \eqref{def:lifting-score}.

\begin{lemma}[name=\prooftext\pageref{proof:lem:purify}, {restate=[name=\restatetext]purify}]\label{lem:purify}
There exists a subroutine
\textsc{Verify-Lift}($W,I,\theta$) that, given a linear subspace $W\subseteq \R^n$, an index set $I\subseteq [n]$, and a threshold $\theta\in \R_{++}$, either returns the answer \emph{`pass'}, verifying
 $\ell^W(I)\le \theta$, or returns the answer \emph{`fail'}, and a pair $i \in I, j \in [n] \setminus I$ such that
 $\theta/n\le \Le^W_{ij}$. The running time can be bounded as $O(n(n-m)^2)$.
\end{lemma}

The proofs of the above statements are given in Section~\ref{sec:basic-kappa-proof}.

\paragraph{A min-max theorem}
We next provide a combinatorial min-max characterization of
$\Le^*_W$. Consider the \emph{circuit ratio digraph} $G=([n],E)$ on the node set
$[n]$ where $(i,j)\in E$ if $\kappa_{ij}>0$, that is, there
exists a circuit $C\in {\cal C}$ with $i,j\in C$.
We will refer to $\Le_{ij}=\Le_{ij}^W$ as the weight of the edge $(i,j)$.
(Note that $(i,j)\in E$ if and only if $(j,i)\in E$, but the weight of these two edges can be different.)

Let $H$ be a \emph{cycle} in $G$, that is, a sequence of indices
$i_1,i_2,\dots,i_k, i_{k+1} = i_1$. We use $|H|=k$ to denote the
length of the cycle.
(In our terminology, `cycles'  always refer to objects in $G$,
whereas `circuits' refer to the minimum supports in $\ker(A)$.)

We use the notation $\Le(H)=\Le_W(H)=\prod_{j=1}^k \Le^W_{i_j
  i_{j+1}}$.
For a vector $d\in \R^n_{++}$, we denote $\Le^d_W(H)=\Le_{\diag(d)W}(H)$.
A simple but important observation is that such a
rescaling does not change the value associated with the cycle, that is,
\begin{equation}\label{eq:cycle-invariant}
\Le^d_W(H)=\Le_W(H)\quad \forall d\in \R^n_{++}\quad \mbox{for any cycle
  $H$ in $G$}\, .
\end{equation}

\begin{theorem}[name=\prooftext\pageref{proof:thm:low_bound_chi_bar_with_circuits}, {restate=[name=\restatetext]minmaxkappa}]\label{thm:low_bound_chi_bar_with_circuits}
For a subspace $W\subset \R^n$, we have
\[
\Le_W^* = \min_{d > 0} \Le_W^d = \max\left\{\Le_W(H)^{1/|H|}:\ \mbox{$H$ is
  a cycle in $G$}\right\}\, .
\]
\end{theorem}
The proof relies on the following formulation:
\begin{align*}
\begin{aligned}
    \Le^*_W=&&\min \; & t \\
            &&\Le_{ij}d_j/d_i & \leq t \quad \forall (i,j) \in E \\
            &&d &> 0.
\end{aligned}
\end{align*}
Taking logarithms, we can rewrite this problem as
\begin{equation*}
\begin{aligned}
\min \; & s \\
\log \Le_{ij} + z_j - z_i & \leq s \quad \forall (i,j) \in E \\
z &\in \R^n.
\end{aligned}
\end{equation*}
This is the dual of the minimum-mean cycle problem with weights $\log\Le_{ij}$, and can be solved in polynomial time (see e.g. \cite[Theorem 5.8]{amo}).

Whereas this formulation verifies \Cref{thm:low_bound_chi_bar_with_circuits}, it does not give a polynomial-time algorithm to compute $\Le^*_W$. The  caveat is that the  values $\Le^W_{ij}$ are typically not available; in fact, approximating them up to a factor $2^{O(m)}$ is NP-hard,
as follows from the work of Tun{\c{c}}el~\cite{Tuncel1999}.

Nevertheless,  the following corollary of \Cref{thm:low_bound_chi_bar_with_circuits}
shows that any arbitrary
circuit containing $i$ and $j$ yields a $(\Le^*)^2$ approximation
to $\kappa_{ij}$.

\begin{corollary}[name=\prooftext\pageref{proof:cor:2-circuit}, {restate=[name=\restatetext]twocircuit}]\label{cor:2-circuit}
Let us be given a linear subspace $W\subseteq \R^n$ and $i,j\in [n]$, $i\neq j$,
and a circuit $C\in \circuits_W$ with $i,j\in C$. Let $g\in W$ be the
corresponding vector with $\supp(g)=C$. Then,
\[
\frac{\kappa^W_{ij}}{\left(\Le_W^*\right)^2}\le
\frac{|g_j|}{|g_i|}\le \kappa^W_{ij}.
\]
\end{corollary}

The above statements are shown in Section~\ref{sec:min-max}. In Section~\ref{sec:matroid}, we use techniques from matroid theory and linear algebra to efficiently identify a circuit for any pair of variables that are contained in the same circuit.
A matroid is non-separable if the circuit hypergraph is connected; precise definitions and background will be described in Section~\ref{sec:matroid}.

\begin{theorem}[name=\prooftext\pageref{proof:thm:pairwise_circuits_in_matroids}, {restate=[name=\restatetext]pairwisecircuits}]\label{thm:pairwise_circuits_in_matroids}
Given $A\in \R^{m\times n}$, there exists an $O(n^2 m^2)$ time algorithm
\textsc{Find-Circuits}($A$)
that obtains a decomposition of ${\cal M}(A)$ to a direct sum of
non-separable linear matroids, and returns a family $\hat \circuits$ of circuits such that if
$i$ and $j$ are in the same
non-separable component, then there exists a circuit in $\hat \circuits$
containing both $i$ and $j$. Further, for each $i\neq j$ in the same component,
the algorithm returns a value $\hat \Le_{ij}$ as the  the
maximum of $|g_j/g_i|$ such that $g\in W$, $\supp(g)=C$ for some $C\in
\hat \circuits$ containing $i$ and $j$. For these values, $\hat\Le_{ij} \le \Le_{ij} \le (\Le^*)^2\hat\Le_{ij}$.
\end{theorem}
Finally, in Section~\ref{sec:approx-alg}, we combine the above results to prove Theorem~\ref{thm:bar-chi-star} on approximating $\bar\chi^*_W$ and $\kappa^*_W$.

Section~\ref{sec:matroid} contains an interesting additional statement, namely that  the
logarithms of the circuit ratios satisfy the triangle inequality. This will also be useful in the analysis of the LLS algorithm. The proof uses similar arguments as the proof of \Cref{thm:pairwise_circuits_in_matroids}. A simpler proof of this statement was subsequently given in \cite{ENV21}.
\begin{lemma}[name=\prooftext\pageref{proof:lem:imbalance_triangle_inequality}, {restate=[name=\restatetext]triangle}]\label{lem:imbalance_triangle_inequality}~
\begin{enumerate}[label=(\roman*)]
\item For any distinct $i,j,k$ in the same connected component of $\circuits_W$,
and any $g^C$ with $i,j \in C$, $C \in \circuits_W$, there exist circuits
$C_1, C_2 \in \circuits_W$, $i,k \in C_1$, $j,k \in C_2$ such that
$|g^C_j/g^C_i| = |g^{C_2}_j/g^{C_2}_k| \cdot |g^{C_1}_k/g^{C_1}_i|$.
\item \label{i:triangle} For any distinct $i,j,k$ in the same connected component of $\circuits_W$,
$\Le_{ij} \leq \Le_{ik}\cdot \Le_{kj}$.
\end{enumerate}
\end{lemma}

\subsection[Basic properties of kappa]{Basic properties of $\kappa_W$}\label{sec:basic-kappa-proof}
\chiapprox*
\begin{proof}
\label{proof:thm:chi-approx}
For the first inequality,  let $C \in \circuits_W$ be the circuit and $i\neq j
\in C$ such that $|g_j/g_i| = \Le_W$ for the corresponding solution
$g=g^C$.
Let us use the characterization of $\bar\chi_W$ in
\Cref{prop:subspace-chi}. Let
$I=([n]\setminus C)\cup
\{i\}$, and $p=g_i e^i$, that is, the vector with $p_i=g_i$ and
$p_k=0$ for $k\neq i$.
Then, the unique vector  $z\in W$ such that $z_I=p$ is
$z=g$. Therefore,
\[
\bar\chi_W\ge \min_{z\in W, z_I=p}\frac{\|z\|}{\|p\|}=\frac{\|g\|}{|g_i|}\ge
  \frac{\sqrt{|g_i|^2 + |g_j|^2}}{|g_i|}=\sqrt{1+\Le_W^2}\ .
\]
The second inequality is immediate from \Cref{prop:subspace-chi} and \Cref{lem:kappa-lift}, and the inequalities between $\ell_1$, $\ell_2$, and $\ell_\infty$ norms. The proof of the slightly weaker $\bar\chi_W\le \sqrt{1+(n\kappa_W)^2}$ follows
from \Cref{lem:purify}.
\end{proof}

The next lemma will be needed to prove \Cref{lem:purify} and also to analyze the LLS algorithm.
Let us say that the vector $y \in \R^n$ \emph{conforms to} 
$x\in\R^n$ if $x_iy_i >0$ whenever $y_i\neq 0$.

\begin{lemma} \label{lem:identify}
For $i \in I \subset [n]$ with $e^i_I \in \pi_I(W)$, let $z = L_I^W(e^i_I)$. Then for
 any $j \in \supp(z)$ we have $\Le_{ij}^W \geq |z_j|$.
\end{lemma}
\begin{proof}
 We consider the cone $F \subset W$ of vectors that conform to $z$.
The faces of $F$ are bounded by inequalities of the form $z_k y_k \geq 0$
or $y_k = 0$.
The edges (rays) of $F$ are of the form $\{\alpha g:\,  \alpha\ge 0\}$  with $\supp(g) \in \circuits_W$.
It is easy to see from the Minkowski--Weyl theorem that $z$ can be
written as
\[
z=\sum_{k=1}^h  g^k,
\]
where
$h\le n$, $C_1,C_2,\ldots,C_h\in
\circuits_W$ are circuits, and the vectors $g^1,g^2,\ldots,g^h\in W$ conform to
$z$ and $\supp(g^k)=C_k$ for all $k\in [h]$.
 Note that $i \in C_k$ for all $k\in [h]$, as otherwise, $z'=z-g^k$ would also
satisfy $z'_I=e^i_I$, but $\|z'\|<\|z\|$ due to $g^k$ being conformal to $z$,
a contradiction to the definition of $z$.

At least one $k \in [h]$ contributes at least as much to
$|z_j| = \frac{\sum_{k=1}^h |g^k_j|}{\sum_{k=1}^h g^k_i}$
as the average. Hence we find $\Le_{ij}^W \geq |g^k_j/g^k_i| \geq |z_j|$.
\end{proof}

\purify*
\begin{proof}
\label{proof:lem:purify}
Take any minimal $I'\subset I$ such that $\dim(\pi_{I'}(W)) = \dim(\pi_I(W))$.
Then we know that $\pi_{I'}(W) = \R^{I'}$ and for $p \in \pi_I(W)$
we can compute $L_I^W(p) = L_{I'}^W(p_{I'})$.
Let $B \in \R^{([n] \setminus I) \times I'}$ be the matrix sending any $q \in \pi_{I'}(W)$
to the corresponding vector $(L_{I'}^W(q))_{[n]\setminus I}$. The column $B_i$
can be computed as $(L_{I'}^W(e^i_{I'}))_{[n]\setminus I}$ for $e^i_{I'} \in \R^{I'}$.
We have $\|L_I^W(p)\|^2 = \|p\|^2 + \|(L_{I'}^W(p_{I'}))_{[n]\setminus I}\|^2 \leq \|p\|^2 + \|B\|^2\|p_{I'}\|^2$
for any $p \in \pi_I(W)$, and so $\ell^W(I)=\sqrt{\|L_I^W\|^2-1} \leq \|B\|$.
We upper bound the operator norm by the Frobenius norm as
$\|B\| \leq \|B\|_F = \sqrt{\sum_{ji} B_{ji}^2} \leq n\max_{ji} |B_{ji}|$.
By Lemma~\ref{lem:identify} it follows that $|B_{ji}| = |(L_{I'}^W(e^i))_j| \leq \Le_{ij}^W$. The algorithm returns the answer `pass' if $n\max_{ji} |B_{ji}|\le \theta$ and `fail' otherwise.

To implement the algorithm, we first need to select a minimal
$I'\subset I$ such that $\dim(\pi_{I'}(W)) = \dim(\pi_I(W))$. This can
be found by computing a matrix $M\in \R^{n \times (n-m)}$ such that
${\rm range} (M)=W$, and selecting a maximal number of linearly independent
columns of $M_{I,\sbullet}$. Then, we compute the matrix $B \in \R^{([n] \setminus I) \times
  I'}$ that implements the transformation $[L_{I'}^W]_{[n]\setminus
  I}:\ \pi_{I'}(W)\to \pi_{[n]\setminus I}(W)$. The algorithm returns
the pair $(i,j)$ corresponding to the entry maximizing $|B_{ji}|$.
The running time analysis will be given in the proof of \Cref{lem:layering_works_correctly}, together with an amortized analysis of a sequence of calls to the subroutine.
\end{proof}

\begin{remark}\em \label{rem:B-i-j}
We note that the algorithm \textsc{Verify-Lift} does not need to compute the circuit as in
\Cref{lem:identify}. The following observation will be important
in the analysis: the algorithm returns the answer `fail' even if
$\ell^W(I)\le \theta< n|B_{ji}|$.
\end{remark}

We now prove the duality property of the circuit imbalances.

\selfdual*

\begin{proof}
  \label{proof:lem:self-dual}

Choose a circuit $C \in \circuits_W$ and corresponding circuit solution $g :=
g^C \in W\cap \R^n_C$ such that $\Le_{ij} = \Le_{ij}(C) = |g_j/g_i|$. We will
construct a circuit solution in $W^\perp$ that certifies $\Le_{ji}^{W^\perp} \geq \Le_{ij}^W$.

Define $h \in \R^C$ by $h_i = g_j, h_j = -g_i$ and $h_k = 0$ for all $k\in
C\setminus\{i,j\}$. Then, $h$ is orthogonal to $g_C$ by construction, and
hence $h \in(\pi_C(W \cap \R^n_C))^\perp = \pi_C(W^\perp)$.
Furthermore, we have $\supp(h) \in \circuits_{\pi_C(W^\perp)}$ since $h
\in \R^C$ is a support minimal vector orthogonal to $g^C$.

Take any vector $\bar{h} \in W^\perp$ satisfying $\bar{h}_C = h$ that is
support minimal subject to these constraints. We claim that $\supp(\bar{h}) \in
\circuits_{W^\perp}$. Assume not, then there exists a non-zero $v \in
W^\perp$ with $\supp(v) \subset \supp(\bar{h})$.
Since $\supp(\pi_C(v)) \subseteq
\supp(\pi_C(\bar{h})) = \supp(h)$, we must have either $v_C=0$ or $v_C = s h$ for $s\neq 0$.
If $v_C=0$, then $\bar h-\alpha v$ is also in $ W^\perp$ satisfying
$\pi_C(\bar{h}_C - \alpha v) = h$
 for all $\alpha\in \R$, and since $v\neq 0$ we can choose $\alpha$ such that
$\bar h-\alpha v$ has smaller support than $\bar h$, a contradiction. If $s\neq 0$
then $v/s \in W^\perp$ satisfies $\pi_C(v/s) = h$ and has smaller support than $\bar h$,
again a contradiction.

By the above construction, we have
\[\Le_{ji}^{W^\perp}\ge \left|\frac{\bar{h}_i}{\bar{h}_j}\right|=\left|\frac{h_i}{h_j}\right|=\left|\frac{g_j}{g_i}\right|=\Le_{ij}^W\, .
\]
By swapping the role of $W$ and $W^\perp$ and $i$ and $j$, we obtain $\Le_{ij}^W\ge \Le_{ji}^{W^\perp}$. The statement follows.
\end{proof}

\subsection{A min-max theorem on \texorpdfstring{$\Le^*_W$}{kappa-star}}
\label{sec:min-max}

The proof of the characterization of $\Le_W^*$ follows.
\minmaxkappa*
\begin{proof}
\label{proof:thm:low_bound_chi_bar_with_circuits}
For the direction $\Le_W(H)^{1/|H|}\le \Le_W^*$ we use
\eqref{eq:cycle-invariant}. Let $d > 0$ be a scaling and $H$ a cycle. We have
$\Le^d_{ij}\le \Le_W^d$ for every $i,j\in [n]$,
and hence $\Le_W(H)=\Le_W^d(H)\le (\Le_W^d)^{|H|}$.
Since this inequality holds for every $d > 0$,
it follows that $\Le_W(H) \le (\Le_W^*)^{|H|}$.

For the reverse direction, consider the following optimization problem.
\begin{align}
\begin{aligned} \label{sys: Le_star}
\min \; & t \\
\Le_{ij}d_j/d_i & \leq t \quad \forall (i,j) \in E \\
d &> 0.
\end{aligned}
\end{align}
For any feasible solution $(d,t)$ and $\lambda>0$, we get another feasible solution $(\lambda d, t)$
with the same objective value. As such, we can strengthen the condition $d > 0$
to $d \geq 1$ without changing the objective value. This makes it clear that the optimum value
is achieved by a feasible solution.

Any rescaling $d > 0$ provides a feasible solution with objective value $\Le^d$,
which means that the optimal value $t^*$ of \eqref{sys: Le_star} is $t^* = \Le^*$.
Moreover, with the variable substitution $z_i=\log d_i$, $s=\log t$, \eqref{sys: Le_star} can be written as a linear program:
\begin{align}
\begin{aligned} \label{sys: min-mean-cycle}
\min \; & s \\
\log \Le_{ij} + z_j - z_i & \leq s \quad \forall (i,j) \in E \\
z &\in \R^n.
\end{aligned}
\end{align}
This is the dual of a minimum-mean cycle problem with respect to the
cost function $\log(\Le_{ij})$. Therefore, an optimal solution
corresponds to the cycle maximizing $\sum_{ij\in H}\log\Le_{ij}/|H|$, or
in other words, maximizing $\Le(H)^{1/|H|}$.
\end{proof}

The following example shows that $\Le^* \leq \bar\chi^*$ can be arbitrarily big.
\begin{example}
Take $W = {\rm span}((0,1,1,M)^\T,(1,0,M,1)^\T)$, where $M > 0$.
Then $\{2,3,4\}$ and $\{1,3,4\}$ are circuits with
$\Le^W_{34}(\{2,3,4\}) = M$ and $\Le^W_{43}(\{1,3,4\}) = M$.
Hence, by \Cref{thm:low_bound_chi_bar_with_circuits},
we see that $\Le^* \geq M$.
\end{example}

\twocircuit*
\begin{proof}
\label{proof:cor:2-circuit}
The second inequality follows by definition. For the first inequality,
note that the same circuit $C$ yields $|g_i/g_j|\le \Le^W_{ji}(C)\le \Le^W_{ji}$.
Therefore, $|g_j/g_i|\ge 1/\Le^W_{ji}$.

From \Cref{thm:low_bound_chi_bar_with_circuits} we see that
$\kappa^W_{ij}\kappa^W_{ji}\le(\kappa^*_W)^2$, giving
$1/\Le^W_{ji}\ge \kappa^W_{ij}/ (\kappa^*_W)^2$,
completing the proof.
\end{proof}
\subsection{Finding circuits: a detour in matroid theory}\label{sec:matroid}
We next prove \Cref{thm:pairwise_circuits_in_matroids}, showing how to efficiently obtain a family $\hat \circuits\subseteq
\circuits_W$ such that for any $i,j\in [n]$, $\hat \circuits$ includes a circuit
containing both $i$ and $j$, provided there exists such a circuit.

We need some simple concepts and results from matroid theory. We refer
the reader to \cite[Chapter 39]{schrijver} or \cite[Chapter
5]{frankbook} for definitions and background.
Let ${\cal M}=([n],{\cal I})$ be a matroid on ground set $[n]$ with
independent sets ${\cal I}\subseteq 2^{[n]}$. The rank $\rk(S)$ of
a set $S\subseteq [n]$ is the maximum size of an independent set
contained in $S$. The maximal independent sets are called
\emph{bases}. All bases have the same cardinality $\rk([n])$.

For the matrix $A\in\R^{m\times n}$, we will work with the linear
matroid ${\cal M}(A)=([n],{\cal I}(A))$, where a subset $I\subseteq [n]$ is independent
if the columns $\{A_i\, : i\in I\}$ are linearly independent. Note
that $\rk([n])= m$ under the assumption that $A$ has full row rank.

The {\em
  circuits} of the matroid are the inclusion-wise minimal
non-independent sets. Let $I\in {\cal I}$ be an independent set, and
$i\in [n]\setminus I$ such that $I\cup \{i\}\notin {\cal I}$. Then,
there exists a unique circuit $C(I,i)\subseteq I\cup \{i\}$ that is
called the \emph{fundamental circuit} of $i$ with respect to $I$. Note
that $i\in C(I,i)$.

The matroid $\cal M$ is \emph{separable}, if the ground set $[n]$ can be
partitioned to two nonempty subsets $[n]=S\cup T$ such that $I\in
{\cal I}$ if and only if $I\cap S,I\cap T\in {\cal I}$. In this case, the matroid is the direct sum of its
restrictions to $S$ and $T$. In particular, every circuit is fully
contained in $S$ or in $T$.

For the linear matroid ${\cal M}(A)$, separability
means that $\ker(A)=\ker(A_S) \times \ker(A_T)$. In this case, solving \eqref{LP_primal_dual} can be decomposed into two
subproblems, restricted to the columns in $A_S$ and in $A_T$,  and
$\kappa_A=\max\{\kappa_{A_S},\kappa_{A_T}\}$.

Hence, we can
focus on \emph{non-separable} matroids. The following characterization
is well-known, see e.g. \cite[Theorems 5.2.5,
5.2.7--5.2.9]{frankbook}.
For a hypergraph $H=([n],{\cal E})$, we define the underlying graph
$H_G=([n],E)$ such that $(i,j)\in E$ if there is a hyperedge $S\in
{\cal E}$ with $i,j\in S$. That is, we add a clique corresponding to
each hyperedge. The hypergraph is called {\em connected} if the underlying
graph $G=([n],E)$ is connected.
\begin{proposition}\label{prop:conn}
For a matroid ${\cal M}=([n],{\cal I})$, the following are equivalent:
\begin{enumerate}[label=(\roman*)]
\item ${\cal M}$ is non-separable.\label{item:non-sep}
\item The hypergraph of the circuits is connected. \label{item:conn}
\item For any base $B$ of ${\cal M}$, the hypergraph formed by  the
  fundamental circuits $\circuits^B=\{ C(B,i)\, : i\in [n]\setminus B\}$ is connected.\label{item:fundamental}
\item For any $i,j\in [n]$, there exists a circuit containing $i$ and $j$.\label{item:circ}
\end{enumerate}
\end{proposition}
\begin{proof}
The implications \ref{item:non-sep} $\Leftrightarrow$ \ref{item:conn},
\ref{item:fundamental} $\Rightarrow$ \ref{item:conn}, and
\ref{item:circ} $\Rightarrow$ \ref{item:conn} are immediate from the
definitions.

For the implication \ref{item:conn} $\Rightarrow$
\ref{item:fundamental}, assume for a contradiction that the hypergraph
of the fundamental circuits with respect to $B$ is not connected. This
means that we can partition $[n]=S\cup T$ such that for each $i\in S$,
$C(B,i)\subseteq S$, and for each $i\in T$, $C(B,i)\subseteq
T$. Consequently, $\rk(S)=|B\cap S|$, $\rk(T)=|B\cap T|$, and
therefore $\rk([n])=\rk(S)+\rk(T)$. It is easy to see that this
property is equivalent to separability to $S$ and $T$; see
e.g. \cite[Theorem 5.2.7]{frankbook} for a proof.

Finally, for the implication \ref{item:conn} $\Rightarrow$
\ref{item:circ}, consider the undirected graph $([n],E)$ where
$(i,j)\in E$ if there is a circuit containing both $i$ and $j$.
This graph is transitive according to \cite[Theorem 5.2.5]{frankbook}:
if $(i,j), (j,k)\in E$, then also $(i,k)\in E$. Consequently, whenever
$([n],E)$ is connected, it must be a complete directed graph.
\end{proof}

We give a different proof of \ref{item:fundamental} $\Rightarrow$
\ref{item:circ} in Lemma \ref{lem:all-circuits} that will be convenient for our algorithmic
purposes. First, we need a simple lemma that is commonly used in
matroid optimization, see e.g. \cite[Lemma 13.1.11]{frankbook} or
\cite[Theorem 39.13]{schrijver}.
\begin{lemma}\label{lem:matroid-exchange}
Let $I$ be  an independent set of a matroid ${\cal M}=([n],{\cal
  I})$, and $U=\{u_1,u_2,\ldots, u_\ell\}\subseteq I$,
$V=\{v_1,v_2,\ldots, v_\ell\}\subseteq [n]\setminus I$ such that $I\cup \{v_i\}$ is dependent for each $i\in [\ell]$.
Further, assume that for each $t\in [\ell]$, $u_t\in C(I,v_t)$ and
$u_t \notin C(I,v_h)$ for all $h<t$. Then, $(I\setminus
U)\cup V \in {\cal I}$.
\end{lemma}

We give a sketch of the proof. First, we note that for each $t\in [\ell]$, $u_t\in C(I,v_t)$ means that
 exchanging $v_t$ for $u_t$ maintains
independence. The statement follows by induction on $\ell$: we consider the independent set $I'=(I\setminus \{u_\ell\})\cup \{v_\ell\}$. We can apply induction for $I'$, $U'=\{u_1,u_2,\ldots, u_{\ell-1}\}$, and $V'=\{v_1,v_2,\ldots, v_{\ell-1}\}$, noting that the assumption guarantees that $C(I',v_t)=C(I,v_t)$ for all $t\in [\ell-1]$.
 Based on this lemma, we
show the following exchange property.
\begin{lemma}\label{lem:all-circuits}
Let $B$ be a basis of the matroid ${\cal M}=([n],{\cal I})$,
and let $U=\{u_1,u_2,\ldots, u_\ell\}\subseteq B$, and $V=\{v_1,v_2,\ldots, v_\ell,v_{\ell+1}\}\subseteq
[n]\setminus B$.
Assume $C(B,v_1)\cap U=\{u_1\}$, $ C(B,v_{\ell+1})\cap U=\{u_\ell\}$, and for each $2\le
t\le \ell$, $ C(B,v_t)\cap U=\{u_{t-1}, u_t\}$.
Then $(B\setminus U)\cup V$ contains a unique circuit $C$, and
$V\subseteq C$.
\end{lemma}
The situation described here corresponds to a minimal path in the
hypergraph $\circuits^B$ of the fundamental circuits with respect to a basis
$B$. The hyperedges $C(B,v_i)$ form a path from $v_1$ to $v_{\ell+1}$
such that no shortcut is possible (note that this is weaker than
requiring a shortest path).
\begin{proof}[Proof of Lemma~\ref{lem:all-circuits}]
 Note that $S = (B \setminus U)\cup V \notin {\cal I}$ since
$|S|>|B|$ and $B$ is a basis.
For any $i\in [\ell+1]$, we can use
Lemma~\ref{lem:matroid-exchange} to show that $S\setminus
\{v_{i}\} = (B \setminus U) \cup (V \setminus \{v_i\}) \in {\cal I}$
(and thus, is a basis).
To see this, we apply Lemma~\ref{lem:matroid-exchange} for the ordered
sets $V'=\{v_1,\ldots,v_{i-1},v_{\ell+1},v_\ell,\ldots,v_{i+1}\}$ and
$U'=\{u_1,\ldots,u_{i-1},u_\ell,u_{\ell-1},\ldots,u_i\}$.

Consequently, every circuit in $S$ must contain the entire set $V$.
The uniqueness of the circuit in $S$ follows by the well-known circuit
axiom asserting that if $C,C'\in\circuits$, $C \neq C'$ and $v\in C\cap C'$, then
there exists a circuit $C''\in\circuits$ such that $C''\subseteq
(C\cup C')\setminus \{v\}$, contradicting the claim that
every circuit in $S$ contains the entire set $V$.
\end{proof}

We are ready to describe the algorithm that will be used to obtain
lower bounds on all $\Le_{ij}$ values.
\pairwisecircuits*
\begin{proof}
\label{proof:thm:pairwise_circuits_in_matroids}
Once we have found the set of circuits $\hat \circuits$, and computed
$\hat \Le_{ij}$ as in the statement, the inequalities $\hat\Le_{ij} \le \Le_{ij} \le
(\Le^*)^2\hat\Le_{ij}$ follow easily. The first inequality is by the
definition of $\Le_{ij}$, and the second inequality is from \Cref{cor:2-circuit}.

We now turn to the computation of $\hat \circuits$. We first obtain a basis $B\subseteq [n]$ of $\ker(A)$ via Gauss-Jordan elimination
in time $O(nm^2)$. Recall the assumption that $A$ has full row-rank.
Let us assume that $B=[m]$ is the set of first $m$
indices. The elimination transforms it to the form
$A=(\mathbf I_m|H)$, where $H\in \R^{m \times (n-m)}$ corresponds to the non-basis elements.
In this form,  the fundamental circuit
$C(B,i)$ is the support of the $i$th column of $A$ together with $i$ for every $m+1\le i\le n$.
We let $\circuits^B$ denote the set of all these
fundamental circuits.

We construct an undirected graph
$G=(B,E)$ as follows. For each $i\in [n]\setminus B$, we add a
clique
between the nodes in $C(B,i)\setminus \{i\}$. This graph can be
constructed in
$O(nm^2)$ time.

The connected components of $G$ correspond to the connected
components of $\circuits^B$ restricted to $B$. Thus, due to the equivalence shown in \Cref{prop:conn} we can obtain the decomposition by
identifying the connected components of $G$. For the rest of the proof, we
assume that the entire hypergraph is connected; connectivity can be
checked in $O(m^2)$ time.

 We initialize
$\hat\circuits$ as $\circuits^B$. We will then check all pairs $i,j\in
[n]$, $i\neq j$. If no circuit $C\in \hat\circuits$ exists with
$i,j\in C$, then we will add such a circuit to $\hat\circuits$ as follows.

Assume first $i,j\in
[n]\setminus B$. We can find a shortest path in $G$ between the sets
$C(B,i)\setminus \{i\}$ and $C(B,j)\setminus \{j\}$ in time $O(m^2)$. This can be
represented by the sequences of points
$V=\{v_1,v_2,\ldots,v_{\ell+1}\}\subseteq [n]\setminus B$, $v_1=i$,
$v_{\ell+1}=j$, and $U=\{u_1,u_2,\ldots,u_\ell\}\subseteq B$ as in
Lemma~\ref{lem:all-circuits}. According to the lemma,
$S=(B\setminus U)\cup V$
contains a unique circuit $C$ that contains all $v_t$'s, including $i$
and $j$.

We now show how this circuit can be identified in $O(m)$ time, along with the vector
$g^C$.
 Let $A_S$ be the submatrix corresponding to the columns in
$S$.
Since $g=g^C$ is unique up to scaling, we can set
$g_{v_1}=1$. Note that for each $t\in [\ell]$, the row of $A_S$ corresponding to $u_t$ contains
only two nonzero entries: $A_{u_tv_t}$ and $A_{u_tv_{t+1}}$. Thus,
the value $g_{v_1}=1$ can be propagated to assigning unique values to
$g_{v_2},g_{v_3},\ldots,g_{v_{\ell+1}}$. Once these values are set,
there is a unique extension of $g$ to the indices $t\in B\cap S$
in the basis. Thus, we have identified $g$ as the unique element of
$\ker(A_S)$ up to scaling. The circuit $C$ is obtained as $\supp(g)$.
Clearly, the above procedure can be implemented in $O(m)$ time.

The argument easily extends to finding circuits for the case
$\{i,j\}\cap B\neq \emptyset$. If $i\in B$, then for any choice of
$V=\{v_1,v_2,\ldots,v_{\ell+1}\}$ and $U=\{u_1,u_2,\ldots,u_\ell\}$ as in
Lemma~\ref{lem:all-circuits} such that $i\in C(B,v_1)$ and $i\notin
C(B,v_t)$ for $t>1$, the unique circuit in $(B\setminus U)\cup V$
also contains $i$. This follows from \Cref{lem:matroid-exchange} by
taking $V' = \set{v_{\ell+1},v_\ell,\dots,v_1}$ and
$U' = \set{u_\ell,\dots,u_1, i}$, which proves that
$S \setminus \set{i} = (B\setminus U') \cup V' \in \mathcal I$.
Similarly, if $j \in B$ with $j \in C(B,v_{\ell + 1})$ and
$j\notin C(B,v_t)$ for $t < \ell + 1$, taking $V'' = V$ and
$U'' = \set{u_1,\dots,u_\ell, j}$ gives $S \setminus \set{j} \in \mathcal I$.

The bottleneck for the running time is finding the shortest paths for
the $n(n-1)$ pairs, in time $O(m^2)$ each.
\end{proof}

\paragraph{The triangle inequality}
An interesting additional fact about the circuit ratio graph is that the
logarithm of the weights satisfy the triangle inequality. The proof uses similar arguments as the proof of \Cref{thm:pairwise_circuits_in_matroids} above.
\triangle*
\begin{proof}
  \label{proof:lem:imbalance_triangle_inequality}
Note that part {\em (ii)} immediately follows from part {\em (i)} when taking
 $C \in \circuits_W$ such that $\Le_{ij}(C) = \Le_{ij}$. We now prove part {\em (i)}.

Let $A \in \R^{m \times n}$ be a full-rank matrix with $W = \ker(A)$.
If $C = \set{i,j}$, then the columns $A_i, A_j$ are linearly dependent. Writing $A_i = \lambda A_j$, we have $\lambda = -g^C_j/g^C_i$.
Let $h$ be any circuit solution with $i,k \in \supp(h)$, and hence $j \notin \supp(h)$. By assumption,
the vector $h' = h - h_i e_i + \lambda h_i e_j$ will satisfy $Ah' = 0$ and have $i \notin\supp(h'), j,k\in\supp(h')$.
We know that $h'$ is a circuit
solution, because any circuit $C' \subset \supp(h')$ could, by the above process in reverse, be used to produce
a kernel solution with strictly smaller support than $h$, contradicting the assumption that
$h$ is a circuit solution.
Now we have $|h'_j/h'_k|\cdot |h_k/h_i| = |h'_j/h_i| = |\lambda|$ by construction. Thus, $h$ and $h'$ are the circuit solutions
we are looking for.

Now assume $C \neq \set{i,j}$.
If $k \in C$, the statement is trivially true with $C = C_1 = C_2$, so assume $k \notin C$.
Pick $l \in C$, $l \notin \{i,j\}$ and set $B = C\setminus\set{l}$.
Assume without loss of generality that $B \subseteq [m]$ and apply row operations to $A$ such that
$A_{B,B} = \mathbf I_{B\times B}$ is an identity submatrix and
$A_{[m]\setminus B,B} = 0$.
Then the column $A_{l}$ has support given by $B$, for otherwise $g^C$ could not be in the kernel.
The given circuit solution satisfies $g^C_t = -A_{t,l}g^C_l$ for all $t \in B$,
and in particular $g^C_j/g^C_i = A_{j,l}/A_{i,l}$.

Take any circuit solution $h \in \ker(A)$ such that $l, k \in \supp(h)$ and such that $C \cup \supp(h)$
is inclusion-wise minimal. Such a vectors exists by \Cref{prop:conn}\ref{item:circ}.  Now let $J = \supp(h) \setminus C$.
Because $A_{[m]\setminus B, C} = 0$ and $Ah = 0$, we must have $0 \neq h_J \in \ker(A_{[m]\setminus B, J})$.
We show that we can uniquely lift any vector $x \in \ker(A_{B, C\cup \set{k}})$
to a vector $x' \in \ker(A_{C \cup J})$ with $ x_{C\cup k}'= x$. Since this lift will
send circuit solutions to circuit solutions by uniqueness,
it suffices to find our desired circuits as solutions
to the smaller linear system.

We first prove that $\dim(\ker(A_{[m]\setminus B, J})) = 1$.
For suppose that $\dim(\ker(A_{[m]\setminus B, J})) \geq 2$,
then $|J| \geq 2$ and there would exist some vector $y \in \ker(A_{[m]\setminus B, J})$
linearly independent from $h_J$ with $k \in \supp(y)$. This vector could be uniquely lifted to
a vector $\bar y \in \ker(A)$, and we could then find a linear combination
$h + \alpha \bar y$ such that $\supp(h + \alpha \bar y) \subsetneq C \cup J$
but $l,k\in\supp(h + \alpha \bar y)$.
The existence of such a vector contradicts the minimality of $C \cup \supp(h)$.
As such, we know that $\dim(\ker(A_{[m]\setminus B, J})) = 1$.

This clear linear relation between any two entries in $J$ for any vector in
$\ker(A_{[m]\setminus B, J})$ implies that
we can apply row operations to $A$ such that $A_{B, J}$ has non-zero entries
only in the column $A_{B, \set{k}}$. Note that these row operations leave $A_C$ unchanged
because $A_{[m]\setminus B, C} = 0$. From this, we can see that any
element in $\ker(A_{B, C \cup \set{k}})$ can be uniquely lifted to an element in $\ker(A_{C \cup J})$.
Hence we can focus on $\ker(A_{B, C\cup \set{k}})$.

If $A_{i,k} = A_{j,k} = 0$, then any
$x \in \ker(A_{B,C \cup \set{k}})$ satisfies
$x_i + A_{i,l}x_l = x_j + A_{j,l}x_l = 0$ and, in particular,
any circuit $l,k \in \bar{C} \subset C \cup \{k\}$ contains $\{i,j\} \subset \bar C$
and fulfills $|g^C_j/g^C_i| = |A_{j,l}/A_{i,l}| = |g_j^{\bar C}/g_i^{\bar C}| = |g_j^{\bar C}/g_k^{\bar C}| |g_k^{\bar C}/g_i^{\bar C}|$. Choosing $C_1 = C_2 = \bar C$ concludes the case.

Otherwise we know that $A_{i,k} \neq 0$ or $A_{j,k} \neq 0$, meaning that
$\ker(A_{\set{i,j},\set{i,j,l,k}})$ contains at least one circuit solution with $k$ in its support.
Observe that any circuit in $\ker(A_{\set{i,j},\set{i,j,l,k}})$ can be lifted uniquely to an element in
$\ker(A_{B,C \cup \set{k}})$ since $A_{B,B}$ is an identity matrix and we can set
the entries of $B\setminus\set{i,j}$ individually to satisfy the equalities.
Note that this lifted vector is a circuit as well, again by uniqueness of the lift.
Hence we may restrict our attention to the matrix $A_{\set{i,j},\set{i,j,l,k}}$. If the columns
$A_{\set{i,j},k}, A_{\set{i,j},l}$ are linearly
dependent, then any circuit solution to $A_{\set{i,j},\set{i,j,l}}x = 0, x_l \neq 0$, such as $g^C_{\set{i,j,l}}$,
is easily transformed into a circuit solution to $A_{\set{i,j},\set{i,j,k}}x = 0, x_k \neq 0$ and we are done.

If $A_{\set{i,j},k}, A_{\set{i,j},l}$ are independent, we can write
$A_{\set{i,j},\set{i,j,l,k}} = \begin{psmallmatrix}1 & 0 & a & c \\ 0 & 1 & b & d\end{psmallmatrix}$,
where $g^C_j/g^C_i = b/a$.
For $\alpha = ad-bc$, which is non-zero since
$\alpha = \det(\begin{psmallmatrix} a & c \\ b & d\end{psmallmatrix}) \neq 0$
by the independence assumption, we can check that $(\alpha, 0, -d, b)^\T$ and
$(0,\alpha,c,-a)^\T$ are the circuits we are looking for.
\end{proof}

\subsection[Approximating the condition numbers]{Approximating $\bar\chi$ and $\bar\chi^*$}\label{sec:approx-alg}
Equipped with \Cref{thm:low_bound_chi_bar_with_circuits} and
\Cref{thm:pairwise_circuits_in_matroids},
we are ready to prove \Cref{thm:bar-chi-star}.
Recall that we defined $\Le_{ij}^d := \Le_{ij}^{\diag(d)W} = \Le_{ij} d_j/d_i$ when $d > 0$.
We can similarly define $\hat\Le_{ij}^d := \hat\Le_{ij} d_j/d_i$,
and $\hat\Le_{ij}^d$ approximates $\Le_{ij}^d$ just as in \Cref{thm:pairwise_circuits_in_matroids}.

\barchistar*

\begin{proof}\label{proof:bar-chi-star}
Let us run the algorithm \textsc{Finding-Circuits}$(A)$ described in
 \Cref{thm:pairwise_circuits_in_matroids} to obtain the values
$\hat\Le_{ij}$ such that  $\hat\Le_{ij} \le \Le_{ij} \le
(\Le^*_W)^2\hat\Le_{ij}$.
We let $G=([n],E)$ be the circuit ratio digraph, that is,
$(i,j)\in E$ if $\Le_{ij}>0$.

To show the first statement on approximating $\bar\chi$, we simply set
$\xi=\max_{(i,j)\in E}\hat \Le_{ij}$.
Then,
\[
\xi\le \kappa_W \le\bar\chi_W\le  n\kappa_W\le n (\kappa^*_W)^2\xi\le n (\bar\chi^*_W)^2\xi\]
follows by \Cref{thm:chi-approx}.

For the second statement on finding a nearly optimal rescaling for
$\bar\chi^*_W$, we
consider the following optimization problem,
which is an approximate version of \eqref{sys: Le_star}
from \Cref{thm:low_bound_chi_bar_with_circuits}.
\begin{align}
\begin{aligned} \label{sys: apx_for_Le_star}
\min \; & t \\
\hat\Le_{ij}d_j/d_i & \leq t \quad\forall (i,j) \in E \\
d &> 0.
\end{aligned}
\end{align}
Let $\hat d$ be an optimal solution to \eqref{sys: apx_for_Le_star} with value $\hat t$.
We will prove that $\Le^{\hat d} \leq (\Le^*_W)^3$.

First, observe that
$\Le_{ij}^{\hat d} = \Le_{ij}\hat d_j/\hat d_i \leq (\Le^*_W)^2 \hat\Le_{ij} \hat d_j/\hat d_i \leq (\Le^*_W)^2 \hat t$
for any $(i,j) \in E$.
Now, let $d^* > 0$ be such that $\Le^{d^*} = \Le^*_W$.
The vector $d^*$ is a feasible solution to \eqref{sys: apx_for_Le_star}, and so
$\hat t \leq \max_{i\neq j} \hat\Le_{ij}d^*_j/d^*_i \leq \max_{i\neq
  j} \Le_{ij}d^*_j/d^*_i = \Le^{d^*}$. Hence we find that $\hat d$
gives a rescaling with
\[
\bar\chi_{W\widehat D} \leq n\Le^{\hat d} \leq n(\Le^*_W)^3\le n(\bar\chi_W)^3\, ,
\]
where we again used \Cref{thm:chi-approx}.

We can obtain the optimal value $\hat t$ of \eqref{sys: apx_for_Le_star}
by solving the corresponding maximum-mean cycle problem
(see \Cref{thm:low_bound_chi_bar_with_circuits}). It is easy to develop a
multiplicative version of the standard dynamic programming algorithm
of the classical minimum-mean cycle problem
(see e.g. \cite[Theorem 5.8]{amo}) that allows finding the optimum to
\eqref{sys: apx_for_Le_star} directly, in the same $O(n^3)$ time.

\medskip

It is left to find the labels $d_i>0$, $i \in [n]$ such that $\hat\Le_{ij}d_j/d_i \leq \hat t$ for all $(i,j) \in E$.
We define the following weighted directed graph. We associate the weight
$w_{ij}=\log \hat t - \log \hat\Le_{ij}$ with every $(i,j)\in E$, and add an extra
source vertex $r$ with edges $(r,i)$ of weight $w_{ri}=0$ for all
$i\in [n]$.

By the choice of $\hat t$, this graph does not contain any negative weight
directed cycles.
We can compute the shortest paths from $r$ to all nodes in $O(n^3)$ using
the Bellman-Ford algorithm; let $\sigma_i$ be the shortest path label
for $i$. We then set $d_i=\exp(\sigma_i)$. \b{One can avoid computing logarithms by using a multiplicative variant of the Bellman-Ford algorithm instead.}

The running time of the whole algorithm will be bounded by
$O(n^2m^2 + n^3)$. The running time is dominated by the $O(n^2m^2)$
complexity of \textsc{Finding-Circuits}$(A)$ and the $O(n^3)$
complexity of solving the minimum-mean cycle problem and shortest
path computation.
\end{proof}

\section{A scaling-invariant layered least squares interior-point algorithm}
\label{sec:scaling-invariant}

\subsection{Preliminaries on interior-point methods}
\label{sec:ipm-prelims}
In this section, we introduce the standard definitions, concepts and
results from the interior-point literature that will be required for
our algorithm.
We consider an LP problem in the form \eqref{LP_primal_dual}, or
equivalently, in the subspace form \eqref{LP-subspace} for $W=\ker(A)$.
We let
\[
\mathcal{P}^{++} = \{x \in \R^n: Ax = b, x > 0\}\, , \quad
\mathcal{D}^{++} = \{(y,s) \in \R^{m+n}: A^\top y + s = c, s > 0\}\, .
\]
Recall the  {\em central path} defined in \eqref{eq:central-path}, with
$w(\mu)=(x(\mu),y(\mu),s(\mu))$ denoting the central path point
corresponding to $\mu>0$.
We let $w^*=(x^*,y^*,s^*)$ denote the  primal and dual optimal
solutions to \eqref{LP_primal_dual} that correspond to the limit of the central path for
$\mu\to 0$.

For a point  $w = (x, y, s) \in \mathcal{P}^{++} \times \mathcal{D}^{++}$, the \emph{normalized
  duality gap} is $\mu(w)=x^\top s/n$.

The {\em $\ell_2$-neighborhood of the central path with opening
$\beta>0$} is the set
\begin{align*}
\mathcal{N}(\beta) &= \left\{w \in \mathcal{P}^{++} \times \mathcal{D}^{++}: \norm{\frac{xs}{\mu(w)} - e} \leq \beta \right\}\, .
\end{align*}
Furthermore, we let $\overline{\mathcal{N}}(\beta)\coloneqq \operatorname{cl}(\mathcal{N}(\beta))$ denote the closure of $\mathcal{N}(\beta)$. 
Throughout the paper, we will assume $\beta$ is chosen from $(0,1/4]$;
in Algorithm~\ref{alg:overall} we use the value $\beta=1/8$.
The following proposition gives a bound on the distance between $w$
and $w(\mu)$ if $w\in {\cal N}(\beta)$. See e.g., \cite[Lemma 5.4]{Gonzaga92}, \cite[Proposition 2.1]{Monteiro2003}.
\begin{proposition}\label{prop:near-central} Let $w = (x, y, s) \in
  {\cal N}(\beta)$ for $\beta\in (0,1/4]$ and  $\mu=\mu(w)$, and
  consider the central path point
$w(\mu)=(x(\mu),y(\mu),s(\mu))$. For each $i\in[n]$,
\[
\begin{aligned}
\frac{x_i}{1+2\beta}\le \frac{1-2\beta}{1-\beta}\cdot x_i&\le x_i(\mu)\le
\frac{x_i}{1-\beta}\, ,\quad \mbox{and}\\
\frac{s_i}{1+2\beta}\le  \frac{1-2\beta}{1-\beta}\cdot s_i&\le s_i(\mu)\le
\frac{s_i}{1-\beta}\, .
\end{aligned}
\]
\end{proposition}

We will often use the following proposition which is immediate from definiton of $\cal \beta$.
\begin{proposition}\label{prop:x_i-s_i}
Let $w = (x, y, s) \in
  {\cal N}(\beta)$ for $\beta\in (0,1/4]$, and $\mu=\mu(w)$. Then for each $i \in [n]$
\[
(1-\beta)\sqrt{\mu}\le \sqrt{s_i x_i}\le (1+\beta)\sqrt{\mu}\, .
\]
\end{proposition}
\begin{proof}
By definition of $\cal N(\beta)$ we have for all $i \in [n]$ that
$|\frac{x_is_i}{\mu} - 1| \le \|\frac{x s}{\mu} - e\| \le \beta$ and so $(1-\beta) \mu \le x_is_i \le (1+\beta) \mu$. Taking roots gives the results.
\end{proof}

A key property of the central path is \emph{``near monotonicity''},
formulated in the following lemma, see \cite[Lemma 16]{Vavasis1996}.
\begin{lemma}\label{lem: central_path_bounded_l1_norm}
Let $w = (x, y, s)$ be a central path point for $\mu$ and $w' = (x',
y', s')$ be a central path point for $\mu' \leq \mu$. Then
$\|x'/x + s'/s\|_\infty \leq n$. Further, for the optimal
solution $w^*=(x^*,y^*,s^*)$ corresponding to the central path limit
$\mu\to 0$, we have  $\|x^*/x\|_1 + \|s^*/s\|_1 = n$.
\end{lemma}
\begin{proof}
We  show that $\|x'/x\|_1 + \|s'/s\|_1
\leq 2n$ for any feasible primal $x'$ and dual $(y',s')$ such that $(x')^\top s'\le x^\top s=n\mu$; this implies the first statement with the weaker bound $2n$. For the stronger bound  $\|x'/x + s'/s\|_\infty \leq n$, see the proof of \cite[Lemma 16]{Vavasis1996}.
Since $x-x'\in W$ and $s-s'\in W^\perp$, we have $(x-x')^\top
(s-s')=0$. This can be rewritten as $x^\top s'+(x')^\top s=x^\top s+
(x')^\top s'$. By our assumption on $x'$ and $s'$, the right hand side is bounded by $2n\mu$.
 Dividing by
$\mu$, and noting that $x_is_i=\mu$ for all $i\in [n]$, we obtain
\[
\left\|\frac{x'}{x}\right\|_1+\left\|\frac{s'}{s}\right\|_1=\sum_{i=1}^n \frac{x'_i}{x_i}+\frac{s'_i}{s_i} \le 2n\, .
\]
The second statement follows
by using this to central path points $(x',y',s')$ with parameter $\mu'$, and taking the limit $\mu'\to 0$.
\end{proof}

\subsection{The affine scaling and layered-least-squares steps}\label{sec:aff-lls}
Given $w = (x,y,s) \in \mathcal{P}^{++} \times \mathcal{D}^{++}$,
the search directions commonly used in interior-point methods are
obtained as the solution $(\Delta x,\Delta y,\Delta s)$ to the
following linear system for some $\sigma\in [0,1]$.
\begin{align}
A \Delta x &= 0 \label{aff:x} \\
A^\top \Delta y + \Delta s &= 0 \label{aff:s}\\
s\Delta x + x \Delta s &=\sigma \mu e -xs \label{aff:sum}
\end{align}
Predictor-corrector methods, such as the Mizuno-Todd-Ye
Predictor-Corrector (MTY P-C) algorithm \cite{MTY}, alternate
between two types of steps. In \emph{predictor steps}, we use
$\sigma=0$. This direction is also called the \emph{affine scaling
  direction}, and will be denoted as $\Delta w^\as=(\Delta x^\as, \Delta y^\as, \Delta s^\as)$
throughout. In \emph{corrector steps}, we use $\sigma=1$. This gives the \emph{centrality
direction}, denoted as $\Delta w^\cs=(\Delta x^\cs, \Delta y^\cs, \Delta s^\cs)$.

In the predictor steps, we  make progress along the central path. Given the search direction on the current iterate $w = (x,y,s) \in \mathcal{N}(\beta)$, the step-length is chosen  such that the line segment between the current and next steps remain in $\overline{\mathcal{N}}(2\beta)$, i.e.,
\begin{align*}
  \alpha^\as \le \sup\{\alpha \in [0,1]\, : \forall \alpha' \in [0 ,\alpha]:  w+ \alpha' \Delta w^\as \in
  \mathcal{N}(2\beta)\}.
\end{align*}
Thus, we obtain a point $w^+=w+\alpha^\as \Delta w^\as\in \overline{\cal N}(2\beta)$.
The corrector step finds a next iterate $w^c=w^+ +\Delta w^\cs$, where
$\Delta w^\cs$ is the centrality direction computed at $w^+$. The next
proposition summarizes well-known properties, see e.g. \cite[Section 4.5.1]{Ye-book}.
\begin{proposition}\label{prop:predictor-corrector} Let $w = (x,y,s)
  \in \mathcal{N}(\beta)$ for $\beta\in (0,1/4]$.
\begin{enumerate}[label=(\roman*)]
\item
 For the affine scaling step, we have $\mu(w^+)=(1-\alpha)\mu(w)$.
\label{i:gap}
\item The affine scaling step-length can be chosen as
\[\alpha^\as\ge \max\left\{\frac{\beta}{\sqrt{n}},1-\frac{\|\Delta
    x^\as\Delta s^\as\|}{\beta\mu(w)}\right\}\, .
\]\label{i:stepsize}
\item
 For $w^+ \in \overline{\cal N}(2\beta)$ with $\mu(w^+) > 0$, let $\Delta w^\cs$ be the centrality direction at $w^+$. Then for
  $w^\cs=w^+ +\Delta w^\cs$, we have $\mu(w^\cs)=\mu(w^+)$ and $w^\cs\in {\cal
    N}(\beta)$.\label{i:corrector}
\item After a sequence of $O(\sqrt{n} t)$ predictor and corrector
  steps, we obtain an iterate $w'=(x',y',s')\in {\cal N}(\beta)$ such
  that $\mu(w')\le \mu(w)/2^t$.
\end{enumerate}
\end{proposition}

\paragraph{Minimum norm viewpoint and residuals}
For any point $w = (x,y,s) \in \mathcal{P}^{++} \times \mathcal{D}^{++}$ we define
\begin{align}\label{def:delta}
\delta = \delta(w) = s^{1/2}x^{-1/2} \in \R^n.
\end{align}
With this notation, we can write \eqref{aff:sum} \b{for $\sigma = 0$} in the form
\begin{equation}\label{eq:delta-sum}
\delta\Delta x+\delta^{-1} \Delta s = -s^{1/2}x^{1/2}\, .
\end{equation}
\b{Note that for a point $w(\mu)=(x(\mu),y(\mu),s(\mu))$ on the central path, we have $\delta_i(w(\mu))=s_i(\mu)/\sqrt{\mu}=\sqrt{\mu}/x_i(\mu)$ for all $i\in [n]$.}
From Proposition~\ref{prop:near-central}, we see that if $w\in {\cal
  N}(\beta)$, and $\mu=\mu(w)$, then for each $i\in [n]$,
\begin{equation}\label{eq:delta-beta}
  \sqrt{1-2\beta} \cdot \delta_i(w(\mu)) \le \delta_i(w)\le
  \frac{1}{\sqrt{1-2\beta}} \cdot \delta_i(w(\mu))\, .
\end{equation}
The matrix $\diag(\delta(w))$ will be often used for rescaling in
the algorithm. That is, for the current iterate $w=(x,y,s)$ in the
interior-point method, we will perform projections in the space
$\diag(\delta(w))W$. To simplify notation, for $\delta=\delta(w)$, we use
$L^\delta_I$ and $\kappa^\delta_{ij}$ as shorthands for
$L^{\diag(\delta)W}_I$ and $\kappa^{\diag(\delta)W}_{ij}$. The subspace
$W=\ker(A)$ will be fixed throughout.

It is easy to see from the optimality conditions that the components
of the affine
scaling direction $\Delta w^\as=(\Delta x^\as,\Delta y^\as,\Delta s^\as)$ are the
optimal solutions of the following minimum-norm problems.
\begin{equation}\label{eq:aff-minnorm}
\begin{aligned}
\Delta x^\as &= \argmin_{\Delta x \in \R^n}\{\|\delta(x+\Delta x)\|^2 : A\Delta x = 0\}  \\
(\Delta y^\as, \Delta s^\as) &= \argmin_{(\Delta y, \Delta s) \in \R^m \times \R^n} \{\|\delta^{-1}(s+\Delta s)\|^2 : A^\top \Delta y + \Delta s = 0 \}
\end{aligned}
\end{equation}
Following \cite{MonteiroT05}, for a search direction $\Delta w = (\Delta x, \Delta y, \Delta s)$, we
define the \emph{residuals} as
\begin{align}\label{eq:residuals}
\Rx&\coloneqq  \frac{\delta(x+\Delta x)}{\sqrt{\mu}},& \Rs &\coloneqq  \frac{\delta^{-1}(s
                                                  + \Delta
                                                  s)}{\sqrt{\mu}}\, .
\end{align}
We let $\Rx^\as$ and $\Rs^\as$ denote the residuals for the affine scaling direction $\Delta w^\as$.
Hence, the primal affine scaling direction  $\Delta x^\as$ is the one
that minimizes the $\ell_2$-norm of the primal residual $\Rx^\as$, and the dual affine scaling
direction $(\Delta y^\as,\Delta s^\as)$ minimizes the $\ell_2$-norm of
the dual residual $\Rs^\as$. The next lemma summarizes simple properties of
the residuals, see \cite{MonteiroT05}.

\begin{lemma}\label{lem:affscale}
 For $\beta \in (0,1/4]$ such that $w = (x,y,s) \in \mathcal{N}(\beta)$ and the affine scaling direction $\Delta w = (\Delta x^\as, \Delta y^\as,
  \Delta s^\as)$, we have
\begin{enumerate}[label=(\roman*)]
\item\label{i:affscale-identity}
\begin{align}
\Rx^\as \Rs^\as=\frac{\Delta x^\as\Delta s^\as}{\mu},\quad  \Rx^\as+\Rs^\as=\frac{x^{1/2}s^{1/2}}{\sqrt{\mu}}\, , \label{eq: affscal_identity}
\end{align}
\item\label{i:sum-residuals}
\[
\|\Rx^\as\|^2+\|\Rs^\as\|^2= n \, ,
\]
\item\label{i:lower-bound-in-neighbourhood}
We have $\|\Rx^\as\|,\|\Rs^\as\|\le\sqrt{n}$, and for each $i\in [n]$,
$\max\{\Rx_i^\as, \Rs_i^\as\} \geq \frac12(1-\beta)$.
\item\label{i:affscal:residuals_relation_to_dual_side}
\[
\Rx^\as = -\frac{1}{\sqrt{\mu}}\delta^{-1}\Delta s^\as, \quad \Rs^\as =
- \frac{1}{\sqrt{\mu}}\delta \Delta x^\as \, .
\]
\end{enumerate}
\end{lemma}
\begin{proof}
 \b{ Parts \ref{i:affscale-identity} and
  \ref{i:affscal:residuals_relation_to_dual_side} are immediate from
  the definitions and from \eqref{aff:x}-\eqref{aff:sum} and \eqref{eq:delta-sum}. In part \ref{i:sum-residuals}, we use part  \ref{i:affscale-identity} and $({\Rx^\as})^\top \Rs^\as=0$.
  In part, \ref{i:lower-bound-in-neighbourhood}, the first statement follows
by part \ref{i:sum-residuals}, and the
second statement
 follows from
  \ref{i:affscale-identity} and \Cref{prop:x_i-s_i}.}
\end{proof}
For a subset $I \subset [n]$,
we define
\begin{equation}
\epsilon^\as_I(w) \coloneqq  \max_{i \in I} \min\{|\Rx^\as_i|, |\Rs^\as_i|\}\, ,\quad
  \mbox{and}\quad \epsilon^\as(w) \coloneqq  \epsilon_{[n]}^\as(w)\, . \label{def:epsilon}
\end{equation}

The next claim shows that for the affine scaling direction, a small
$\epsilon(w)$ yields a long step; see \cite[Lemma 2.5]{MonteiroT05}.
\begin{lemma}\label{lem:affscale-progress}
Let $w = (x,y,s) \in \mathcal{N}(\beta)$ for $\beta\in (0,1/4]$. Then the affine scaling
step can be chosen such that 
\[
\frac{\mu(w+\alpha^\as \Delta w^\as) }{\mu(w)}\le
\min \left\{1-\frac{\beta}{\sqrt n},\frac{2\sqrt{n}\epsilon^\as(w)}{\beta}\right\}\, .
\]
\end{lemma}
\begin{proof}
Let $\epsilon\coloneqq \epsilon^\as(w)$.
From Lemma~\ref{lem:affscale}\ref{i:affscale-identity}, we get
$\|\Delta x^\as\Delta s^\as\|/\mu=\|\Rx^\as \Rs^\as\|$.
We can bound $\|\Rx^\as \Rs^\as\|\le \epsilon(\|\Rx^\as\|+\|\Rs^\as\|)\le
2\epsilon\sqrt{n}$, where the latter inequality follows by Lemma~\ref{lem:affscale}\ref{i:lower-bound-in-neighbourhood}.
From
Proposition~\ref{prop:predictor-corrector}\ref{i:stepsize}, we get
$\alpha^\as\ge\max\{\beta/\sqrt n,1-2\sqrt{n}\epsilon/\beta\}$. The claim follows by
part \ref{i:gap} of the same proposition.
\end{proof}

\subsubsection{The layered-least-squares direction}
\label{sec:lls}

Let ${\cal J}=(J_1,J_2,\ldots, J_p)$ be an \emph{ordered partition} of
$[n]$.\footnote{In contrast to how ordered partitions were defined in \cite{MonteiroT05}, we use the term \emph{ordered} only to the $p$-tuple $(J_1, \ldots, J_p)$, which is to be viewed independently of $\delta$.}
For $k\in [p]$, we use the notations $J_{<k}\coloneqq J_1\cup \ldots\cup
J_{k-1}$, $J_{>k}\coloneqq J_{k+1}\cup\ldots\cup J_p$, and similarly $J_{\le
  k}$ and $J_{\ge k}$. We will also refer to the sets $J_k$ as
\emph{layers}, and ${\cal J}$ as a \emph{layering}. Layers with lower indices will be referred to as `higher' layers.

Given $w = (x,y,s) \in \mathcal{P}^{++} \times \mathcal{D}^{++}$, and
the layering ${\cal J}$, the \emph{layered-least-squares
  (LLS) direction} is defined as follows. For the primal direction, we
proceed backwards, with $k=p,p-1,\ldots,1$. Assume the
components on the lower layers $\Delta x_{J_{>k}}^\lal$ have already been determined.
 We define the
components in $J_k$ as the coordinate projection
$\Delta x_{J_k}^\lal = \pi_{J_k}(X_k)$, where the affine subspace $X_k$ is
defined as the set of minimizers
\begin{align}\label{def:aff-x}
X_k &\coloneqq  \argmin_{\Delta x \in \R^n}\left\{\left\|\delta_{J_k}(x_{J_k} +
           \Delta x_{J_k})\right\|^2 :\, A\Delta x=0, \Delta x_{J_{>k}} =
           \Delta x_{J_{>k}}^\lal \right\}\, .
\end{align}
The dual direction $\Delta s^\lal$ is determined in the forward
order of the layers $k=1,2,\ldots, p$. Assume we already fixed the
components $\Delta s_{J_{<k}}^\lal$ on the higher layers. Then,
$\Delta s_{J_k}^\lal =
\pi_{J_k}(S_k)$ for
\begin{align}\label{def:aff-s}
S_k &= \argmin_{\Delta s \in \R^{n}}\left\{\left\|\delta_{J_k}^{-1}(s_{J_k} +
           \Delta s_{J_k})\right\|^2 :\,\exists y \in \R^m, A^\top\Delta y+\Delta s=0, \Delta
           s_{J_{<k}} = \Delta s_{J_{<k}}^\lal \right\}\, .
\end{align}
The component $\Delta y^\lal$ is obtained as the optimal $\Delta
y$ for the final layer $k=p$. We use the notation $\Rx^\lal$ and
$\varepsilon^\lal(w)$ analogously to the affine scaling direction.
This search direction was first introduced in \cite{Vavasis1996}.

\medskip

The affine scaling direction is a special
case for the single element partition. In this case, the definitions
\eqref{def:aff-x} and \eqref{def:aff-s}
coincide with those in \eqref{eq:aff-minnorm}.

\subsection{Overview of ideas and techniques}\label{sec:lls-overview}

A key technique in the analysis of layered least-squares algorithms \cite{Vavasis1996,Monteiro2003,LMT09} is to argue about variables that have `converged'. According to
\Cref{prop:near-central} and \Cref{lem: central_path_bounded_l1_norm},
for any iterate $w=(x,y,s)\in {\cal N}(\beta)$ and the limit optimal solution $w^*=(x^*,y^*,s^*)$, the bounds $x^*_i\le O(n) x_i$ and $s^*_i\le O(n) s_i$ hold.
We informally say that $x_i$ (or $s_i$) has \emph{converged}, if  $x_i\le O(n)x_i^*$ ($s_i\le O(n) s_i^*$) hold for the current iterate. Thus, the value of $x_i$ (or $s_i$) remains within a multiplicative factor $O(n^2)$ for the rest of the algorithm. Note that if $\mu>\mu'$ and  $x_i$ has converged at $\mu$, then $\frac{s_i(\mu')/s_i(\mu)}{\mu'/\mu}\in \left[\frac{1}{O(n^2)},O(n^2)\right]$; thus, $s_i$ keeps ``shooting down'' with the central path parameter.

\paragraph{Converged variables in the affine scaling algorithm} Let us start by showing that at any point of the algorithm, at least one primal or dual variable has converged.

Suppose for simplicity that our current iterate is exactly on the central path,
i.e., that $xs = \mu e$. This assumption will be maintained throughout this overview. In this case, the residuals can be simply written
as $\Rx^\as=(x+\Delta x^\as)/x$, $\Rs^\as=(s+\Delta s^\as)/s$.
Recall from \eqref{eq:aff-minnorm} that the affine scaling direction corresponds to minimizing the residuals
$\Rx^\as$ and $\Rs^\as$. From this choice, we see that
\begin{equation}\label{eq:x-star-R-x}
 \left\| \frac{x^*}{x} \right\| \geq \left\|\frac{x + \Delta x^\as}{x}\right\|\, ,\quad
 \left\| \frac{s^*}{s} \right\| \geq \left\|\frac{s + \Delta s^\as}{s}\right\|\, .
\end{equation}
We have $\|\Rx^\as\|^2 + \|\Rs^\as\|^2 = n$ by \Cref{lem:affscale}\ref{i:sum-residuals}.
Let us assume  $\|\Rx^\as\|^2\ge n/2$; thus,
there exists a $i \in [n]$ such that $x^*_i \geq x_i/\sqrt 2$. In other words,
just by looking at the residuals, we get the guarantee that a primal or a dual variable has already converged. Based on the value of the residuals, we can guarantee this to be a primal or a dual variable, but cannot identify which particular $x_i$ or $s_i$ this might be.

\medskip

For $\|\Rx^\as\|^2\ge n/2$, a primal variable has already converged \emph{before} performing the predictor and corrector steps.
We now show that even if $\|\Rx^\as\|$ is small, a primal variable will have converged \emph{after} a single iteration. From \eqref{eq:x-star-R-x}, we see that
there is an index $i$ with $x^*_i/x_i \geq \|\Rx^\as\|/\sqrt{n}$.

Furthermore, \Cref{prop:predictor-corrector}\labelcref{i:stepsize} and \Cref{lem:affscale} imply that
$1-\alpha \leq  {\|\Rx^\as\|\cdot\|\Rs^\as\|}/{\beta}\leq {\sqrt{n} \|\Rx^\as\|}/{\beta}$, since $\|\Rs^\as\|\le \sqrt{n}$.
The predictor step moves to $x^+ \coloneqq  x + \alpha \Delta x^\as = (1-\alpha) x + \alpha(x + \Delta x^\as)$.
Hence, $x^+\leq \left(\frac{\sqrt{n} \|\Rx^\as\|}{\beta} + \|\Rx^\as\|\right)x$.
Putting the two inequalities together, we learn that $x^+_i\le O(n)x^*_i$
for some $i \in [n]$.
Since $w^+=(x^+,y^+,s^+)\in \overline{\cal N}(2\beta)$, \Cref{prop:near-central}  implies that $x_i$ will have converged after this iteration.
An analogous argument proves that  some $s_j$ will also have converged after the iteration. We again emphasize that the argument only shows the existence of converged variables, but we cannot identify them in general.

\paragraph{Measuring combinatorial progress}
Tying the above together, we find that after a single affine scaling step,
at least one primal variable $x_i$ and at least one dual variable $s_j$ has converged. This means that for any $\mu'<\mu$,
 $\frac{x_i(\mu')/x_j(\mu')}{x_i(\mu)/x_j(\mu)}\in \left[\frac{\mu}{O(n^4)\mu'},\frac{O(n^4)\mu}{\mu'}\right]$;
thus, the ratio of these variables keeps asymptotically increasing. The  $x_i/x_j$ ratios serve as the main progress measure in the Vavasis--Ye algorithm. If $x_i/x_j$ is between $1/(\mathrm{poly}(n)\bar\chi)$ and  $\mathrm{poly}(n)\bar\chi$ before the affine scaling step for
 the pair of converged variables $x_i$ and $s_j$, then after $\mathrm{poly}(n)\log\bar\chi$ iterations, the $x_i/x_j$ ratio must leave this interval  and never return. Thus, we obtain a `crossover-event' that cannot again occur for the same pair of variables. In the affine scaling algorithm, there is no guarantee that $x_i/x_j$ falls in such a bounded interval for the converging variables $x_i$ and $s_j$; in particular, we  may obtain the same pairs of converged variables after each step.

The main purpose of layered-least-squares methods is to proactively force that in every certain number of iterations, some `bounded' $x_i/x_j$ ratios become `large' and remain so for the rest of the algorithm.

\medskip

In our approach, the first main insight is to focus on the scaling invariant quantities $\kappa^W_{ij} x_i/x_j$ instead. For simplicity's sake, we first present the algorithm with the assumption that all values $\kappa^W_{ij}$ are known. We will then explain how this assumption can be removed by using gradually improving estimates on the values.

The combinatorial progress will be observed in the \emph{`long edge graph'}.
\label{longedgegraph}
For a primal-dual feasible point $w = (x,y,s)$ and $\sigma =1/O(n^6)$, this is defined as  $G_{w,\sigma}=([n], E_{w,\sigma})$ with edges $(i,j)$ such that $ \Le^W_{ij} x_i/x_j \ge
\sigma$. Observe that for any $i,j\in [n]$, at least one of $(i,j)$ and $(j,i)$ are long edges: this follows since for any circuit $C$ with $i,j\in C$, we get lower bounds $|g^C_j/g^C_i|\le \Le^W_{ij}$ and $|g^C_i/g^C_j|\le \Le^W_{ji}$.

Intuitively, our algorithm  will enforce the following two types of events. The analysis in Section~\ref{sec:analysis} is based on a potential function analysis capturing roughly the same progress.
\begin{itemize}
 \item For an iterate $w$ and a value $\mu > 0$, we have $i,j\in [n]$ in a strongly
 connected component in $G_{w,\sigma}$ of size $\le \tau$, and for any iterate $w'$ with
        $\mu(w') > \mu$, if $i, j$ are in a strongly connected component of $G_{w',\sigma}$
        then this component has size  $\ge 2\tau$.
 \item For an iterate $w$ and a value $\mu > 0$, we have $(i,j) \notin E_{w,\sigma}$, and for
 any iterate $w'$ with $\mu(w') > \mu$ we have $(i,j) \in E_{w',\sigma}$.
\end{itemize}
At most $O(n^2 \log n)$ such events can happen overall, so if we can prove that
on average an event will happen every $O(\sqrt{n} \log(\bar\chi^*_A + n))$
iterations or the algorithm terminates, then we have the desired convergence bound of
$O(n^{2.5}\log(n) \log(\bar\chi^*_A + n))$  iterations.

\usetikzlibrary{positioning}

\usepgfplotslibrary{groupplots}

\usetikzlibrary{decorations.pathreplacing}
\usetikzlibrary{pgfplots.groupplots}
\usetikzlibrary{arrows,shapes}
\usetikzlibrary{patterns}
\usetikzlibrary{graphs}
\usetikzlibrary{graphs.standard}
\colorlet{colorx}{blue!20}
\colorlet{colorxopt}{blue!90}

\pgfplotstableread{images/layers.dat}\loadedtable

\pgfmathsetmacro\curmu{60}
\pgfmathsetmacro\nextmu{45}
\pgfmathsetmacro\aftermu{35}

\pgfmathsetmacro\barwidth{7}
\pgfmathsetmacro\optwidth{\barwidth - 3}

\pgfplotsset{
    node near coord/.style args={#1/#2/#3}{nodes near coords*={
            \ifnum\coordindex=#1 #2\fi
        },
        scatter/@pre marker code/.append code={
            \ifnum\coordindex=#1 \pgfplotsset{every node near coord/.append style=#3}\fi
        }
    },
    nodes near some coords/.style={ scatter/@pre marker code/.code={},scatter/@post marker code/.code={},node near coord/.list={#1} }
}

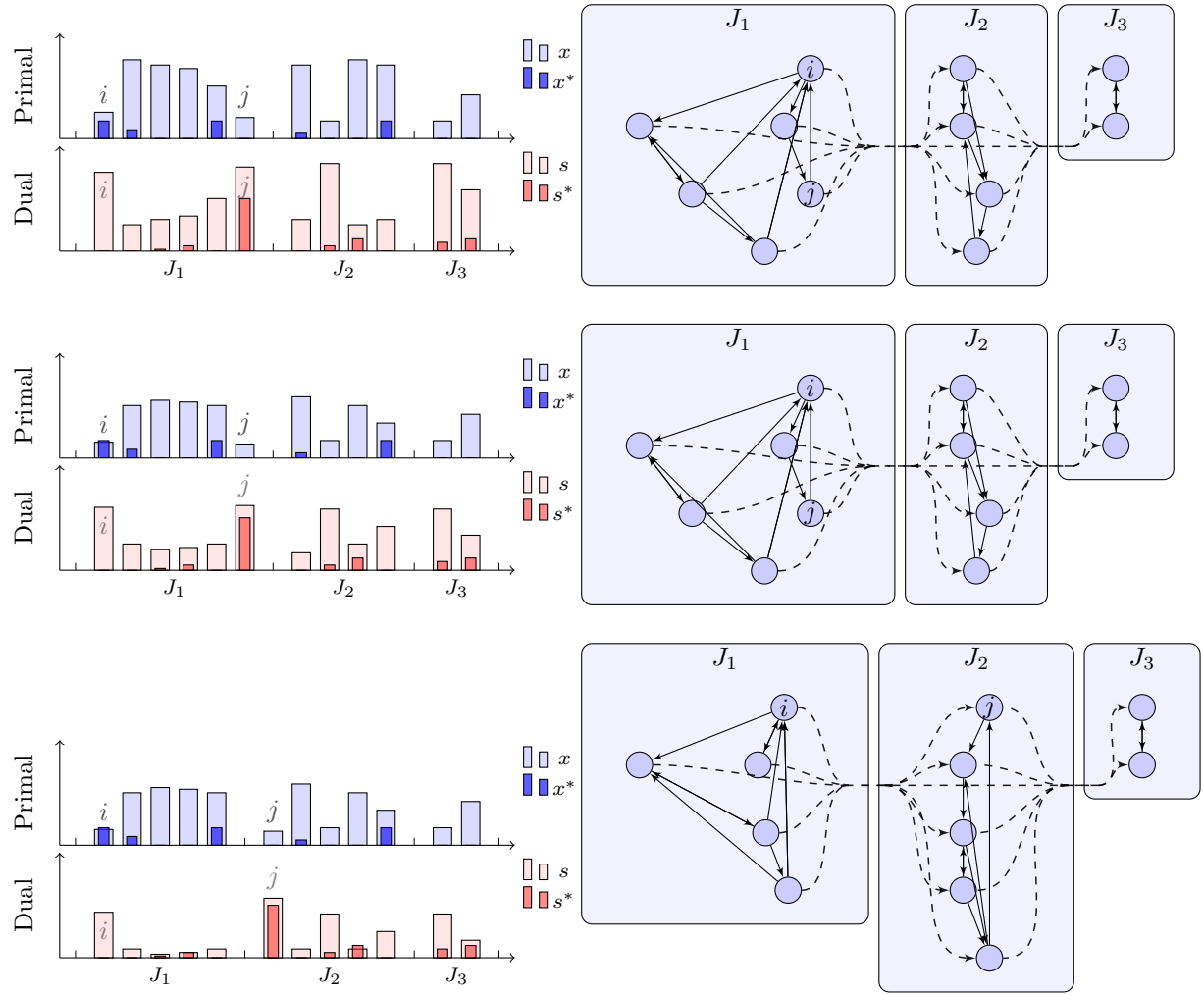
\begin{figure}[ht!]
  \captionsetup{singlelinecheck=off}

  \begin{subfigure}[b]{0.475\textwidth}

\begin{tikzpicture}[scale=1]
    \begin{groupplot}[
    group style={
    group name=my plots,
    group size=1 by 2,
x descriptions at=edge bottom,
    y descriptions at=edge right,
    horizontal sep=0.5cm,
    vertical sep=.1cm,
    },
    footnotesize,
    width= 1.1\linewidth,
    height=3cm,
    ybar,
    ymin=0,
    ymax=60,
    ylabel near ticks,
    axis lines=left,
    axis line style={->},
    enlarge x limits=.12,
y tick style={draw=none},
    yticklabels={},
    xtick={-1,6,11,14},
    xticklabels={$J_1$,$J_2$,$J_3$}, 
    x tick label as interval,
    tick align=inside,
    tick style={black, font=\boldmath},
legend style={at={(1,1)},anchor=north west},
    legend style={draw=none},
    ]
      \nextgroupplot[ylabel=Primal]
  
      \addplot[bar width = \barwidth, fill opacity=.7,fill=colorx,bar shift=0pt, 
      nodes near some coords={0/{$i$}/above,5/{$j$}/above}
      ,]  table [x expr=\coordindex + 1 * (\thisrow{layer} - 1),y=x_cur] from \loadedtable ;
      \addplot[bar width = \optwidth, fill opacity=.7,fill=colorxopt,bar shift=0pt]  table [x expr=\coordindex + 1 * (\thisrow{layer} - 1),y=x_opt] from \loadedtable ;
      
      \legend{$x$,$x^*$}
      
      \nextgroupplot[ylabel=Dual]
  
      \addplot[bar width = \barwidth, fill opacity=0.5,fill=red!20,bar shift=0pt,
      nodes near some coords={0/{$i$}/below,5/{$j$}/below}
      ] table [x expr=\coordindex + 1 * (\thisrow{layer} - 1),y expr= (\curmu - \thisrow{x_cur}),] from \loadedtable ;
      \addplot[bar width = \optwidth, fill opacity=0.5,fill=red!90,bar shift=0pt] table [x expr=\coordindex + 1 * (\thisrow{layer} - 1),y expr=\thisrow{s_opt})] from \loadedtable ;
  
      \legend{$s$,$s^*$}
    \end{groupplot} 
  \end{tikzpicture}    \end{subfigure}
  \hfill
  \begin{subfigure}[b]{0.475\textwidth}
    \begin{tikzpicture}[>=latex',scale=0.5]
\tikzstyle{n} = [draw,shape=circle,minimum size=1em,
                        inner sep=0pt,fill=colorx]
    
    \begin{scope}
  \pgfsetstrokecolor{black}
  \definecolor{strokecol}{rgb}{0.0,0.0,0.0};
  \pgfsetstrokecolor{strokecol}
  \draw [pattern color=red,name=cl1,rounded corners,,fill=blue!5] (8.0bp,8.0bp) -- (8.0bp,223.0bp) -- (246.0bp,223.0bp) -- (246.0bp,8.0bp) -- cycle;
  \draw (127.0bp,211.5bp) node {$J_1$};
\end{scope}
\begin{scope}
  \pgfsetstrokecolor{black}
  \definecolor{strokecol}{rgb}{0.0,0.0,0.0};
  \pgfsetstrokecolor{strokecol}
  \draw [pattern color=red,name=cl2,rounded corners,fill=blue!5] (254.0bp,8.0bp) -- (254.0bp,223.0bp) -- (362.0bp,223.0bp) -- (362.0bp,8.0bp) -- cycle;
  \draw (308.0bp,211.5bp) node {$J_2$};
\end{scope}
\begin{scope}
  \pgfsetstrokecolor{black}
  \definecolor{strokecol}{rgb}{0.0,0.0,0.0};
  \pgfsetstrokecolor{strokecol}
  \draw [pattern color=red,rounded corners,fill=blue!5] (370.0bp,104.0bp) -- (370.0bp,223.0bp) -- (458.0bp,223.0bp) -- (458.0bp,104.0bp) -- cycle;
  \draw (414.0bp,211.5bp) node {$J_3$};
\end{scope}
  \node (1) at (182.0bp,174.0bp) [n] {$$};
  \node (2) at (52.0bp,130.0bp) [n] {$$};
  \node (5) at (162.0bp,130.0bp) [n] {$$};
  \node (3) at (92.0bp,78.0bp) [n] {$$};
  \node (4) at (147.0bp,34.0bp) [n] {$$};
  \node (6) at (182.0bp,78.0bp) [n] {$$};
  \node (7) at (298.0bp,174.0bp) [n] {$$};
  \node (8) at (298.0bp,130.0bp) [n] {$$};
  \node (9) at (318.0bp,78.0bp) [n] {$$};
  \node (10) at (308.0bp,34.0bp) [n] {$$};
  \node (11) at (414.0bp,174.0bp) [n] {$$};
  \node (12) at (414.0bp,130.0bp) [n] {$$};
  \draw [->] (1) -- (2);
  \draw [->] (1) -- (5);
  \draw [->] (2) -- (3);
  \draw [->] (3) -- (1);
  \draw [->] (3) -- (2);
  \draw [->] (3) -- (4);
  \draw [->] (4) -- (1);
  \draw [->] (4) -- (1);
  \draw [->] (4) -- (2);
  \draw [->] (5) -- (6);
  \draw [->] (6) -- (1);
  \draw [->] (7) -- (8);
  \draw [->] (7) -- (9);
  \draw [->] (8) -- (7);
  \draw [->] (8) -- (9);
  \draw [->] (9) -- (10);
  \draw [->] (10) -- (8);
  \draw [->] (11) -- (12);
  \draw [->] (12) -- (11);
    
    \node (j1) at (1) {$i$};
    \node (i1) at (6) {$j$};
    
    \node[draw=none,below right = 3em] at (1.south)  (M1in) {};
    \node[draw=none,right=1em] at (M1in) (M2out) {};
  
    \node[draw=none,right=6em] at (M1in) (M2in) {};
    \node[draw=none,right=.8em] at (M2in) (M3out) {};
  
    \tikzset{
    crosslayer/.style={
      opacity=.9, dashed, line width=.2mm, looseness=1
    },
  }
  
    \foreach \i in {11,...,12} 
       \draw[->,crosslayer] (M3out.center) to[out=0,in=180, looseness=.8] (\i);
  
    \foreach \i in {7,...,10} 
        { \draw[->,crosslayer] (M2out.center) to[out=0, in=180] (\i);
        \draw[crosslayer] (\i)  to[in=180, out=0, looseness=0.5] (M2in.center);
        }
    \foreach \j in {1,...,6} 
        \draw[crosslayer] (\j) to[in=180, out=0] (M1in.center);

    \draw[crosslayer] (M2out.center) to[out=180,in=0] (M1in.center);
    \draw[crosslayer] (M3out) to[out=180,in=0] (M1in);  
    \draw[crosslayer] (M3out.center) to[out=180,in=0] (M2in.center);

  \end{tikzpicture}   \end{subfigure}
  \begin{subfigure}[b]{0.475\textwidth}

\begin{tikzpicture}
    \edef\mylst{"An arbitrary string","String","Custom label","Not this data"}
    \begin{groupplot}[
    group style={
    group name=my plots,
    group size=1 by 2,
x descriptions at=edge bottom,
    y descriptions at=edge right,
    horizontal sep=0.5cm,
    vertical sep=.1cm,
    },
    footnotesize,
    width=1.1 \linewidth,
    height=3cm,
    ybar,
    ymin=0,
    ymax=60,
    ylabel near ticks,
    axis lines=left,
    axis line style={->},
    enlarge x limits=.12,
y tick style={draw=none},
    yticklabels={},
    xtick={-1,6,11,14},
    xticklabels={$J_1$,$J_2$,$J_3$}, 
    x tick label as interval,
    tick align=inside,
    tick style={black, font=\boldmath},
legend style={at={(1,1)},anchor=north west},
    legend style={draw=none},
    ]
      \nextgroupplot[ylabel=Primal]
  
      \addplot[bar width = \barwidth, fill opacity=.7,fill=colorx,bar shift=0pt, 
      nodes near some coords={0/{$i$}/above,5/{$j$}/above}
      ,]  table [x expr=\coordindex + 1 * (\thisrow{layer} - 1),y=x_next] from \loadedtable ;
      \addplot[bar width = \optwidth, fill opacity=.7,fill=colorxopt,bar shift=0pt]  table [x expr=\coordindex + 1 * (\thisrow{layer} - 1),y=x_opt] from \loadedtable ;
      
      \legend{$x$,$x^*$}
      
      \nextgroupplot[ylabel=Dual]
  
      \addplot[bar width = \barwidth, fill opacity=0.5,fill=red!20,bar shift=0pt,
      nodes near some coords={0/{$i$}/below,5/{$j$}/above}
      ] table [x expr=\coordindex + 1 * (\thisrow{layer} - 1),y expr= (\nextmu - \thisrow{x_next}),] from \loadedtable ;
      \addplot[bar width = \optwidth, fill opacity=0.5,fill=red!90,bar shift=0pt] table [x expr=\coordindex + 1 * (\thisrow{layer} - 1),y expr=\thisrow{s_opt})] from \loadedtable ;
  
      \legend{$s$,$s^*$}
    \end{groupplot} 
  \end{tikzpicture}   \end{subfigure}
  \hfill
  \begin{subfigure}[b]{0.475\textwidth}
    \begin{tikzpicture}[>=latex',scale=0.5]
\tikzstyle{n} = [draw,shape=circle,minimum size=1em,
                        inner sep=0pt,fill=colorx]

   \begin{scope}
  \pgfsetstrokecolor{black}
  \definecolor{strokecol}{rgb}{0.0,0.0,0.0};
  \pgfsetstrokecolor{strokecol}
  \draw [pattern color=red,name=cl1,rounded corners,,fill=blue!5] (8.0bp,8.0bp) -- (8.0bp,223.0bp) -- (246.0bp,223.0bp) -- (246.0bp,8.0bp) -- cycle;
  \draw (127.0bp,211.5bp) node {$J_1$};
\end{scope}
\begin{scope}
  \pgfsetstrokecolor{black}
  \definecolor{strokecol}{rgb}{0.0,0.0,0.0};
  \pgfsetstrokecolor{strokecol}
  \draw [pattern color=red,name=cl2,rounded corners,fill=blue!5] (254.0bp,8.0bp) -- (254.0bp,223.0bp) -- (362.0bp,223.0bp) -- (362.0bp,8.0bp) -- cycle;
  \draw (308.0bp,211.5bp) node {$J_2$};
\end{scope}
\begin{scope}
  \pgfsetstrokecolor{black}
  \definecolor{strokecol}{rgb}{0.0,0.0,0.0};
  \pgfsetstrokecolor{strokecol}
  \draw [pattern color=red,rounded corners,fill=blue!5] (370.0bp,104.0bp) -- (370.0bp,223.0bp) -- (458.0bp,223.0bp) -- (458.0bp,104.0bp) -- cycle;
  \draw (414.0bp,211.5bp) node {$J_3$};
\end{scope}
  \node (1) at (182.0bp,174.0bp) [n] {$$};
  \node (2) at (52.0bp,130.0bp) [n] {$$};
  \node (5) at (162.0bp,130.0bp) [n] {$$};
  \node (3) at (92.0bp,78.0bp) [n] {$$};
  \node (4) at (147.0bp,34.0bp) [n] {$$};
  \node (6) at (182.0bp,78.0bp) [n] {$$};
  \node (7) at (298.0bp,174.0bp) [n] {$$};
  \node (8) at (298.0bp,130.0bp) [n] {$$};
  \node (9) at (318.0bp,78.0bp) [n] {$$};
  \node (10) at (308.0bp,34.0bp) [n] {$$};
  \node (11) at (414.0bp,174.0bp) [n] {$$};
  \node (12) at (414.0bp,130.0bp) [n] {$$};
  \draw [->] (1) -- (2);
  \draw [->] (1) -- (5);
  \draw [->] (2) -- (3);
  \draw [->] (3) -- (1);
  \draw [->] (3) -- (2);
  \draw [->] (3) -- (4);
  \draw [->] (4) -- (1);
  \draw [->] (4) -- (1);
  \draw [->] (4) -- (2);
  \draw [->] (5) -- (6);
  \draw [->] (6) -- (1);
  \draw [->] (7) -- (8);
  \draw [->] (7) -- (9);
  \draw [->] (8) -- (7);
  \draw [->] (8) -- (9);
  \draw [->] (9) -- (10);
  \draw [->] (10) -- (8);
  \draw [->] (11) -- (12);
  \draw [->] (12) -- (11); 

    \node (j1) at (1) {$i$};
    \node (i1) at (6) {$j$};
    
    \node[draw=none,below right = 3em] at (1.south)  (M1in) {};
    \node[draw=none,right=1em] at (M1in) (M2out) {};
  
    \node[draw=none,right=6em] at (M1in) (M2in) {};
    \node[draw=none,right=.8em] at (M2in) (M3out) {};
  
    \tikzset{
    crosslayer/.style={
      opacity=.9, dashed, line width=.2mm, looseness=1
    },
  }
  
    \foreach \i in {11,...,12} 
       \draw[->,crosslayer] (M3out.center) to[out=0,in=180, looseness=.8] (\i);
  
    \foreach \i in {7,...,10} 
        { \draw[->,crosslayer] (M2out.center) to[out=0, in=180] (\i);
        \draw[crosslayer] (\i)  to[in=180, out=0, looseness=0.5] (M2in.center);
        }
    \foreach \j in {1,...,6} 
        \draw[crosslayer] (\j) to[in=180, out=0] (M1in.center);

    \draw[crosslayer] (M2out.center) to[out=180,in=0] (M1in.center);
    \draw[crosslayer] (M3out) to[out=180,in=0] (M1in);  
    \draw[crosslayer] (M3out.center) to[out=180,in=0] (M2in.center);

  \end{tikzpicture}   \end{subfigure}
  \begin{subfigure}[b]{0.475\textwidth}

\begin{tikzpicture}
    \edef\mylst{"An arbitrary string","String","Custom label","Not this data"}
    \begin{groupplot}[
    group style={
    group name=my plots,
    group size=1 by 2,
x descriptions at=edge bottom,
    y descriptions at=edge right,
    horizontal sep=0.5cm,
    vertical sep=.1cm,
    },
    footnotesize,
    width=1.1 \linewidth,
    height=3cm,
    ybar,
    ymin=0,
    ymax=60,
    ylabel near ticks,
    axis lines=left,
    axis line style={->},
    enlarge x limits=.12,
y tick style={draw=none},
    yticklabels={},
    xtick={-1,5,11,14},
    xticklabels={$J_1$,$J_2$,$J_3$}, 
    x tick label as interval,
    tick align=inside,
    tick style={black, font=\boldmath},
legend style={at={(1,1)},anchor=north west},
    legend style={draw=none},
    ]
      \nextgroupplot[ylabel=Primal]
  
      \addplot[bar width = \barwidth, fill opacity=.7,fill=colorx,bar shift=0pt, 
      nodes near some coords={0/{$i$}/above,5/{$j$}/above}
      ,]  table [x expr=\coordindex + 1 * (\thisrow{layer_after} - 1),y=x_next] from \loadedtable ;
      \addplot[bar width = \optwidth, fill opacity=.7,fill=colorxopt,bar shift=0pt]  table [x expr=\coordindex + 1 * (\thisrow{layer_after} - 1),y=x_opt] from \loadedtable ;
      
      \legend{$x$,$x^*$}
      
      \nextgroupplot[ylabel=Dual]
  
      \addplot[bar width = \barwidth, fill opacity=0.5,fill=red!20,bar shift=0pt,
      nodes near some coords={0/{$i$}/below,5/{$j$}/above}
      ] table [x expr=\coordindex + 1 * (\thisrow{layer_after} - 1),y expr= (\aftermu - \thisrow{x_after}),] from \loadedtable ;
      \addplot[bar width = \optwidth, fill opacity=0.5,fill=red!90,bar shift=0pt] table [x expr=\coordindex + 1 * (\thisrow{layer_after} - 1),y expr=\thisrow{s_opt})] from \loadedtable ;
  
      \legend{$s$,$s^*$}
    \end{groupplot} 
  \end{tikzpicture}
     \end{subfigure}
  \hfill
  \begin{subfigure}[b]{0.475\textwidth}

\begin{tikzpicture}[>=latex',scale=0.5]
\tikzstyle{n} = [draw,shape=circle,minimum size=1em,
                        inner sep=0pt,fill=colorx]
    
    \begin{scope}
  \pgfsetstrokecolor{black}
  \definecolor{strokecol}{rgb}{0.0,0.0,0.0};
  \pgfsetstrokecolor{strokecol}
  \draw [pattern color=red,name=cl1,rounded corners,,fill=blue!5] (8.0bp,60.0bp) -- (8.0bp,275.0bp) -- (226.0bp,275.0bp) -- (226.0bp,60.0bp) -- cycle;
  \draw (117.0bp,263.5bp) node {$J_1$};
\end{scope}
\begin{scope}
  \pgfsetstrokecolor{black}
  \definecolor{strokecol}{rgb}{0.0,0.0,0.0};
  \pgfsetstrokecolor{strokecol}
  \draw [pattern color=red,name=cl2,rounded corners,fill=blue!5] (234.0bp,8.0bp) -- (234.0bp,275.0bp) -- (382.0bp,275.0bp) -- (382.0bp,8.0bp) -- cycle;
  \draw (308.0bp,263.5bp) node {$J_2$};
\end{scope}
\begin{scope}
  \pgfsetstrokecolor{black}
  \definecolor{strokecol}{rgb}{0.0,0.0,0.0};
  \pgfsetstrokecolor{strokecol}
  \draw [pattern color=red,rounded corners,fill=blue!5] (390.0bp,156.0bp) -- (390.0bp,275.0bp) -- (478.0bp,275.0bp) -- (478.0bp,156.0bp) -- cycle;
  \draw (434.0bp,263.5bp) node {$J_3$};
\end{scope}
  \node (1) at (162.0bp,226.0bp) [n] {$$};
  \node (2) at (52.0bp,182.0bp) [n] {$$};
  \node (5) at (142.0bp,182.0bp) [n] {$$};
  \node (3) at (148.0bp,130.0bp) [n] {$$};
  \node (4) at (165.0bp,86.0bp) [n] {$$};
  \node (6) at (318.0bp,226.0bp) [n] {$$};
  \node (10) at (298.0bp,182.0bp) [n] {$$};
  \node (7) at (298.0bp,86.0bp) [n] {$$};
  \node (8) at (298.0bp,130.0bp) [n] {$$};
  \node (9) at (318.0bp,34.0bp) [n] {$$};
  \node (11) at (434.0bp,226.0bp) [n] {$$};
  \node (12) at (434.0bp,182.0bp) [n] {$$};
  \draw [->] (1) -- (2);
  \draw [->] (1) -- (5);
  \draw [->] (2) -- (3);
  \draw [->] (3) -- (1);
  \draw [->] (3) -- (2);
  \draw [->] (3) -- (4);
  \draw [->] (4) -- (1);
  \draw [->] (4) -- (1);
  \draw [->] (4) -- (2);
  \draw [->] (5) -- (1);
  \draw [->] (6) -- (10);
  \draw [->] (7) -- (8);
  \draw [->] (7) -- (9);
  \draw [->] (8) -- (7);
  \draw [->] (8) -- (9);
  \draw [->] (9) -- (6);
  \draw [->] (9) -- (10);
  \draw [->] (10) -- (8);
  \draw [->] (11) -- (12);
  \draw [->] (12) -- (11); 
    
    \node (j1) at (1) {$i$};
    \node (i1) at (6) {$j$};
    
    \node[draw=none,below right = 3em] at (1.south)  (M1in) {};
    \node[draw=none,right=1em] at (M1in) (M2out) {};
  
    \node[draw=none,right=8em] at (M1in) (M2in) {};
    \node[draw=none,right=.8em] at (M2in) (M3out) {};
  
    \tikzset{
    crosslayer/.style={
      opacity=.9, dashed, line width=.2mm, looseness=1
    },
  }
  
    \foreach \i in {11,...,12} 
       \draw[->,crosslayer] (M3out.center) to[out=0,in=180] (\i);
  
    \foreach \i in {6,...,10} 
        { \draw[->,crosslayer] (M2out.center) to[out=0, in=180] (\i);
        \draw[crosslayer] (M2in.center) to[in=0,out=180] (\i);
        }
    \foreach \j in {1,...,5} 
        \draw[crosslayer] (\j) to[in=180, out=0, looseness=.8] (M1in.center);
    \draw[crosslayer] (M2out.center) to[out=180,in=0] (M1in.center);
    \draw[crosslayer] (M3out) to[out=180,in=0] (M1in);  
    \draw[crosslayer] (M3out.center) to[out=180,in=0] (M2in.center);

\end{tikzpicture}
     \end{subfigure}
  \caption[]{Top-down we have a chart of primal/dual variables and the \emph{estimated} subgraph of the circuit ratio digraph (\Cref{def:aux-graph}) for three different iterations: 
  \begin{enumerate*}[label=\arabic*)]
    \item All variables except $x_i$ are far away from their optimal values. 
    \item On $J_1$ there is a primal variable ($i$) and dual variable ($j$) that have converged, i.e. $x_i$ is close to $x_i^*$ and $s_i$ is close to $s_i^*$. 
    \item $j$ moves to layer $J_2$ due to a change in the underlying subgraph of the circuit ratio digraph.
  \end{enumerate*}
  }
\end{figure}
 
\paragraph{Converged variables cause combinatorial progress}
We now show that combinatorial progress as above must happen in  the affine scaling step in the case when the graph $G_{w,\sigma}$ is  strongly
connected. As noted above, for the pair of converged variables $x_i$ and $s_j$ after the affine scaling step, $x_i/x_j$, and thus $\kappa^W_{ij} x_i/x_j$, will asymptotically increase by a factor 2 in every $O(\sqrt{n})$ iterations.

By the strong connectivity assumption, there is a directed path
in the long edge graph from $i$ to $j$ of length at most $n-1$.
Each edge has length at least $\sigma$, and by the cycle characterization (\Cref{thm:low_bound_chi_bar_with_circuits}) we know that
$(\kappa^W_{ji} x_j/x_i) \cdot \sigma^{n-1} \leq (\kappa_W^*)^n$.
As such, $\kappa^W_{ji} x_j / x_i \leq (\kappa_W^*)^n/\sigma^{n-1}$. Since $\kappa^W_{ij} \kappa^W_{ji}\ge 1$, we obtain the lower bound
$\kappa^W_{ij} x_i / x_j \geq \sigma^{n-1}/(\kappa_W^*)^{n}$.

This means that after $O(\sqrt{n} \log((\kappa_W^*/\sigma)^n)) = O(n^{1.5}\log(\kappa_W^* + n))$
affine scaling steps,
the weight of the edge $(i,j)$ will be more than $(\kappa_W^*/\sigma)^{4n}$.
There can never again be a length $n$ or shorter path from $j$ to $i$ in the long edge graph, for otherwise
the resulting cycle would violate \Cref{thm:low_bound_chi_bar_with_circuits}.
Moreover, by the triangle inequality (\Cref{lem:imbalance_triangle_inequality}), any other $k \neq i,j$ will have either $(i,k)$ or $(k,j)$
of length at least $(\kappa_W^*/\sigma)^{2n}$, similarly causing a pair of variables
to never again be in the same connected component.
As such, we took $O(n^{1.5}\log(\kappa_W^* + n))$ affine scaling steps
and in that time at least $n-1$ combinatorial progress events have occured.

\paragraph{The layered least squares step}
 Similarly to the Vavasis--Ye algorithm \cite{Vavasis1996} and subsequent literature, our algorithm is a predictor-corrector method  using  \emph{layered least squares (LLS)} steps as in Section~\ref{sec:lls}
 for certain predictor iterations.  Our algorithm (\Cref{alg:overall}) uses LLS steps only sometimes, and most steps are the simpler affine scaling steps; but for simplicity of this overview, we can assume every predictor iteration uses an LLS step.

We define the ordered partition
\label{orderedpartition}
${\cal J}=(J_1,J_2,\ldots, J_p)$ corresponding to the strongly connected components in  topological ordering. Recalling that either $(i,j)$ or $(j,i)$ is a long edge for every pair $i,j\in [n]$, this order is unique and such that there is a complete directed graph of long edges from every $J_k$ to $J_{k'}$ for $1\le k<k'\le p$.

The first important property of the LLS step is that it is very close to the affine scaling step.
In Section~\ref{sec:partition-lifting}, we introduce the partition lifting cost
$\ell^W({\cal J})=\max_{2 \leq k \leq p}\ell^W(J_{\ge k})$ as the cost of lifting from lower to higher layers; we let $\ell^{1/x}({\cal J})$ be a shorthand for $\ell^{\diag(1/x)W}({\cal J})$. Note that this same rescaling is used for the affine scaling step in \eqref{eq:aff-minnorm}, since  $\delta=\sqrt{\mu}/x$ if $w$ is on the central path.
In \Cref{lem:ll-decompose}\ref{i:lls-aff}, we show that for a small partition lifting cost, the LLS residuals will remain near the affine scaling residuals. Namely,
\[
\|\Rx^\lal -\Rx^\as\|, \|\Rs^\lal-\Rs^\as\|\le  6n^{3/2}\ell^{1/x}({\cal J})\, .
\]
Recall that the LLS residuals can be written as $\Rx^\lal = ({x + \Delta x^\lal})/{x}$, $\Rs^\lal = (s + \Delta s^\lal)/{s}$ for a point on the central path.
For ${\cal J}$ defined as above, Lemma~\ref{lem:purify}  yields
$\ell^{1/x}({\cal J}) \le
n \max_{i \in J_{> k}, j \in J_{\le  k}, k \in [p]} \Le^W_{ij}{x_i}/{x_j}$.
This will be sufficiently small as this maximum is taken over `short' edges (not in $E_{w,\sigma}$).

\medskip

A second, crucial property of the LLS step is that it
``splits'' our LP into $p$ separate LPs that have ``negligible'' interaction. Namely, the direction $(\Delta x_{J_k}^\lal,\Delta s_{J_k}^\lal)$ will be very close to the affine scaling step obtained in the problem restricted to the subspace
$W_{{\cal J},k} = \{x_{J_k}: x \in W, x_{J_{>k}} = 0\}$ (\Cref{lem:ll-decompose}\ref{i:prox})

Since each component $J_k$
is strongly connected in the long edge graph $G_{w,\sigma}$, if there is at least one primal $x_i$ and dual $s_j$ in $J_k$ that have converged after the LLS step, we can use the above argument to show combinatorial progress regarding the $\Le^W_{ij}x_i/x_j$ value (\Cref{lem:potential-master}).

Exploiting the proximity between the LLS and affine scaling steps, \Cref{lem:ll-decompose}\ref{i:progress} gives a lower bound on the step size
$\alpha\ge 1-\frac{3\sqrt{n}}{\beta}\max_{i\in [n]}\min\{|\Rx_i^\lal|,|\Rs_i^\lal|\}$.
Let $J_k$ be the component  where $\min\{\|\Rx_{J_k}^\lal\|,\|\Rs_{J_k}^\lal\|\}$ is the largest. Hence, the step size $\alpha$ can be lower bounded in terms of $\min\{\|\Rx_{J_k}^\lal\|,\|\Rs_{J_k}^\lal\|\}$.

The analysis now distinguishes two cases. Let $w^+=w+\alpha\Delta s^\lal$ be the point obtained by the predictor LLS step. If the corresponding partition lifting cost $\ell^{1/x^+}({\cal J})$ is still small, then a similar argument that has shown the convergence of primal and dual variables in the affine scaling step will imply that after the LLS step, at least one $x_i$ and one $s_j$ will have converged for $i,j\in J_k$. Thus, in this case we obtain the combinatorial progress (\Cref{lem: case_lemma_does_not_crash}).

The remaining case is when $\ell^{1/x^+}({\cal J})$ becomes large. In \Cref{lem: case_layer_crashes}, we show that in this case a new edge will enter the long edge graph, corresponding to the second combinatorial event listed previously. Intuitively, in this case one layer ``crashes'' into another.

\paragraph{Refined estimates on circuit imbalances} In the above overview, we assumed the circuit imbalance values $\Le^W_{ij}$ are given, and thus the graph $G_{w,\sigma}$ is available. Whereas these quantities are difficult to compute, we can naturally work with lower estimates. For each $i,j\in [n]$ that are contained in a circuit together, we start with the lower bound $\hat\Le^W_{ij}=|g^C_j/g^C_i|$ obtained for an arbitrary circuit $C$ with $i,j\in C$. We use the graph $\hat G_{w,\sigma}=([n],\hat E_{w,\sigma})$ corresponding to these estimates. Clearly, $\hat E_{w,\sigma}\subseteq E_{w,\sigma}$, but some long edges may be missing. We determine the partition $\cal J$ of the strongly connected components of $\hat G_{w,\sigma}$ and estimate the partition lifting cost $\ell^{1/x}({\cal J})$. If this is below the desired bound, the argument works correctly. Otherwise, we can identify a pair $i,j$ responsible for this failure. Namely, we find a circuit $C$ with $i,j\in C$ such that
$\hat \Le^W_{ij}<|g^C_j/g^C_i|$. In this case, we update our estimate, and recompute the partition; this is described in Algorithm~\ref{alg:layering_procedure}. At each LLS step, the number of updates is bounded by $n$, since every update leads to a decrease in the number of partition classes. This finishes the overview of the algorithm.

\subsection{A linear system viewpoint of layered least squares}\label{sec:linsys}
We now continue with the detailed exposition of our algorithm. We
present an equivalent definition of the LLS step introduced in Section~\ref{sec:lls}, generalizing
the linear system \eqref{aff:s}--\eqref{aff:sum}. We use the subspace
notation. With this notation, \eqref{aff:s}--\eqref{aff:sum} for the
affine scaling direction can be written as
\begin{equation}\label{eq:aff-delta}
s\Delta x^\as+x\Delta s^\as=-xs\, , \quad \Delta
x^\as\in W\, , \quad \mbox{and}\quad \Delta s^\as\in
W^\perp\, ,\
\end{equation}
which is further equivalent to $\delta\Delta x^\as+\delta^{-1}\Delta s^\as=-x^{1/2}s^{1/2}$.

Given the layering ${\cal J}$ and $w=(x,y,s)$, for each $k\in [p]$ we define the subspaces
\[
W_{{\cal J},k} \coloneqq  \{x_{J_k}: x \in W, x_{J_{>k}} = 0\}\,\quad
\mbox{and}\quad  W_{{\cal J},k}^\perp \coloneqq  \{x_{J_k}: x \in W^\perp, x_{J_{< k}}
= 0\}\, .\]
We  emphasize that $W_{{\cal J},k}$ and $W_{{\cal J},k}^\perp$ live on the variables in layer $k$. That is, $W_{{\cal J},k},  W_{{\cal J},k}^\perp \subseteq \R^{J_k}$.
It is easy to see that these two subspaces are orthogonal
complements. Our next goal is to show that, analogously to \eqref{eq:aff-delta},
the primal LLS step $\Delta x^{\lal}$ is obtained as the unique
solution to the linear system
\begin{equation}\label{eq:ll-primal}
\delta \Delta x^\lal + \delta^{-1} \Delta s = -x^{1/2} s^{1/2}\, , \quad \Delta x^\lal \in W\,
,\quad \mbox{and}\quad  \Delta s \in W_{{\cal J},1}^\perp \times \cdots
\times W_{{\cal J},p}^\perp\, ,
\end{equation}
and the dual LLS step $\Delta s^{\lal}$ is the unique solution to
\begin{equation}\label{eq:ll-dual}
\delta \Delta x  + \delta^{-1} \Delta s^\lal = -x^{1/2} s^{1/2}\, , \quad \Delta x\in W_{{\cal J},1}
\times \cdots \times W_{{\cal J},p}\, ,\quad \mbox{and} \quad \Delta s^\lal
\in W^\perp\, .
\end{equation}
It is important to note that $\Delta s$ in \eqref{eq:ll-primal} may be
different from $\Delta s^\lal$, and $\Delta x$ in \eqref{eq:ll-dual} may be
different from $\Delta x^\lal$. In fact, $\Delta s^\lal=\Delta s$ and
$\Delta x^\lal=\Delta x$ can only be the case for the affine scaling step.

The following lemma proves that the above linear systems are indeed
uniquely solved by the LLS step.
\begin{lemma}\label{lem:lls-as-linsys}
For $t \in \R^n$, $W \subseteq \R^n$, $\delta \in \R^n_{++}$, and
$\mathcal J = (J_1,J_2,\dots,J_p)$,
let $w = \mathrm{LLS}^{W,\delta}_{\mathcal J}(t)$ be defined by
\[
\delta w + \delta^{-1} v = \delta t,\qquad w \in W, \qquad
v \in W_{{\cal J},1}^\perp \times \cdots \times W_{{\cal J},p}^\perp.
\]
Then $\mathrm{LLS}^{W,\delta}_{\mathcal J}(t)$ is well-defined and
\[
\norm{\delta_{J_k}(t_{J_k} - w_{J_k})}
= \min\set{\norm{\delta_{J_k}(t_{J_k} - z_{J_k})} : z \in W, z_{J_{>k}} = w_{J_{>k}}}
\]
for every $k\in[p]$.
\end{lemma}
In the notation of the above lemma we have,
for ordered partitions $\mathcal J = (J_1,J_2,\dots,J_p)$,
$\bar{\cal J} = (J_p,J_{p-1},\dots,J_1)$,
and $(x,y,s) \in \mathcal P^{++} \times \mathcal D^{++}$
with $\delta = s^{1/2}x^{-1/2}$,
that $\Delta x^\lal = \mathrm{LLS}^{W,\delta}_{\mathcal J}(-x)$
and $\Delta s^\lal = \mathrm{LLS}^{W^\perp,\delta^{-1}}_{\cal {\bar J}}(-s)$.

\begin{proof}[Proof of \Cref{lem:lls-as-linsys}]
We first prove the equality
$W \cap (W^\perp_{\mathcal J,1} \times \dots \times W^\perp_{\mathcal J,p}) = \set{0}$,
and by a similar argument we have
$W^\perp \cap (W_{\mathcal J,1} \times \dots \times W_{\mathcal J,p}) = \set{0}$.
By duality, this last equality tells us that
\[(W^\perp \cap (W_{\mathcal J,1} \times \dots \times W_{\mathcal J,p}))^\perp = W + (W^\perp_{\mathcal J,1} \times \dots \times W^\perp_{\mathcal J,p}) = \R^n.\]
Thus, the linear decomposition defining $\mathrm{LLS}^{W,\delta}_{\mathcal J}(t)$
has a solution and its solution is unique.

Suppose $y \in W \cap (W^\perp_{\mathcal J,1} \times \dots \times
W^\perp_{\mathcal J,p})$. We prove $y_{J_k} = 0$ by induction on $k$, starting
at $k=p$. The induction hypothesis is that $y_{J_{>k}} = 0$, which is an empty
requirement when $k = p$. The hypothesis $y_{J_{>k}} = 0$ together with the
assumption $y \in W$ is equivalent to $y \in W \cap \R^n_{J_{\leq k}}$, and
implies $y_{J_k} \in \pi_{J_k}(W \cap \R^n_{J_{\leq k}}) \coloneqq  W_{{\cal J},k}$. Since
we also have $y_{J_k} \in W_{{\cal J},k}^\perp$ by assumption, which is the
orthogonal complement of $W_{{\cal J},k}$, we must have $y_{J_k} = 0$. Hence, by
induction $y = 0$. This finishes the proof that
$\mathrm{LLS}^{W,\delta}_{\mathcal J}(t)$ is well-defined.

Next we prove that $w$ is a minimizer of $\min\set{\norm{\delta_{J_k}(t_{J_k} -
z_{J_k})} : z \in W, z_{J_{>k}} = w_{J_{>k}}}$.  The optimality condition is for
$\delta_{J_k}(t_{J_k} - z_{J_k})$ to be orthogonal to $\delta_{J_k}u$ for any $u
\in W_{{\cal J},k}$. By the LLS equation, we have $\delta_{J_k}(t_{J_k} - w_{J_k})
= \delta_{J_k}^{-1} v_{J_k}$, where $v_{J_k} \in W^\perp_{\mathcal J, k}$.
Noting then that $\langle \delta_{J_k} u, \delta_{J_k}^{-1} v\rangle = \langle
u_{J_{k}}, v_{J_k} \rangle = 0$ for $u \in W_{{\cal J},k}$, the optimality
condition follows immediately.
\end{proof}

With these tools, we can prove that the lifting costs are self-dual.
This explains the reverse order in the dual vs primal LLS step
and justifies our attention on the lifting cost in a self-dual algorithm.
The next proposition generalizes the result of \cite{gonzaga_lara}.
\begin{proposition}[name=\prooftext\pageref{proof:prop:self-dual}, {restate=[name=\restatetext]propselfdual}]\label{prop:self-dual}
For a linear subspace $W \subseteq \R^n$ and index set $I \subseteq [n]$
with $J = [n]\setminus I$,
\[
\|L_I^W\| \leq \max\{1, \|L_J^{W^\perp}\|\}.
\]
In particular, $\ell^W(I) = \ell^{W^\perp}(J)$.
\end{proposition}
We defer the proof to Section~\ref{sec:prox-proof}.
Note that this proposition also implies \Cref{prop:chibar}\ref{i:chibar-dual}.

\subsubsection{Partition lifting scores}\label{sec:partition-lifting}
A key insight is that if the layering $\cal J$ is ``well-separated'',
then we indeed have $x \Delta s^\lal  + s \Delta x^\lal \approx -xs$,
that is, the LLS direction is close to the affine scaling
direction. This will be shown in Lemma~\ref{lem:ll-decompose}.
The notion of ``well-separatedness'' can be formalized as follows.
Recall the definition of the lifting score \eqref{def:lifting-score}.
The lifting score of the layering ${\cal J}=(J_1,
J_2,\ldots, J_p)$ of $[n]$ with respect to $W$ is defined as
\[
\ell^W({\cal J})\coloneqq \max_{2 \leq k \leq p}\ell^W(J_{\ge k})\, .
\]
For $\delta\in \R^n_{++}$, we use $\ell^{W,\delta}(I)
\coloneqq \ell^{\diag(\delta)W}(I)$ and $\ell^{W,\delta}({\cal J})
\coloneqq \ell^{\diag(\delta)W}({\cal J})$. When the context is clear, we omit $W$ and
write $\ell^{\delta}(I) \coloneqq  \ell^{W,\delta}(I)$ and $\ell^\delta({\cal J}) \coloneqq 
\ell^{W,\delta}({\cal J})$.

The following important duality claim asserts that the lifting score
of a layering equals the lifting score of the reverse layering in the
orthogonal complement subspace. It is an immediate consequence of \Cref{prop:self-dual}.
\begin{lemma}\label{lem:balanced-dual}
Let $W \subseteq \R^n$ be a linear subspace, $\delta \in \R^n_{++}$.
For an ordered partition ${\cal J}=(J_1,J_2,\ldots, J_p)$, let $\mathcal{
    \bar J}=(J_p,J_{p-1},\ldots,J_1)$ denote the reverse ordered
    partition.
Then, we have
\[
\ell^{W,\delta}({\cal J})=\ell^{W^\perp,\delta^{-1}}(\cal{
  \bar J}).
\]
\end{lemma}
\begin{proof}
Let $U = \diag(\delta)W$. Note that $U^\perp = \diag(\delta^{-1}) W^\perp$. Then by \Cref{prop:self-dual}, for $2 \leq k \leq p$, we
have that
\[
\ell^{W,\delta}(J_{\geq k}) = \ell^{U}(J_{\geq k}) = \ell^{U^\perp}(J_{\leq k-1}) =
\ell^{U^\perp}(\bar{J}_{\geq p-k+2}) = \ell^{W^\perp,\delta^{-1}}(\bar{J}_{\geq
p-k+2}).
\]
In particular, $\ell^{W,\delta}({\cal J}) = \ell^{W^\perp,\delta^{-1}}(\cal{
\bar J})$, as needed.
\end{proof}

The next lemma summarizes key properties of the LLS steps, assuming
the partition has a small lifting score. We show that if $\ell^\delta({\cal J})$ is sufficiently small, then on the one hand, the LLS step will be very close to the affine scaling step, and on the other hand, on each layer $k\in [p]$, it will be very close to the affine scaling step restricted to this layer for the subspace $W_{{\cal J},k}$.
The proof is
deferred to Section~\ref{sec:prox-proof}.

\begin{lemma}[name=\prooftext\pageref{proof:lem:ll-decompose}, {restate=[name=\restatetext]lldecompose}]\label{lem:ll-decompose}
Let  $w=(x,y,s)\in {\cal N}(\beta)$ for $\beta\in (0,1/4]$, let
$\mu=\mu(w)$ and  $\delta=\delta(w)$.
Let ${\cal J}=(J_1,\ldots,J_p)$ be a layering with
$\ell^\delta({\cal J})\le \beta/(32 n^2)$, and let $\Delta w^\lal = (\Delta x^\lal, \Delta y^\lal, \Delta
s^\lal)$ denote the LLS direction for the layering
${\cal J}$. Let furthermore
$\epsilon^\lal(w)=\max_{i\in [n]}\min\{|\Rx_i^\lal|,|\Rs_i^\lal|\}$, and
define the maximal step length as
\begin{align*}
\alpha^* &\coloneqq  \sup\{\alpha' \in [0,1] : \forall \bar\alpha \in [0,\alpha']: w + \bar\alpha \Delta w^\lal \in
         \mathcal{N}(2\beta)\}\, .
\end{align*}
Then the following properties hold.
\begin{enumerate}[label=(\roman*)]
\item\label{i:prox} We have
\begin{align}
\|\delta_{J_k} \Delta x^\lal_{J_k} + \delta^{-1}_{J_k} \Delta s^\lal_{J_k} +x^{1/2}_{J_k} s^{1/2}_{J_k}\|
&\le 6n\ell^\delta({\cal J})\sqrt{\mu}\, , \quad \forall k\in [p], \mbox{ and} \label{eq:prox-layer}\\
\label{eq:prox-all}
\|\delta \Delta x^\lal + \delta^{-1} \Delta s^\lal +x^{1/2} s^{1/2}\|
&\le 6n^{3/2}\ell^\delta({\cal J})\sqrt{\mu}\, .
\end{align}
\item\label{i:lls-aff}
For the affine scaling direction $\Delta w^\as=(\Delta x^\as,\Delta
y^\as,\Delta s^\as)$,
\[
\|\Rx^\lal -\Rx^\as\|, \|\Rs^\lal-\Rs^\as\|\le  6n^{3/2}\ell^\delta({\cal J})\, .
\]
\item\label{i:norm}
For the residuals of the LLS steps we have
$\|\Rx^\lal\|,\|\Rs^\lal\|\le \sqrt{2n}$. For each   $i \in [n]$,
$\max\{|\Rx^\lal_i|,|\Rs^\lal_i|\}\ge \frac 12-\frac{3}4 \beta$.
\item\label{i:progress}

We have 
\begin{equation}\label{eq:alpha-max}
\alpha^*\ge 1-\frac{3\sqrt{n}\epsilon^\lal(w)}{\beta}\, ,
\end{equation}
and for any $\alpha\in [0,1]$
\begin{equation*}
\mu(w + \alpha \Delta w^\lal) = (1-\alpha)\mu\,, 
\end{equation*}

\item\label{i:terminate}
We have $\epsilon^\lal(w)=0$ if and only if $\alpha^*=1$. These are further
equivalent to $w+ \Delta w^\lal=(x+\Delta x^\lal, y+\Delta y^\lal,s+ \Delta
s^\lal)$ being an optimal solution to \eqref{LP_primal_dual}.
\end{enumerate}
\end{lemma}

\subsection{The layering procedure}
\label{sec:layering}

Our algorithm performs LLS steps on a layering with a low lifting
score. A further requirement is that
within  each layer, the circuit imbalances $\Le^\delta_{ij}$ defined
in \eqref{eq:Le-def} are suitably bounded. The rescaling here is with respect
to $\delta=\delta(w)$ for the current iterate $w=(x,y,s)$.
To define the precise requirement on the layering, we first introduce
an auxiliary graph.
Throughout we use the parameter
\begin{equation}\label{def:gamma}
\gamma\coloneqq \frac{\beta}{2^{10} n^{5}}\, .
\end{equation}

\paragraph{The auxiliary graph} For a vector $\delta\in \R^n_{++}$ and $\sigma>0$, we define the directed graph
$G_{\delta,\sigma}=([n],E_{\delta,\sigma})$ such that $(i,j)\in
E_{\delta,\sigma}$ if $\Le^\delta_{ij}\ge \sigma$. This is a subgraph
of the {\em circuit ratio digraph} studied in
Section~\ref{sec:rescale}, including only the edges where the circuit
ratio is at least the threshold $\sigma$. Note that we do not
have direct access to this graph, as we cannot efficiently compute the values $\Le^\delta_{ij}$.

At the beginning of the entire algorithm, we run the subroutine
\textsc{Find-Circuits}($A$)  as in
Theorem~\ref{thm:pairwise_circuits_in_matroids},
where $W=\ker(A)$. We assume the matroid
${\cal M}(A)$ is non-separable. For a separable matroid, we can solve the
subproblems of our LP on the components separately. Thus, for each $i\neq j$,
$i,j\in [n]$, we obtain an estimate $\hat \Le_{ij}\le
\Le_{ij}$. These estimates will be gradually improved throughout the algorithm.

Note that $\Le^\delta_{ij}=\Le_{ij}\delta_j/\delta_i$ and $\hat \Le^\delta_{ij}=\hat \Le_{ij}\delta_j/\delta_i$.
If $\hat \Le^\delta_{ij}\ge \sigma$, then we are
guaranteed $(i,j)\in E_{\delta,\sigma}$.

\begin{definition}
\label{def:aux-graph}
Define $\hat
G_{\delta,\sigma}=([n],\hat E_{\delta,\sigma})$ to be the directed graph
with edges $(i,j)$ such that $\hat \Le^\delta_{ij}\ge
\sigma$; clearly, $\hat G_{\delta,\sigma}$ is a subgraph of $G_{\delta,\sigma}$.
\end{definition}

\begin{lemma}\label{lem:tournament}
Let $\delta\in\R^n_{++}$. For every $i\neq j$, $i,j\in[n]$,
$\hat \kappa_{ij}^\delta\cdot\hat \kappa_{ji}^\delta\ge 1$. Consequently, for
any $0<\sigma\le 1$, at least one of $(i,j)\in \hat
E_{\delta,\sigma}$ or $(j,i)\in \hat E_{\delta,\sigma}$.
\end{lemma}
\begin{proof}
\b{We show that this property holds at the initialization. Since the estimates can only increase, it remains true throughout the algorithm.}
Recall the definition of $\hat\Le_{ij}$ from
Theorem~\ref{thm:pairwise_circuits_in_matroids}. This is defined as the maximum of $|g_j/g_i|$ such that $g\in W$, $\supp(g)=C$ for some $C\in
\hat \circuits$ containing $i$ and $j$. For the same vector $g$, we get
$\hat\Le_{ji}\ge |g_i/g_j|$. Consequently, $\hat\Le_{ij}\cdot
\hat\Le_{ji}\ge 1$, and also $\hat\Le^\delta_{ij}\cdot
\hat\Le_{ji}^\delta\ge 1$. The second claim follows by the assumption
$\sigma\le 1$.
\end{proof}

\paragraph{Balanced layerings}
We are ready to define the requirements on the layering in the
algorithm. In the algorithm, $\delta=\delta(w)$ will correspond to the scaling
of the current iterate $w=(x,y,s)$.

\begin{definition}\label{def:balanced}
Let $\delta\in \R^n_{++}$. The layering
${\cal J}=(J_1,
J_2,\ldots, J_p)$ of $[n]$
is \emph{$\delta$-balanced} if
\begin{enumerate}[label=(\roman*)]
\item $\ell^\delta({\cal J})\le \gamma$, and\label{i:bound-norm}
\item  $J_k$ is strongly connected in $G_{\delta,\gamma/n}$ for all $k\in [p]$.\label{i:strong-conn}
\end{enumerate}
\end{definition}

The following lemma shows that within each layer, the
$\Le_{ij}^\delta$ values are within a bounded range. This will play an
important role in our potential analysis.
\begin{lemma}\label{lem:balanced-bound} Let  $0<\sigma< 1$ and $t>0$,
and $i,j\in [n]$, $i\neq j$.
\begin{enumerate}[label=(\roman*)]
\item\label{i-j-upper} If the graph $G_{\delta,\sigma}$
contains a directed path of at most $t-1$ edges from $j$ to $i$,  then
\[
\Le_{ij}^\delta<\left(\frac{\kappa^*}{\sigma}\right)^{t}\, .
\]
\item\label{i-j-lower}
If $G_{\delta,\sigma}$ contains a directed path of at most $t-1$ edges from $i$ to $j$, then
\[
\Le_{ij}^\delta>
\left(\frac{\sigma}{\kappa^*}\right)^{t}\, .
\]
\end{enumerate}
\end{lemma}
\begin{proof}
For part~\ref{i-j-upper},
let $j=i_1,i_2,\ldots,i_h=i$ be a path in $G_{\delta,\sigma}$ in $J$
from $j$ to $i$ with $h\le t$. That is, $\Le^\delta_{i_\ell i_{\ell+1}}\ge \sigma$ for
each $\ell\in [h]$. \Cref{thm:low_bound_chi_bar_with_circuits}
yields
\[
(\bar\kappa^*)^{t}\ge \Le_{ij}^\delta\cdot \sigma^{h-1}>
\Le_{ij}^\delta\cdot \sigma^{t}\, ,
\]
since $h\le t$ and $\sigma< 1$.
Part~\ref{i-j-lower} follows using part \ref{i-j-upper} for $j$ and
$i$, and that $\Le_{ij}^\delta\cdot \Le_{ji}^\delta\ge 1$
according to Lemma~\ref{lem:tournament}.
\end{proof}

\paragraph{Description of the layering subroutine}
Consider
an iterate $w=(x,y,s)\in {\cal N}(\beta)$ of the algorithm with $\delta=\delta(w)$,
The
subroutine \textsc{Layering}$(\delta,\hat\Le)$, described in Algorithm~\ref{alg:layering_procedure}, constructs a
$\delta$-balanced layering. We recall that the approximated auxilliary graph $\hat
G_{\delta,\gamma/n}$ with respect to $\hat{\Le}$ is as in \Cref{def:aux-graph}

\begin{algorithm}[htb]
\caption{\textsc{Layering}($\delta, \hat\Le$)}
\label{alg:layering_procedure}
\SetKwInOut{Input}{Input}
\SetKwInOut{Output}{Output}
\SetKw{And}{\textbf{and}}

\Input{$\delta\in \R^{n}_{++}$ and $\hat\Le\in \R^E_{++}$.}
\Output{$\delta$-balanced layering $\mathcal J =({J}_1,
  \ldots,{J}_p)$ and updated values $\hat\Le\in \R^E_{++}$.}
Compute the strongly connected components $C_1,C_2,\ldots,C_\ell$ of $\hat G_{\delta,\gamma/n}$, listed in
the ordering imposed by $\hat G_{\delta,\gamma/n}$\;
$\bar E\gets \hat E_{\delta,\gamma/n}$\;
\For{$k=2,\ldots,\ell$}{
Call \textsc{Verify-Lift}$(\diag(\delta)W, C_{\ge k},\gamma)$ that answers
{\em `pass'} or {\em `fail'}\;
\If{the answer is {\em `fail'}}{
Let  $i\in C_{\ge k}$, $j\in C_{<k}$, and $t$ be the output of
 \textsc{Verify-Lift} such that
$\gamma/n\le t\le \Le^\delta_{ij}$ \;
$\hat\Le_{ij}\gets t\delta_i/\delta_j$\;
$\bar E\gets \bar E\cup \{(i,j)\}$\;
}
}
Compute strongly connected components $J_1,J_2,\ldots,J_p$ of
$([n],\bar E)$, listed in the ordering imposed by $\hat G_{\delta,\gamma/n}$\;
\Return{${\cal J}=(J_1,J_2,\ldots,J_p), \hat{\Le}$.}
\end{algorithm}

We now give an overview of the subroutine  \textsc{Layering}$(\delta,\hat\Le)$. We start by computing the strongly connected components (SCCs) of the
directed graph $\hat G_{\delta,\gamma/n}$. The edges of this graph are
obtained using the current estimates $\hat\Le_{ij}^\delta$. According to
\Cref{lem:tournament}, we have $(i,j) \in \hat
E_{\delta,\gamma/n}$ or $(j,i)\in \hat
E_{\delta,\gamma/n}$ for every $i,j\in [n]$, $i\neq j$. Hence,
 there is a
linear ordering of the components $C_1,C_2,\ldots,C_\ell$ such that
$(u,v)\in \hat E_{\delta,\gamma/n}$ whenever
$u\in C_i$, $v\in C_j$, and $i<j$. We call this the ordering
imposed by $\hat G_{\delta, \gamma/n}$.

Next, for each $k= 2,\ldots,\ell$, we
use the subroutine \textsc{Verify-Lift}$(\diag(\delta)W, C_{\ge
  k},\gamma)$ described in
\Cref{lem:purify}. If the subroutine returns `pass', then we conclude
$\ell^\delta(C_{\ge k})\le \gamma$, and  proceed to the
next layer.  If the answer is `fail', then the
subroutine returns as certificates $i\in C_{\ge k}$, $j\in C_{<k}$, and $t$ such that
$\gamma/n \le t\le\Le_{ij}^\delta$.
In this case, we update $\hat\Le_{ij}^\delta$ to the higher
value $t$. We add  $(i,j)$ to an edge set $\bar E$;
this edge set was initialized to contain $\hat E_{\delta,\gamma/n}$.
After adding $(i,j)$, all components $C_\ell$ between those containing $i$
and $j$ will be merged into a single strongly connected
component. To see this, recall that if $i'\in C_{\ell}$ and $j'\in
C_{\ell'}$ for $\ell<\ell'$, then $(i',j')\in \hat
E_{\delta,\gamma/n}$ according to \Cref{lem:tournament}.

Finally, we compute the strongly connected components of
$([n],\bar E)$. We let $J_1,J_2,\ldots,J_p$ denote
their unique acyclic order, and return these layers.

\begin{lemma} \label{lem:layering_works_correctly}
The subroutine \textsc{Layering}$(\delta,\hat \Le)$ returns a $\delta$-balanced
layering in $O(nm^2 + n^2)$ time.
\end{lemma}
The difficult part of the proof is showing the running time bound. We note that the weaker bound $O(n^2 m^2)$ can be obtained by a simpler argument.
\begin{proof}
We first verify that the output layering is indeed $\delta$-balanced.
For property \ref{i:bound-norm} of \Cref{def:balanced}, note that each $J_q$ component is the union of some
of the $C_k$'s. In particular, for every $q\in[p]$, the set $J_{\ge
  q}=C_{\ge k}$ for some $k\in [\ell]$. Assume now
$\ell^\delta(C_{\ge k})>\gamma$.  At step $k$ of the main cycle, the
subroutine \textsc{Verify-Lift} returned the answer `fail', and a new edge
$(i,j)\in E$ was added with $i\in C_{\ge k}$, $j\in C_{<k}$. Note that
we already had
$(j,i)\in \hat E_{\delta,\gamma/n}$, since $j\in C_r$ for some $r<k$,
and $i\in C_{r'}$ for $r'\ge k$. This contradicts the choice of
$J_{\ge q}$ as a maximal strongly connected component in $([n],E)$.

Property \ref{i:strong-conn} follows since all new edges added to $E$
have $\Le_{ij}\ge \gamma/n$. Therefore, $([n],E)$ is a subgraph of
$G_{\delta,\gamma/n}$.

\medskip

Let us now turn to the computational cost. The initial
strongly-connected components can be obtained in time $O(n^2)$, and
the same bound holds for the computation of the final
components. (The
latter can be also done in linear time, exploiting the special structure
that the components $C_i$ have a complete linear ordering.)

The second computational bottleneck is the subroutine \textsc{Verify-Lift}. We
assume a matrix $M\in \R^{n \times (n-m)}$ is computed at the very beginning
such that ${\rm range}(M)=W$. We first explain how to implement one call to
\textsc{Verify-Lift} in $O(n (n-m)^2)$ time. We then sketch how to amortize the
work across the different calls to \textsc{Verify-Lift}, using the nested structure
of the layering, to implement the whole procedure in $O(n (n-m)^2)$ time.  To
turn this into $O(n m^2)$, we recall that the layering procedure is the same for
$W$ and $W^\perp$ due to duality (Proposition~\ref{prop:self-dual}).  Since
$\dim(W^\perp)=m$, applying this subroutine on $W^\perp$ instead of $W$ achieves
the same result but in time $O(nm^2)$.

We now explain the implementation of \textsc{Verify-Lift}, where we are given as
input $C \subseteq [n]$ and the basis matrix $M \in \R^{n \times (n-m)}$ as
above with ${\rm range}(M) = W$. Clearly, the running time is dominated by the
computation of the set $I \subseteq C$ and the matrix $B \in \R^{([n] \setminus C)
\times |I|}$ satisfying $L_C^W(x)_{[n] \setminus C} = B x_{I}$, for $x \in
\pi_C(W)$. We explain how to compute $I$ and $B$ from $M$ using column
operations (note that these preserve the range). The valid choices for $I
\subseteq C$ are in correspondence with maximal sets of linear independent rows
of $M_{C,\sbullet}$, noting then that $|I| = r$ where $r \coloneqq  \rk(M_{C,\sbullet})$. Let
$D_1 = [n-m-r]$ and $D_2 = [n-m] \setminus [n-m-r]$. By applying columns
operations to $M$, we can compute $I \subseteq C$ such that $M_{I,D_2} = \mathbf I_{r}$
($r \times r$ identity) and $M_{C,D_1} = 0$. This requires $O(n(n-m)|C|)$ time
using Gaussian elimination. At this point, note that $\pi_C(W) = {\rm
range}(M_{C,D_2})$, $\pi_{I}(W) = \R^{I}$ and ${\rm range}(M_{\sbullet, D_1}) = W
\cap \R^n_{[n] \setminus C}$. To compute $B$, we must transform the columns of
$M_{\sbullet,D_2}$ into minimum norm lifts of $e_i \in \pi_{I}(W)$ into $W$, for
all $i \in I$. For this purpose, it suffices to make the columns of $M_{[n] \setminus
C,D_2}$ orthogonal to the range of $M_{[n] \setminus C,D_1}$. Applying
Gram-Schmidt orthogonalization, this requires $O((n-|C|)(n-m)(n-m-r))$ time.
From here, the desired matrix $B = M_{[n] \setminus C, D_2}$. Thus, the total
running time of \textsc{Verify-Lift} is $O(n(n-m)|C| + (n-|C|)(n-m)(n-m-r)) =
O(n(n-m)^2)$.

We now sketch how to amortize the work of all the calls of \textsc{Verify-Lift}
during the layering algorithm, to achieve a total $O(n(n-m)^2)$ running time.
Let $C_1,\dots,C_\ell$ denote the candidate SCC layering. Our task is to compute
the matrices $B_k$, $2 \leq k \leq \ell$, needed in the calls to
\textsc{Verify-Lift} on $W, C_{\geq k}$, $2 \leq k\leq \ell$, in total
$O(n(n-m)^2)$ time. We achieve this in three steps working with the basis matrix $M$
as above. Firstly, by applying column operations to $M$, we compute sets $I_k
\subseteq C_k$ and $D_k = [|I_{\leq k}|] \setminus [|I_{< k}|]$, $k \in [\ell]$,
such that $M_{I_k,D_k} = \mathbf I_{r_k}$, where $r_k = |I_k|$, and $M_{C_{\geq
k},D_{<k}} = 0$, $2 \leq k \leq \ell$. Note that this enforces $\sum_{k=1}^\ell r_k =
(n-m)$. This computation requires $O(n(n-m)^2)$ time using Gaussian elimination.
This computation achieves ${\rm range}(M_{C_k,D_k}) = \pi_{C_k}(W \cap
\R^n_{C_{\leq k}})$, ${\rm range}(M_{C_{\geq k},D_{\geq k}}) = \pi_{C_{\geq
k}}(W)$ and ${\rm range}(M_{\sbullet,D_{\leq k}}) = W \cap \R^n_{C_{\leq k}}$, for
all $k \in [\ell]$.

From here, we block orthogonalize $M$, such that the columns of $M_{\sbullet, D_k}$
are orthogonal to the range of $M_{\sbullet,D_{<k}}$, $2 \leq k \leq \ell$. Applying
an appropriately adapted Gram-Schmidt orthogonalization, this requires
$O(n(n-m)^2)$ time. Note that this operation maintains $M_{I_k,D_k} = \mathbf I_{r_k}$,
$k \in [\ell]$, since $M_{C_{\geq k},D_{<k}} = 0$. At this point, for $k \in [\ell]$
the columns of $M_{\sbullet,D_k}$ are in correspondence with minimum norm lifts of
$e_i \in \pi_{D_{\geq k}(W)}$ into $W$, for all $i \in I_k$. Note that to
compute the matrix $B_k$ we need the lifts of $e_i \in \pi_{D_{\geq k}(W)}$, for
all $i \in I_{\geq k}$ instead of just $i \in I_k$.

We now compute the matrices $B_\ell,\dots,B_2$ in this order via the following
iterative procedure. Let $k$ denote the iteration counter, which decrements from
$\ell$ to $2$. For $k=\ell$ (first iteration), we let $B_\ell= M_{C_{<\ell},D_\ell}$ and
decrement $k$. For $k < \ell$, we eliminate the entries of $M_{I_k,D_{>k}}$ by
using the columns of $M_{\sbullet,D_k}$. We then let $B_k = M_{C_{<k},D_{\geq k}}$
and decrement $k$. To justify correctness, one only has to notice that at the end
of iteration $k$, we maintain the orthogonality of $M_{\sbullet,D_{\geq k}}$ to
the range of $M_{\sbullet,D_{< k}}$ and that $M_{I_{\geq k},D_{\geq k}} = \mathbf I_{|I_{\geq
k}|}$ is the appropriate identity. The cost of this procedure is the same as a
full run of Gaussian elimination and thus is bounded by $O(n(n-m)^2)$. The
calls to \textsc{Verify-Lift} during the layering procedure can thus be executed in
$O(n(n-m)^2))$ amortized time as claimed.
\end{proof}

\subsection{The overall algorithm}
\label{sec: the_algorithm}

\begin{algorithm}[htb]
\caption{\textsc{LP-Solve}($A,b,c,w^0$)}
\label{alg:overall}
\SetKwInOut{Input}{Input}
\SetKwInOut{Output}{Output}
\SetKw{And}{\textbf{and}}

\Input{$A\in \R^{m\times n}$, $b\in \R^m$, $c\in \R^n$, and an initial
feasible solution $w^0=(x^0,y^0,s^0)\in {\cal N}(1/8)$ to \eqref{LP_primal_dual}.
}
\Output{Optimal solution $w^*=(x^*,y^*,s^*)$ to
  \eqref{LP_primal_dual}.} Call \textsc{Find-Circuits}$(A)$ to obtain the lower bounds
$\hat\Le_{ij}$ for each $i,j\in [n]$, $i\neq j$\;
$k\gets 0, \alpha \gets 0$\;
\Repeat{$\mu(w^k)=0$}{
\tcc{Predictor step}
Compute  affine scaling direction $\Delta w^\as=(\Delta x^\as,\Delta
y^\as,\Delta s^\as)$ for $w$\;
\If(\tcp*[f]{Recall $\epsilon^\as(w)$ defined in \eqref{def:epsilon}}){$\epsilon^\as(w)<10n^{3/2}\gamma$}
{$\delta\gets (s^k)^{1/2}(x^k)^{-1/2}$\;
$({\cal J}, \hat{\Le}) \gets $\textsc{Layering}($\delta$, $\hat{\Le}$)\;
Compute Layered Least Squares direction $\Delta w^\lal=(\Delta x^\lal,\Delta
y^\lal,\Delta s^\lal)$ for the layering $\cal J$ and $w$\;
$\Delta w\gets \Delta w^\lal$\;
$\alpha\gets1-24\sqrt{n}\epsilon^\lal(w)$\tcp*{As in \Cref{lem:ll-decompose}\ref{i:stepsize}}
}
\Else{$\Delta w\gets \Delta w^\as$\;
$\alpha\gets \min \left\{1/(8\sqrt n),1-{8\|\Delta
    x^\as\Delta s^\as\|}/{\mu(w)}\right\}$\tcp*{As in \Cref{prop:predictor-corrector}\ref{i:stepsize}}}
$w'\gets w^k+\alpha \Delta w$\;
\tcc{Corrector step}
Compute centrality direction $\Delta w^\cs=(\Delta x^\cs,\Delta
y^\cs,\Delta s^\cs)$ for $w'$\;
$w^{k+1}\gets w'+ \Delta w^\cs$\;
$k\gets k+1$\;
}
\Return{$w^{k}=(x^k,y^k,s^k)$.}
\end{algorithm}

Algorithm~\ref{alg:overall} presents the overall algorithm
\textsc{LP-Solve}$(A,b,c,w^0)$.
We assume that an initial feasible solution $w^0=(x^0,y^0,s^0)\in
{\cal N}(\beta)$ is given. We address this in
Section~\ref{sec:initialization}, by adapting the extended system used
in \cite{Vavasis1996}. We note that this subroutine requires an upper bound
on $\bar \chi^*$. Since computing $\bar\chi^*$ is hard, we can
implement it by a doubling search on $\log \bar\chi^*$, as explained
in Section~\ref{sec:initialization}. Other than for initialization,
the algorithm does not require an estimate on $\bar\chi^*$.

The algorithm starts with the subroutine \textsc{Find-Circuits}$(A)$
as in Theorem~\ref{thm:pairwise_circuits_in_matroids}.  The iterations are
similar to the MTY Predictor-Corrector  algorithm  \cite{MTY}. The
main difference is that certain affine scaling steps are replaced by
LLS steps. In every predictor step, we compute the affine scaling
direction, and consider the quantity $\epsilon^\as(w)=\max_{i\in
  [n]}\min\{|\Rx^\as_i|,|\Rs^\as_i|\}$. If this is above the threshold
$10n^{3/2}\gamma$, then we perform the affine scaling step. However, in
case $\epsilon^\as(w)<10n^{3/2}\gamma$, we use the LLS direction instead.
In each such iteration, we call the subroutine
\textsc{Layering}($\delta,\hat\Le$) (Algorithm~\ref{alg:layering_procedure}) to compute
the layers, and we compute the LLS step for this layering.

Another important difference is that the algorithm does not require a
final rounding step. It terminates with the exact optimal solution
$w^*$ once a predictor step is able to perform a full step with
$\alpha=1$.

\begin{theorem}\label{thm:overall-running}
For given $A\in \R^{m\times n}$, $b\in \R^m$, $c\in \R^n$, and an initial
feasible solution $w^0=(x^0,y^0,s^0)\in {\cal N}(1/8)$,
Algorithm~\ref{alg:overall} finds an optimal solution
to \eqref{LP_primal_dual}
in $O(n^{2.5}\log n \log( \bar\chi^*_A+n))$
iterations.
\end{theorem}

\begin{remark}\label{rem:amortized}\em
Whereas using LLS steps enables us to give a strong bound on the total
number of iterations, finding LLS directions has a significant computational
overhead as compared to finding affine scaling directions.
The layering $\cal J$ can be computed in time $O(nm^2)$
(\Cref{lem:layering_works_correctly}), and the LLS steps also require $O(nm^2)$ time,
see \cite{Vavasis1996,MMT98}. This is in contrast to
the computational cost $O(n^\omega)$ of an affine scaling direction.
Here $\omega<2.373$ is the matrix multiplication constant \cite{Vassilevska2012}.

We now sketch a possible approach to amortize the computational cost
of the LLS steps over the sequence of affine scaling steps. It was
shown in \cite{MonteiroT05} that for the MTY P-C algorithm, the
``bad'' scenario between two crossover events amounts to a series of
affine scaling steps where the progress in $\mu$ increases
exponentially from every iteration to the next. This corresponds  to the term $O(\min\{n^2 \log \log
(\mu_0/\eta), \log (\mu_0/\eta)\})$ in their running time analysis. Roughly
speaking, such a sequence of affine scaling steps indicates that an LLS
step is necessary.

Hence, we could observe these accelerating
sequences of affine scaling steps, and perform an LLS step after we
see a sequence of length $O(\log n)$. The progress made by these
affine scaling steps offsets the cost of computing the LLS direction.
\end{remark}

\section{The potential function and the overall analysis}\label{sec:analysis}
Let $\mu>0$ and $\delta(\mu)=s(\mu)^{1/2}x(\mu)^{-1/2}=\sqrt{\mu}/x(\mu)=s(\mu)/\sqrt{\mu}$ correspond to
the point on the central path and recall the definition of $\gamma$ in \eqref{def:gamma}. For $i,j\in [n]$, $i\neq j$, we define
\[
\rho^\mu(i,j):=\frac{\log \Le_{ij}^{\delta(\mu)}}{ \log
\left(4n\Le^*_W/\gamma\right)}\, ,
\]
and the main potentials in the algorithm as
\[
\Psi^\mu(i,j):=\max\left\{1,\min\left\{2n,\inf_{0<\mu'<\mu}\rho^{\mu'}(i,j)\right\}\right\}
\quad \mbox{and} \quad  \Psi(\mu):=\sum_{i,j\in [n], i\neq j}\log
\Psi^\mu(i,j)\, .
\]
The motivation for $\rho^\mu(i,j)$ and $\Psi^\mu(i,j)$ comes from
\Cref{lem:balanced-bound}, using $\sigma=\gamma/(4n)$. Thus, 
$\log \Le_{ij}^{\delta(\mu)}/ \log \left(4n\Le^*_W/\gamma\right)$ can be seen as a lower bound on the length of the shortest $j$--$i$ path.
Recall that the layers are defined as strongly connected components of $\hat G_{\delta,\gamma/n}$, which is a subgraph of $G_{\delta(\mu),\gamma/(4n)}$ (using the bound \eqref{eq:delta-beta}). Consequently,  whenever $\rho^\mu(i,j)\ge n$, the nodes $i$ and $j$ cannot be in the same strongly connected component for the normalized duality gap $\mu$.
Thus, our potentials $\Psi^\mu(i,j)$ can be seen as fine-grained
analogues of the crossover events analyzed in
\cite{Vavasis1996,Monteiro2003,MonteiroT05}: the definition of $\Psi^\mu(i,j)$ contains a minimization over $0<\mu'<\mu$; therefore, $\Psi^\mu(i,j)> n$ implies that $i$ and $j$ may never appear on the same layer for any $\mu'\le \mu$.
On the other hand, these potentials are more fine-grained: even for $t < n$, if $\Psi^\mu(i,j)\ge t$ then whenever a layer contains both $i$ and $j$ for $\mu'\le \mu$, this layer must have size $\ge t$.

By definition, for all pairs $(i,j) \in [n] \times [n]$ we have $\Psi^{\mu'}(i,j)\ge \Psi^{\mu}(i,j)$ for $0<\mu'\le \mu$; and we enforce $\Psi^{\mu}(i,j)\in [1,2n]$. The upper bound can be imposed since values $\Psi^{\mu'}(i,j)\ge n$ do not yield any new information on the layering.  Hence, the overall potential $\Psi(\mu)$ is between 0 and $O(n^2\log n)$. The overall analysis in the proof of Theorem~\ref{thm:overall-running} divides the iterations into phases. In each phase, we can identify a set $J\subseteq [n]$, $|J|>1$ arising as a layer or as the union of two layers in the LLS step at the beginning of the phase.
We show that  $\Psi^{\mu}(i,j)$ doubles for at least $|J|-1$ pairs $(i,j) \in J \times J$ during the  subsequent $O(\sqrt{n}|J|\log(\bar\chi^*+n))$ iterations; consequently, $\Psi(\mu)$ increases by at least $|J|-1$ during these iterations. This leads to the overall iteration bound $O(n^{2.5}\log (n)\log(\bar\chi^*+n))$. In comparison, the crossover analysis would correspond to showing that within $O(n^{1.5}\log(\bar\chi^*+n))$ iterations, one of the $\Psi^{\mu}(i,j)$ values previously $<n$ becomes larger than $n$.
The following statement formalizes the above mentioned properties of $\Psi^{\mu}(i,j)$.

\begin{lemma}\label{lem:rho-bound}
Let $w=(x,y,s)\in {\cal N}(\beta)$ for $\beta\in (0,1/4]$. Let $i,j\in [n]$, $i\neq j$, and let $\mu=\mu(w)$.\begin{enumerate}
\item If $\hat G_{\delta,\gamma/n}$ contains a path from $j$
to $i$ of at most $t-1$ edges, then $\rho^\mu(i,j)<t$.
\item If $\hat G_{\delta,\gamma/n}$ contains a path from $i$ to $j$ of at most $t-1$ edges,
then $\rho^\mu(i,j) > -t$.
\item
If $\Psi^\mu(i,j)\ge t$, then in any $\delta(w')$-balanced
layering, where $w'=(x',y',s')\in {\cal N}(\beta)$ with $\mu(w') \leq \mu$,
\begin{itemize}
\item $i$ and $j$ cannot be
together on a layer of size at most $t$, and 
\item $j$ cannot be on a layer preceding the layer containing $i$.
\end{itemize}
\end{enumerate}
\end{lemma}
\begin{proof}
From  \eqref{eq:delta-beta}, we see that for any $i,j$, 
\[\hat\Le^\delta_{ij}\le \Le^\delta_{ij}\le (1-2\beta)^{-1}\Le^{\delta(\mu)}_{ij}\le 4\Le^{\delta(\mu)}_{ij}\, .
\]
Consequently, $\hat G_{\delta,\gamma/n}$ is a subgraph of $G_{\delta(\mu),\gamma/(4n)}$. The statement now follows from \Cref{lem:balanced-bound} with $\sigma=\gamma/(4n)$.
\end{proof}
In what follows, we formulate four important lemmas crucial for the proof of
\Cref{thm:overall-running}.
For the lemmas, we only highlight some key ideas here, and defer the full proofs to Section~\ref{sec:main-lemmas}.

For a triple $w\in {\cal N}(\beta)$, $\Delta w^\lal$ refers to the
LLS direction found in the algorithm, and $\Rx^\lal$ and $\Rs^\lal$
denote the residuals as in \eqref{eq:residuals}. For a subset $I \subset [n]$ recall the definition
\begin{align*}
\epsilon_I^\lal(w) := \max_{i \in I} \min \{|\Rx_i^\lal|,
  |\Rs_i^\lal|\}\, .
\end{align*}
We introduce another important quantity $\xi$ for the analysis:
\begin{equation}
\label{eq:xi}
\xi^\lal_I(w):= \min \{\|\Rx^\lal_I\|, \|\Rs^\lal_I\|\}\,
\end{equation}
for a subset $I \subset [n]$.
For a layering ${\cal J}=(J_1,J_2,\ldots,J_p)$, we let
\[
\xi^\lal_{\cal J}(w)=\max_{k \in [p]} \xi^\lal_{J_k}(w)\, .
\]
The key idea of the analysis is to extract information about the
optimal solution $w^*=(x^*,y^*,s^*)$ from the LLS direction.
The first main lemma shows that if $\|\Rx^\lal_{J_q}\|$ is large
on some
layer $J_q$, then for at least one index $i\in J_q$,
$x^*_i/x_i\ge 1/\mathrm{poly}(n)$, \b{i.e., the variable $x_i$ has ``converged''}.  The analogous statement holds on the dual side for $\|\Rs^\lal_{J_q}\|$ and an index $j \in J_q$.

\begin{lemma}[name=\prooftext\pageref{proof:lem: lower_bounds_for_crossover_events}, {restate=[name=\restatetext]lemlowerboundsforcrossoverevents}]\label{lem: lower_bounds_for_crossover_events}
Let $w = (x,y,s) \in \mathcal N(\beta)$ for $\beta\in (0,1/8]$ and let $w^* = (x^*, y^*, s^*)$ be the optimal solution corresponding to $\mu^* = 0$ on the central path. Let further ${\cal J}=(J_1, \ldots,  J_p)$ be a $\delta(w)$-balanced layering (\Cref{def:balanced}),
and let $\Delta w^\lal=(\Delta x^\lal, \Delta y^\lal, \Delta s^\lal)$ be the
corresponding LLS direction.
Then the following statement holds for every $q \in [p]$:
\begin{enumerate}[label=(\roman*)]
\item \label{case:lower_x} There exists $i \in  J_q$ such that
\begin{align}
x_i^* \geq\frac{2x_i}{3\sqrt{n}}\cdot(\|\Rx_{J_q}^\lal\| - 2\gamma n)\, . \label{eq: lower_bound_x}
\end{align}
\item \label{case:lower_s}  There exists $j \in  J_q$ such that
\begin{align}
{s_j^*}\geq \frac{2s_j}{3\sqrt{n}} \cdot (\|\Rs_{J_q}^\lal\| - 2\gamma
  n)\, . \label{eq: lower_bound_s}
\end{align}
\end{enumerate}
\end{lemma}
We outline the main idea of the proof of part \ref{case:lower_x}; part
\ref{case:lower_s} follows analogously using the duality of the
lifting scores
(\Cref{lem:balanced-dual}). On layer $q$, the LLS step minimizes
$\|\delta_{J_q}(x_{J_q}+\Delta x_{J_q})\|$, subject to $\Delta
x_{J_{>q}}=\Delta x_{J_{>q}}^\lal$ and subject to existence of $\Delta x_{J_{<q}}$ such that $\Delta x \in W$. By  making use of
$\ell^{\delta(w)}(J_{>q})\le\gamma$ due to $\delta(w)$-balancedness, we can show the existence of a point $z\in W+x^*$
such that $\|\delta_{J_q}(z_{J_q}-x^*_{J_q})\|$ is small, and
$z_{J_{>q}}=x_{J_{>q}}+\Delta x^\lal_{J_{>q}}$. By the choice of $\Delta
x^\lal_{J_q}$, we have $\|\delta_{J_q} z_{J_q}\|\geq\|\delta_{J_q}(x_{J_q}+\Delta x_{J_q}^\lal)\|=\sqrt{\mu}\|\Rx^\lal_{J_q}\|$. Therefore,
$\|\delta_{J_q}x^*_{J_q}/\sqrt{\mu}\|$ cannot be much smaller than
$\|\Rx^\lal_{J_q}\|$. Noting that  $\delta_{J_q}x^*_{J_q}/\sqrt{\mu} \approx
x^*_{J_q}/x_{J_q}$, we obtain a lower bound on
$x_i^*/x_i$ for some  $i\in J_q$.

We emphasize that the lemma only shows the existence of such indices $i$
and $j$, but does not provide an efficient algorithm to identify them. It is also
useful to note that for any $i \in [n]$,
$\max\{|\Rx^\lal_i|,|\Rs^\lal_i|\}\ge \frac 12-\frac{3}{4}\beta$ according to
\Cref{lem:ll-decompose}\ref{i:norm}. Thus, for each $q\in [p]$,
we obtain a strong and positive lower bound in either case
\ref{case:lower_x} on $x_i/x_i^*$ or case \ref{case:lower_s} on $s_i/s_i^*$ for some $i \in J_q$.

\medskip

The next lemma allows us to argue that the potential function $\Psi^{\cdot}(\cdot,\cdot)$ increases
for multiple pairs of variables, if we have
strong lower bounds on both $x_i^*$ and $s_j^*$ for some $i,j\in [n]$, along
with a lower and upper bound on $\rho^\mu(i,j)$.

\begin{lemma}[name=\prooftext\pageref{proof:lem:potential-master}, {restate=[name=\restatetext]lempotentialmaster}]\label{lem:potential-master}
Let $w=(x,y,s)\in {\cal N}(2\beta)$ for $\beta\in (0,1/8]$, let
$\mu=\mu(w)$ and $\delta=\delta(w)$.
Let $i,j\in [n]$ and $2 \le \tau \le n$
such that for the optimal solution $w^*=(x^*,y^*,s^*)$, we have
$x_i^*\ge \beta x_i/(2^{10}n^{5.5})$ and $s_j^*\ge\beta s_j/(2^{10}n^{5.5})$, and assume $\rho^\mu(i,j)\ge -\tau$.
After $O(\beta^{-1}\sqrt{n}\tau\log(\bar\chi^*+n))$ further iterations the duality gap $\mu'$ fulfills
$\Psi^{\mu'}(i,j)\ge 2\tau$, and for every $\ell\in [n]\setminus\{i,j\}$,
either $\Psi^{\mu'}(i,\ell)\ge 2\tau$, or
$\Psi^{\mu'}(\ell,j)\ge 2\tau$.
\end{lemma}
We note that $i$ and $j$ as in the lemma are necessarily
different, since $i=j$ would imply $0=x_i^*  s^*_i\ge \beta^2 \mu/(2^{20} n^{11}) > 0$.

Let us illustrate the idea of the proof of $\Psi^{\mu'}(i,j)\ge
2\tau$. For $i$ and $j$ as in the lemma, and for a central path element $w'=w(\mu')$ for
$\mu'<\mu$, we have $x'_i\ge x_i^*/n\ge \beta x_i/(2^{10}n^{6.5})$ and $s'_j\ge
s_j^*/n\ge\beta  s_j/(2^{10}n^{6.5})$ by the near-monotonicity of the
central path (\Cref{lem: central_path_bounded_l1_norm}).
Note that
\[
\Le_{ij}^{\delta'}=\Le_{ij}\cdot\frac{\delta'_j}{\delta'_i}=\Le_{ij}\cdot\frac{x'_is'_j}{\mu'}\ge
\Le_{ij}\cdot\frac{\beta^2 x_is_j}{2^{20}n^{13}\mu'}\ge\frac{\beta^2(1-\beta)^2}{ 2^{20} n^{13}}\cdot \Le_{ij}^\delta \cdot \frac{\mu}{\mu'}\, ,
\]
where the last inequality uses \Cref{prop:x_i-s_i}.
Consequently, as $\mu'$ sufficiently decreases,   $\Le_{ij}^{\delta'}$
will become much larger than  $\Le_{ij}^\delta$.
The claim on $\ell\in [n]\setminus\{i,j\}$ can be
shown by using the triangle inequality $\Le_{ik}\cdot\Le_{kj}\ge \Le_{ij}$ shown in \Cref{lem:imbalance_triangle_inequality}.

\medskip
Assume now $\xi^\lal_{J_q}(w)\ge 4\gamma n$ for some $q\in [p]$ in the LLS step. Then,
\Cref{lem: lower_bounds_for_crossover_events} guarantees the
existence of $i,j\in J_q$ such that $x_i^*/x_i, s_j^*/s_j\ge \frac{4}{3\sqrt{n}}\gamma n
>\beta/(2^{10}n^{5.5})$. Further, \Cref{lem:rho-bound} gives
$\rho^\mu(i,j)\ge -|J_q|$.
Hence,
\Cref{lem:potential-master} is applicable for $i$ and $j$ with $\tau=|J_q|$.

\medskip

The overall potential argument in the proof of
\Cref{thm:overall-running} uses \Cref{lem:potential-master} in three
cases: $\xi^\lal_{{\cal J}}(w)\ge
4\gamma n$ (\Cref{lem: lower_bounds_for_crossover_events} is
applicable as above); $\xi^\lal_{{\cal J}}(w)<
4\gamma n$ and $\ell^{\delta^+}({\cal J})\le 4\gamma n$
(\Cref{lem: case_lemma_does_not_crash}); and $\xi^\lal_{{\cal J}}(w)<
4\gamma n$ and $\ell^{\delta^+}({\cal J})> 4\gamma n$
(\Cref{lem: case_layer_crashes}). Here, $\delta^+$ refers to the
value of $\delta$ after the LLS step. Note that $\delta^+ > 0$ is well-defined, unless the algorithm terminated with an optimal solution.

To prove these lemmas, we need to study how the layers ``move'' during the
LLS step. We let
$\llspartition B = \{t \in [n] : |\Rs_t^\lal| < 4\gamma n\}$ and $\llspartition N=\{t \in [n] :
|\Rx_t^\lal| < 4\gamma n\}$.
The assumption  $\xi_{{\cal J}}^\lal(w) < 4\gamma n$ means that for
each layer $J_k$, either $J_k\subseteq \llspartition B$ or $J_k\subseteq \llspartition N$;
we  accordingly refer to $\llspartition B$-layers and $\llspartition N$-layers.

\begin{lemma}[name=\prooftext\pageref{proof:lem: case_lemma_does_not_crash}, {restate=[name=\restatetext]lemcaselemmadoesnotcrash}]\label{lem: case_lemma_does_not_crash}
Let $w = (x,y,s) \in \mathcal N(\beta)$ for $\beta\in (0,1/8]$, and
let ${\cal J}=(J_1, \ldots,  J_p)$ be a $\delta(w)$-balanced
partition. Assume that $\xi_{{\cal J}}^\lal(w) < 4\gamma n$, and let $w^+ = (x^+, y^+, s^+)\in \overline{\cal N}(2\beta)$ be the
next iterate obtained by the LLS step with $\mu^+=\mu(w^+)$ and assume $\mu^+ > 0$.
Let $q\in[p]$ such that
$\xi_{{\cal J}}^\lal(w)=\xi_{J_q}^\lal(w)$. If
$\ell^{\delta^+}(\mathcal J) \leq 4\gamma n$, then there exist $i,j\in
J_q$ such that $x_i^*\ge \beta x_i^+/(16n^{3/2})$ and $s_j^*\ge
\beta s_j^+/(16n^{3/2})$. Further, for any $\ell,\ell'\in J_q$, we have  $\rho^{\mu^+}(\ell,\ell')\ge -|J_q|$.
\end{lemma}
For the proof sketch,
without loss of generality, let $\xi_{\cal
  J}^\lal=\xi_{J_q}^\lal=\|\Rx_{J_q}^\lal\|$, that is, $J_q$ is an $\llspartition N$-layer. The case
$\xi_{J_q}^\lal=\|\Rs_{J_q}^\lal\|$ can be treated analogously.
Since the residuals $\|\Rx_{J_q}^\lal\|$ and $\|\Rs_{J_q}^\lal\|$ cannot
be both small,
\Cref{lem: lower_bounds_for_crossover_events} readily provides a $j\in J_q$
such that $s_j^*/s_j\ge 1/(6\sqrt{n})$. Using \Cref{lem:
  central_path_bounded_l1_norm} and \Cref{prop:near-central}, $s_j^*/s_j^+ = s_j^*/s_j \cdot s_j/s_j^+ > (1-\beta)/(6(1+4\beta)n^{3/2})>\beta/(16n^{3/2})$.

The key ideas of
showing the existence of an $i\in
J_q$ such that $x_i^*\ge x_i^+/(16n^{3/2})$ are the following. With  $\approx$, $\lessapprox$ and $\gtrapprox$, we write equalities and inequalities that hold up to small polynomial factors. First, we show that {\em
  (i)} $\|\delta_{J_q}x^+_{J_q}\|\lessapprox \mu^+/ \sqrt{\mu}$, and then, that {\em (ii)}
$\|\delta_{J_q} x^*_{J_q}\| \gtrapprox \mu^+/\sqrt{\mu}\, .$

If we can show {\em (i)} and {\em (ii)} as above, we obtain that $\|\delta_{J_q}x^*_{J_q}\|\gtrapprox
\|\delta_{J_q}x^+_{J_q}\|$, and thus, $x_i^*\gtrapprox x_i^+$ for some $i\in J_q$.

Let us now sketch the first step.
By the assumption $J_q \subset \llspartition N$, one can show $x_{J_q}^+/x_{J_q} \approx \mu^+/\mu$, and therefore
\[\|\delta_{J_q}x^+_{J_q}\| \approx \frac{\mu^+}{\mu} \|\delta_{J_q}x_{J_q}\| \approx \frac{\mu^+}{\mu} \sqrt{\mu} = \frac{\mu^+}{\sqrt{\mu}}\, .
\]
The second part of the proof, namely, lower bounding
$\|\delta_{J_q}x^*_{J_q}\|$, is more difficult.  Here, we only sketch it for the special case when $J_q=[n]$. That is, we have a single layer only; in particular, the LLS step is the same as the affine scaling step $\Delta x^\lal=\Delta x^\as$. The general case of multiple layers follows by making use of \Cref{lem:ll-decompose}, i.e.\ exploting that for a sufficiently small $\ell^\delta({\cal J})$, the LLS step is close to the affine scaling step.

Hence, assume that $\Delta x^\lal=\Delta x^\as$. Using the equivalent definition of the affine scaling step \eqref{eq:aff-minnorm} as a minimum-norm point, we have  $\|\delta x^*\|\ge \|\delta(x+\Delta x^\lal)\|=\sqrt{\mu}\|\Rx^\lal\|=\sqrt{\mu}\xi_{\cal J}^\lal$.
From \Cref{lem:affscale-progress},
$\mu^+/\mu\le 2\sqrt{n}\epsilon^\as(w)/\beta\le 2\sqrt{n}\xi_{\cal J}^\lal/\beta$. Thus, we see that $\|\delta x^*\|\ge \beta\mu^+/(2\sqrt{n\mu})$.

The final statement on lower bounding $\rho^{\mu^+}(\ell,\ell')\ge -|J_q|$ for
any $\ell,\ell'\in J_q$ follows by showing that
$\delta^+_\ell/\delta^+_{\ell'}$ remains close to $\delta_\ell/\delta_{\ell'}$, and hence the values of
$\Le^{\mu^+}(\ell,\ell')$ and $\Le^\mu(\ell,\ell')$ are sufficiently
close for indices on the same layer (\Cref{lem:near_uniform_shootdown}).

\begin{lemma}[name=\prooftext\pageref{proof:lem: case_layer_crashes}, {restate=[name=\restatetext]lemcaselayercrashes}]\label{lem: case_layer_crashes}
Let $w = (x,y,s) \in \mathcal N(\beta)$ for $\beta\in (0,1/8]$, and
let ${\cal J}=(J_1, \ldots,  J_p)$ be a $\delta(w)$-balanced
partition. Assume that $\xi_{{\cal J}}^\lal(w) < 4\gamma n$, and let $w^+ = (x^+, y^+, s^+)\in \overline{\cal N}(2\beta)$ be the
next iterate obtained by the LLS step with $\mu^+=\mu(w^+)$ and assume $\mu^+ > 0$.
If $\ell^{\delta^+}(\mathcal J) > 4\gamma n$, then
there exist two layers $J_q$ and $J_r$
 and $i\in J_q$ and $j\in J_r$ such that $x_i^*\ge x^+_i/(8n^{3/2})$, and $s_j^*\ge s^+_j/(8n^{3/2})$.
Further, $\rho^{\mu^+}(i,j)\ge -|J_q\cup J_r|$, and for all
$\ell,\ell'\in J_q\cup J_r$, $\ell \neq \ell'$ we have $\Psi^\mu(\ell,\ell')\le |J_q\cup J_r|$.
\end{lemma}
Consider now any $\ell \in J_k\subseteq \llspartition B$. Then, since $\Rx_\ell^\lal$
is multiplicatively close to 1, $x_\ell^+\approx x_\ell$; on the other hand
$s_\ell^+$ will ``shoot down'' close to the small value
$\Rs_\ell^\lal \cdot s_\ell$. Conversely, for $\ell\in J_k\subseteq \llspartition N$,
$s_\ell^+\approx s_\ell$, and $x_\ell^+$ will ``shoot down'' to a small
value.

The key step of the analysis is showing that the increase in
$\ell^{\delta^+}(\mathcal J)$ can be attributed to an $\llspartition N$-layer $J_r$ ``crashing
into'' a $\llspartition B$-layer $J_q$. That is, we show the existence of an edge $(i',j')\in E_{\delta^+,\gamma/(4n)}$ for $i'\in
J_q$ and $j'\in J_r$, where $r<q$ and $J_q\subseteq \llspartition B$,  $J_r\subseteq
\llspartition N$. This can be achieved by analyzing the  matrix $B$ used in the
subroutine \textsc{Verify-Lift}.

For the layers $J_q$ and $J_r$, we can use Lemma~\ref{lem:
  lower_bounds_for_crossover_events} to show that
there exists an $i\in J_q$ where $x_i^*/x_i$ is lower bounded, and there exists a $j\in J_r$ where $s_j^*/s_j$ is lower
bounded. The lower bound on $\rho^{\mu^+}(i,j)$ and
the upper bounds on the $\Psi^\mu(\ell,\ell')$ values can be shown
by tracking the changes between the $\Le^\delta(\ell,\ell')$ and
$\Le^{\delta^+}(\ell,\ell')$
values, and applying  \Cref{lem:rho-bound} both at $w$ and at $w^+$.

\begin{proof}[Proof of \Cref{thm:overall-running}]
We analyze the overall potential function $\Psi(\mu)$. By the {\em iteration at
  $\mu$} we mean the iteration where the normalized duality gap of the
current iterate is $\mu$.

By \Cref{prop:predictor-corrector}\ref{i:stepsize} and \Cref{lem:ll-decompose}\ref{i:stepsize}, the predictor step gives $w'\in\overline{\cal N}(1/4)$ in every iteration, and thus by \Cref{prop:predictor-corrector}\ref{i:corrector}, if $\mu(w') > 0$, the iterate $w^{\cs}$ after a corrector step fulfills $w^{\cs} \in{\cal N}(1/8)$.
If $\mu^+ = 0$ at the end of an iteration, the algorithm terminates with an optimal solution. Recall from Lemma~\ref{lem:ll-decompose}\ref{i:terminate} that this happens if and only if $\epsilon^\lal(w)=0$  at a certain iteration.

From now on, assume that $\mu^+ > 0$.
We distinguish three cases at each iteration. These cases are
well-defined even at iterations where affine scaling steps are used. At such
iterations, $\xi^\lal_{{\cal J}}(w)$ still refers to the LLS
residuals, even if these have not been computed by the algorithm.
{\em (Case I)} $\xi^\lal_{{\cal J}}(w)\ge
4\gamma n$; {\em (Case II)} $\xi^\lal_{{\cal J}}(w) <
4\gamma n$ and $\ell^{\delta^+}({\cal J})\le 4\gamma n$; and {\em (Case III)} $\xi^\lal_{{\cal J}}(w) <
4\gamma n$ and $\ell^{\delta^+}({\cal J})> 4\gamma n$.

Recall that the algorithm uses an LLS direction instead of the affine
scaling direction whenever $\epsilon^\as(w)<10n^{3/2}\gamma$. Consider now
the case when an affine scaling direction is used, that is,
$\epsilon^\as(w)\ge 10n^{3/2}\gamma$.
According to
\Cref{lem:ll-decompose}\ref{i:lls-aff}, $\|\Rx^\lal -\Rx^\as\|,
\|\Rs^\lal-\Rs^\as\|\le  6n^{3/2}\gamma$. This implies that
$\xi^\lal_{{\cal J}}(w)\ge 4n^{3/2}\gamma\ge 4n\gamma$. Therefore, in cases II and III, an LLS step will be performed.

Starting with any given iteration, in each case we will identify a set
$J\subseteq [n]$ of indices with $|J|>1$, and start a {\em phase} of $O(\sqrt{n}|J|\log(\bar\chi^*+n))$
iterations (that can be either affine scaling or LLS steps). In each phase,
 we will guarantee that $\Psi$ increases by at least $|J|-1$. By definition, $0\le
 \Psi(\mu)\le n(n-1)(\log_2n+1)$, and if
 $\mu'<\mu$ then $\Psi(\mu')\ge \Psi(\mu)$. As we can partition the union of all iterations into disjoint phases,
 this yields the bound $O(n^{2.5}\log n\log(\bar\chi^*+n))$ on the
 total number of iterations.

We now consider each of the cases. We always let $\mu$ denote the
normalized duality gap at the current iteration, and  we let $q\in
[p]$ be the
layer such that $\xi^\lal_{{\cal J}}(w)= \xi^\lal_{J_q}(w)$.
\paragraph{Case I: $\xi^\lal_{{\cal J}}(w)\ge
4\gamma n$.}
\Cref{lem:
lower_bounds_for_crossover_events}  guarantees the existence of $x_i,s_j\in J_q$ such that $x_i^*/x_i, s_j^*/s_j\ge
4\gamma n/(3\sqrt{n})>1/(2^{10}n^{5.5})$.
Further, according to \Cref{lem:rho-bound}, $\rho^{\mu}(i,j)\ge
-|J_q|$. Thus, \Cref{lem:potential-master} is applicable for
$J=J_q$. The phase starting at $\mu$ comprises
$O(\sqrt{n}|J_q|\log(\bar\chi^*+n))$ iterations, after which we get a
normalized duality gap $\mu'$ such that $\Psi^{\mu'}(i,j)\ge 2|J_q|$, and
for each $\ell\in [n]\setminus \{i,j\}$, either
$\Psi^{\mu'}(i,\ell)\ge 2|J_q|$, or $\Psi^{\mu'}(\ell,j)\ge 2|J_q|$.

We can take advantage of these bounds for indices $\ell\in J_q$. Again by
\Cref{lem:rho-bound},
for any $\ell,\ell'\in J_q$, we have $\Psi^\mu(\ell,\ell')\le
\rho^\mu(\ell,\ell')\le |J_q|$. Thus, there are at least $|J_q|-1$ pairs
of indices $(\ell,\ell')$ for which  $\Psi^\mu(\ell,\ell')$ increases by at least
a factor 2 between iterations at $\mu$ and $\mu'$. The increase in the contribution of these terms to $\Psi(\mu)$ is at least $|J_q|-1$ during these iterations.

We note that this analysis works regardless whether an LLS
step or an affine scaling step was performed in the iteration at $\mu$.

\paragraph{Case II: $\xi^\lal_{{\cal J}}(w) <
4\gamma n$ and $\ell^{\delta^+}({\cal J})\le 4\gamma n$.} As
explained above, in this case we perform an LLS step in the iteration
at $\mu$, and we let $w^+$ denote the iterate obtained by the LLS
step.
 For $J=J_q$, \Cref{lem: case_lemma_does_not_crash}
guarantees the existence of $i,j\in J_q$ such that
$x_i^*/x_i^+,s_j^*/s_j^+>\beta/(16n^{3/2})$, and further, $\rho^{\mu^+}(i,j)>-|J_q|$.
We can therefore apply Lemma~\ref{lem:potential-master}. The phase
starting at $\mu$ includes the LLS step leading to $\mu^+$ (and the
subsequent centering step), and the additional
$O(\sqrt{n}|J_q|\log(\bar\chi^*+n))$ iterations ($\beta$ is a fixed constant in \Cref{alg:overall}) as in  Lemma~\ref{lem:potential-master}. As in Case I,
we get the desired potential increase compared to the potentials at $\mu$ in
layer $J_q$.

\paragraph{Case III: $\xi^\lal_{{\cal J}}(w) <
4\gamma n$ and $\ell^{\delta^+}({\cal J})>4\gamma n$.}
Again, the iteration at $\mu$ will use an LLS step. We apply
\Cref{lem: case_layer_crashes}, and set $J=J_q\cup J_r$ as in the
lemma.
The argument is the same as in Case II, using that \Cref{lem:
  case_layer_crashes} explicitly states that $\Psi^\mu(\ell,\ell')\le
|J|$ for any $\ell,\ell'\in J$, $\ell \neq \ell'$.
\end{proof}

\subsection{The iteration complexity bound for the Vavasis-Ye algorithm}\label{sec:VY}
We now show that the potential analysis described above also gives an
improved bound $O(n^{2.5}\log n$ $\log(\bar\chi_A+n))$ for the original VY algorithm
\cite{Vavasis1996}.

We recall the VY layering step. Order the variables via $ \pi$ such
that $\delta_{\pi(1)}\le \delta_{\pi(2)}\le\ldots\le
\delta_{\pi(n)}$. The layers will be consecutive sets in the ordering;
a new layer starts with $\pi(i+1)$ each time
$\delta_{\pi(i+1)}>g\delta_{\pi(i)}$, for a parameter $g={\rm
  poly}(n)\bar\chi_A$.

As outlined in the Introduction, the VY
algorithm can be seen as a special implementation of our
algorithm by setting $\hat \Le_{ij}=g\gamma/n$. With these
edge weights, we have that $\hat \Le^\delta_{ij}\ge \gamma/n$
precisely if $g\delta_j\ge \delta_i$.\footnote{For simplicity, in the
  Introduction we used $gx_i\ge x_j$ instead, which is almost the same in the
proximity in the central path.}

With these edge weights, it is easy to see that our \textsc{Layering}($\delta,\hat\Le$) subroutine finds
the exact same components as VY.
Moreover, the layers will be the
initial strongly connected components $C_i$ of $G_{\delta,\gamma/n}$: due to
the choice of $g$, this partition is automatically
$\delta$-balanced. There is no need to call \textsc{Verify-Lift}.

The essential difference compared to our algorithm is that  the values  $\hat\Le_{ij}=g\gamma/n$ are not lower bounds on
$\Le_{ij}$ as we require, but upper bounds instead. This is convenient
to simplify the construction of the layering. On the negative side, the
strongly connected components of $\hat G_{\delta,\gamma/n}$ may not anymore be
strongly connected in $G_{\delta,\gamma/n}$. Hence, we cannot use
\Cref{lem:rho-bound}, and consequently, \Cref{lem:potential-master}
does not hold.

Still, the $\hat\Le_{ij}$ bounds are
overestimating $\Le_{ij}$ by at most a factor
poly$(n)\bar\chi_A$. Therefore, the strongly connected components of
$\hat G_{\delta,n/\gamma}$ are strongly connected in
$G_{\delta,\sigma}$ for some $\sigma=1/({\rm poly}(n)\bar\chi_A)$.

Hence, the entire argument described in this section is applicable to
the VY algorithm, with a
different potential function defined with $\bar\chi_A$ instead of
$\bar\chi^*_A$. This is the reason why the iteration bound in
\Cref{lem:potential-master}, and therefore in \Cref{thm:overall-running}, also changes to $\bar\chi_A$ dependency.

It is worth noting that due to the overestimation of the $\Le_{ij}$
values, the VY algorithm uses a coarser layering than our
algorithm. Our algorithm splits up the VY layers into smaller parts so that
$\ell^\delta({\cal J})$ remains small, but within each part, the gaps
between the variables are bounded as a function of $\bar\chi^*_A$ instead of $\bar\chi_A$.

\section{Properties of the layered least square step}\label{sec:prox-proof}
This section is dedicated to the proofs of \Cref{prop:self-dual} on the duality of lifting scores and \Cref{lem:ll-decompose} on properties of LLS steps.

\propselfdual*
\begin{proof}
\label{proof:prop:self-dual}
We first treat the case where $\pi_I(W) = \{0\}$ or $\pi_J(W^\perp) =
\{0\}$. If $\pi_I(W) = \set{0}$ then $\|L_I^W\| = \ell^W(I) = 0$. Furthermore, in
this case $\R^I = \pi_I(W)^\perp = \pi_I(W^\perp \cap \R^n_I)$, and thus
$\{(0, w_J) : w \in W^\perp\} \subseteq W^\perp$. In particular, $\|L_J^W\| \leq 1$ and
$\ell^{W^\perp}(J) = 0$. Symmetrically, if $\pi_J(W^\perp) = \{0\}$ then
$\|L_J^{W^\perp}\| = \ell^{W^\perp}(J) = 0$, $\|L_I^W\| \leq 1$ and $\ell^{W}(I)
= 0$.

We now restrict our attention to the case where both $\pi_I(W),\pi_J(W^\perp)
\neq \{0\}$. Under this assumption, we show that $\|L_I^W\| = \|L_J^{W^\perp}\|$
and thus that $\ell^W(I) = \ell^{W^\perp}(J)$. Note that by non-emptyness, we
clearly have that $\|L_I^W\|,\|L_J^{W^\perp}\| \geq 1$.

We formulate a more general claim. Let $\{0\} \neq U, V \subset \R^n$ be linear
subspaces such that $U + V = \R^n$ and $U \cap V = \{0\}$.  Note that for the
orthogonal complements in $\R^n$, we also have $\{0\} \neq U^\perp,V^\perp$,
$U^\perp + V^\perp = \R^n$ and $U^\perp \cap V^\perp = \{0\}$.

\begin{claim}
\label{cl:subspace-angle}
Let $\{0\} \neq U, V \subset \R^n$ be linear subspaces
such that $U + V = \R^n$ and $U \cap V = \{0\}$. Thus,
for $z \in \R^n$,
there are unique decompositions $z = u + v$
with $u\in U$, $v \in V$ and $z=u'+v'$ with $u' \in U^\perp$ and $v' \in V^\perp$. Let $T : \R^n \to V$ be the map sending $Tz = v$.
Let $T' : \R^n \to V^\perp$ be the map sending $T'z = v'$.
Then, $\|T\| = \|T'\|$.
\end{claim}
\begin{proof}
To prove the statement, we claim that it suffices to show that if $\|T\| > 1$
then $\|T'\| \geq \|T\|$. To prove sufficiency, note that by symmetry, we
also get that if $\|T'\| > 1$ then $\|T\| \geq \|T'\|$. Note that $V,V^\perp \neq
\{0\}$ by assumption, and $Tz=z$ for $z\in V$, $T'z=z$ for $z\in V^\perp$. Thus, we always have $\|T\|, \|T'\| \geq 1$, and therefore the
equality $\|T\| = \|T'\|$ must hold in all cases. We now assume $\|T\| > 1$ and
show $\|T'\| \geq \|T\|$.

Representing $T$ as an $n \times n$ matrix, we write
$T = \sum_{i=1}^k \sigma_i v_i u_i^\T$ using a singular
value decomposition with $\sigma_1 \geq \dots \geq \sigma_k > 0$.
As such, $v_1,\dots,v_k$ is an orthonormal basis of $V$, since the ${\rm range}(T) =
V$, and $u_1,\dots,u_k$ is an orthonormal basis of $U^\perp$, since $\ker(T) =
U$, noting that we have
restricted to the singular vectors associated with positive singular values.
By assumption, we have that $\|T\| = \|Tu_1\| = \sigma_1 > 1$.

The proof is complete by showing that
\begin{equation}\label{eq:T-prime}
\left\|T'(v_1 - u_1/\sigma_1)\right\|\ge \sigma_1\|v_1 - u_1/\sigma_1\|,
\end{equation}
and that $\|v_1-u_1/\sigma_1\| > 0$, since then the vector $v_1 - u_1/\sigma_1$
will certify that $\|T'\| \geq \sigma_1$.

The map $T$ is a linear projection with $T^2 = T$.
Hence $\pr{u_i}{v_i} = \sigma_i^{-1}$ and
$\pr{u_i}{v_j} = 0$ for all $i \neq j$.

We show that $v_1 - \sigma_1^{-1}u_1$ can be decomposed as
$v_1 - \sigma_1 u_1 + (\sigma_1-\sigma_1^{-1}) u_1$ such that
$v_1 - \sigma_1 u_1\in V^\perp$ and $(\sigma_1-\sigma_1^{-1}) u_1\in U^\perp$.
\b{Therefore, $T'(v_1 - \sigma_1^{-1}u_1)=v_1 - \sigma_1 u_1$.}

The containment $(\sigma_1-\sigma_1^{-1})u_1\in U^\perp$ is immediate.
To show $v_1 - \sigma_1 u_1\in V^\perp$,  we need  that $\pr{v_1 - \sigma_1 u_1}{v_i}=0$ for all $i\in [k]$. For $i\ge 2$, this is true  since $\pr{u_i}{v_j} = 0$  and
$\pr{v_i}{v_j} = 0$.  For $i=1$, we have $\pr{v_1-\sigma_1 u_1}{v_1}=0$
since $\|v_1\|=1$ and $\pr{u_1}{v_1}=\sigma_1^{-1}$. Consequently,
$T'(v_1 - \sigma_1^{-1}u_1)=v_1 - \sigma_1 u_1$.

We compute $\left\|v_1 - \sigma_1^{-1} u_1\right\| = \sqrt{1 - \sigma_1^{-2}} > 0$, since
$\sigma_1 > 1$, and $\|v_1 - \sigma_1 u_1\| = \sqrt{\sigma_1^2 - 1}$. This verifies \eqref{eq:T-prime}, and thus  $\|T'\| \geq \sigma_1 = \|T\|$.
\end{proof}

To prove the lemma, we define
$\mathcal J = (J, I)$,
$U = W_{\mathcal J, 1}^\perp \times W_{\mathcal J, 2}^\perp$
and $V = W$ and let $T: \R^n \rightarrow V$ and $T': \R^n \rightarrow V^\perp$
be as in \Cref{cl:subspace-angle}. By assumption, $\{0\} \neq \pi_I(W) \Rightarrow \{0\} \neq V$ and
$\{0\} \neq \pi_J(W^\perp) = W_{\mathcal J, 1}^\perp \Rightarrow \{0\} \neq U$.
Applying \Cref{lem:lls-as-linsys}, $U, V$ satisfy
the conditions of \Cref{cl:subspace-angle} and $T = \mathrm{LLS}^{W,1}_{\cal
J}$. In particular, $\|T'\|=\|T\|$. Using the fact that $U^\perp = W_{\mathcal
J,1} \times W_{\mathcal J,2}$ and $V^\perp = W^\perp$, we similarly get that $T'
= \mathrm{LLS}^{W^\perp,1}_{\cal{ \bar J}}$, where
${\cal {\bar J}} = (I,J)$. By \eqref{def:aff-x} we have, for any $t
\in \pi_{\R^n_I}(W)$, that $Tt = \mathrm{LLS}^{W,1}_{\mathcal J}(t) =
L_I^W(t_I)$. Thus, $\|T\| \geq \|L_I^W\| \geq 1$.

To finish the proof of the lemma from the claim,
we show that $\|T\| \leq \|L^W_I\|$. By a symmetric
argument we get $\|T'\| = \|L^{W^\perp}_J\|$.

If $x \in \R^n_J$, then $Tx \in W\cap\R^n_J$ because
any $s \in W_{\mathcal J, 2}^\perp, t \in \pi_I(W)$ with
$s + t = 0$ must have $s = t= 0$ since
$W_{\mathcal J, 2}^\perp$ is orthogonal to $\pi_I(W)$.
But $W \cap \R^n_J$ and $W_{\mathcal J, 1}^\perp$ are orthogonal,
so $\|Tx\| \leq \|x\|$ because $x = Tx + (x - Tx)$ is an orthogonal decomposition.

If $y \in \R^n_I$, then $y_J = 0$ and hence $(Ty)_J = (Ty-y)_J$. Since
$(Ty-y)_J \in W_{\mathcal J,1}^\perp = \pi_J(W \cap \R^n_J)^\perp$, we see that
$Ty \in (W \cap \R^n_J)^\perp$. As such, for any $x \in \R^n_J, y \in \R^n_I$,
we see that $x \perp y$ and $Tx \perp Ty$. For $x,y \neq 0$, we thus have that
\[
\frac{\|T(x+y)\|^2}{\|x+y\|^2} = \frac{\|T(x)\|^2+\|T(y)\|^2}{\|x\|^2+\|y^2\|}
\leq \max \left\{\frac{\|T(x)\|^2}{\|x\|^2},\frac{\|T(y)\|^2}{\|y\|^2}\right\}
\leq \max \left\{1,\frac{\|T(y)\|^2}{\|y\|^2}\right\}.
\]
Since $\|L_I^W\| \geq 1$, we must have that $\|Tt\|/\|t\|$ is maximized
by some $t \in \R^n_I$.
From $\ker(T) = U$ it is clear that $\|Tt\|/\|t\|$ is maximized
by some $t \in U^\perp$.
Now, $U^\perp \cap \R^n_I = \pi_{\R^n_I}(W)$, so any $t$ maximizing
$\|Tt\|/\|t\|$ satisfies $Tt = L_I^W(t_I)$.
Therefore, $\|L_I^W\| \geq \|T\|$.
\end{proof}

Our next goal is to show \Cref{lem:ll-decompose}: for a layering
with small enough $\ell^\delta({\cal J})$, the LLS step approximately
satisfies \eqref{aff:sum}, that is, $\delta \Delta x^\lal +
\delta^{-1} \Delta s^\lal\approx -x^{1/2} s^{1/2}$. This also enables
us to derive bounds on the norm of
the residuals and on the step-length. We start by proving a few auxiliary
technical claims.
The next simple lemma allows us to take advantage of
low lifting scores in the layering.
\begin{lemma}\label{lem:force}
Let $u,v\in \R^n$ be two vectors such that $u-v\in W$.  Let
$I\subseteq [n]$, and $\delta\in \R^n_{++}$.
 Then there exists a vector $u' \in W + u$  satisfying $u'_I=v_I$ and 
\[
\|\delta_{[n]\setminus I} (u'_{[n]\setminus I}-u_{[n]\setminus I})\|\le
\ell^\delta(I)\|\delta_I(u_I-v_I)\|\, .
\]
\end{lemma}
\begin{proof}
We let
\[
u':=u+\delta^{-1}L^\delta_I(\delta_I(v_I-u_I))\, .
\]
The claim follows by the definition of the lifting score $\ell^\delta(I)$.
\end{proof}

The next lemma will be the key tool to prove
Lemma~\ref{lem:ll-decompose}. It is helpful to recall the characterization of
the LLS step in Section~\ref{sec:linsys}.
\begin{lemma}\label{lem:prox-sum-space}
Let  $w=(x,y,s)\in {\cal N}(\beta)$ for $\beta\in (0,1/4]$, let
$\mu=\mu(w)$ and  $\delta=\delta(w)$.
Let ${\cal J}=(J_1,\ldots,J_p)$ be a \b{$\delta(w)$-balanced}
layering, and let $\Delta w^\lal = (\Delta x^\lal, \Delta y^\lal, \Delta
s^\lal)$ denote the corresponding LLS direction.
Let $\Delta x\in \bigtimes_{k=1}^p W_{{\cal J},k}$ and $\Delta s\in
\bigtimes_{k=1}^p W^\perp_{{\cal J},k}$ as in \eqref{eq:ll-primal} and
\eqref{eq:ll-dual}, that is
\begin{align}
\delta \Delta x^\lal + \delta^{-1} \Delta s +x^{1/2} s^{1/2}=0 \, ,\label{eq:ll-primal-2}\\
\delta \Delta x  + \delta^{-1} \Delta s^\lal +x^{1/2} s^{1/2}=0\label{eq:ll-dual-2}.
\end{align}
Then, there exist vectors $\Delta \bar x\in \bigtimes_{k=1}^p W_{{\cal
    J},k}$ and $\Delta \bar s\in \bigtimes_{k=1}^p W^\perp_{{\cal
    J},k}$ such that
\begin{align}
\|\delta_{J_k}(\Delta \bar x_{J_k} - \Delta x_{J_k}^\lal)\|&\le 2n\ell^\delta({\cal
  J})\sqrt{\mu}\quad
                                               \forall k\in [p]\,
                                               \quad \mbox{and}\label{eq:x-bar-x}\\
\|\delta^{-1}_{J_k}(\Delta \bar s_{J_k} - \Delta s_{J_k}^\lal)\|&\le
                                                                  2n\ell^\delta({\cal
  J})   \sqrt{\mu}\quad
                                               \forall k\in [p]\, .\label{eq:s-bar-s}
\end{align}
\end{lemma}
\begin{proof}
Throughout, we use the shorthand notation $\lambda=\ell^\delta({\cal
  J})$.
 We
construct  $\Delta\bar x$; one can obtain
$\Delta\bar s$, using that the reverse layering has lifting score
$\lambda$ in $W^\perp\diag(\delta^{-1})$ according to
Lemma~\ref{lem:balanced-dual}.

We proceed by induction, constructing $\Delta \bar x_{J_k}\in W_{{\cal
    J},k}$ for $k=p,p-1,\ldots,1$. This will be given as
$\Delta \bar x_{J_k}=\Delta  x^{(k)}_{J_k}$ for a vector $\Delta
 x^{(k)}\in W$ such that $\Delta x^{(k)}_{J_{>k}}=0$. We prove
 the inductive hypothesis
\begin{equation}\label{ind-x-bar-x}
\left\|\delta _{J_{\le k}}\left(\Delta x^{(k)}_{J_{\le k}}-\Delta x^\lal_{J_{\le
    k}}\right)\right\|\le 2\lambda\sqrt{\mu} \sum_{q=k+1}^p \sqrt{|J_q|}\, .
\end{equation}
Note that \eqref{eq:x-bar-x} follows by restricting the norm on the
LHS to $J_k$ and since the sum on the RHS is $\le n$.

For $k=p$, the RHS is 0. We simply set $\Delta x^{(p)}=\Delta x^\lal$, that is,
$\Delta \bar x_{J_p}=\Delta x^\lal_{J_p}$, trivially satisfying the
hypothesis. Consider now $k<p$, and assume that we have a
$\Delta \bar x_{J_{k+1}}=\Delta  x^{(k+1)}_{J_{k+1}}$ satisfying
\eqref{ind-x-bar-x} for $k+1$.
From \eqref{eq:ll-primal-2} and the induction hypothesis,
we get that
\[
\begin{aligned}
&\|\delta_{J_{k+1}} \Delta \bar x_{J_{k+1}} + \delta^{-1}_{J_{k+1}} \Delta s_{J_{k+1}}\|\le
\|x^{1/2}_{J_{k+1}}
s^{1/2}_{J_{k+1}}\|+\|\delta_{J_{k+1}}(\Delta \bar x_{J_{k+1}}-
\Delta x_{J_{k+1}}^\lal)\|\\
&\le \|x^{1/2}_{J_{k+1}}
s^{1/2}_{J_{k+1}}\| +2\lambda\sqrt{\mu}\sum_{q=k+2}^p \sqrt{|J_q|}
\le \sqrt{1+\beta}\sqrt{\mu|J_{k+1}|}+2n\lambda \sqrt{\mu }<2 \sqrt{\mu|J_{k+1}|}\, ,
\end{aligned}
\]

using also that $w\in {\cal N}(\beta)$, Proposition~\ref{prop:x_i-s_i}, and the assumptions $\beta\le
1/4$, $\lambda\le \beta/(32n^2)$.
Note that $\Delta \bar x_{J_{k+1}}\in W_{{\cal J},k}$ and $\Delta
s_{J_{k+1}}\in W^\perp_{{\cal J},k}$ are orthogonal vectors.
The above inequality therefore implies
\[
\|\delta_{J_{k+1}} \Delta \bar x_{J_{k+1}} \|\le
2\sqrt{\mu|J_{k+1}|}\, .
\]
Let us now use Lemma~\ref{lem:force} to obtain $\Delta  x^{(k)}$ for
$u= \Delta  x^{(k+1)}$, $v=0$, and $I=J_{>k}$. That is, we get $\Delta
x^{(k)}_{J_{>k}}=0$, $\Delta
x^{(k)}\in W$, and
\[
\begin{aligned}
\|\delta_{J_{\le k}}( \Delta  x^{(k)}_{J_{\le k}}-\Delta
x^{(k+1)}_{J_{\le k}})\|&\le \lambda \|\delta_{J_{> k}} \Delta
x^{(k+1)}_{J_{> k}}\|\\
&=\lambda \|\delta_{J_{k+1}} \Delta \bar
x_{J_{k+1}}\|\le 2\lambda \sqrt{\mu|J_{k+1}|}\, .
\end{aligned}
\]
By the triangle inequality and the induction hypothesis \eqref{ind-x-bar-x} for $k+1$,
\begin{align*}
\|\delta_{J_{\le k}}(\Delta x^{(k)}_{J_{\le k}}- \Delta x^{\lal}_{J_{\leq k}})\| &\leq \|\delta_{J_{\le k}}(\Delta x^{(k)}_{J_{\le k}} - \Delta x^{(k+1)}_{J_{\le k}})\| + \|\delta_{J_{\le k}}(\Delta x^{(k+1)}_{J_{\leq k}} - \Delta x^{\lal}_{J_{\leq k}})\| \\
&\leq 2\lambda \sqrt{\mu|J_{k+1}|} + 2 \lambda \sum_{q=k+2}^p \sqrt{\mu |J_q|} ,
\end{align*}
yielding the induction hypothesis for $k$.
\end{proof}

\lldecompose*
\begin{proof}\label{proof:lem:ll-decompose}
Again, we use $\lambda=\ell^\delta({\cal
  J})$.

\noindent{\bf Part \ref{i:prox}.}
Clearly, \eqref{eq:prox-layer} implies \eqref{eq:prox-all}.
To show \eqref{eq:prox-layer}, we use Lemma~\ref{lem:prox-sum-space}
to obtain $\Delta \bar x$ and $\Delta \bar s$ as in \eqref{eq:x-bar-x}
and \eqref{eq:s-bar-s}. We will also use
$\Delta x\in \bigtimes_{k=1}^p W_{{\cal J},k}$ and $\Delta s\in
\bigtimes_{k=1}^p W^\perp_{{\cal J},k}$ as in \eqref{eq:ll-primal-2} and
\eqref{eq:ll-dual-2}.

Select any layer
$k\in [p]$.
From \eqref{eq:ll-primal-2}, we get that
\begin{equation}\label{eq:x-bar-x-diff}
\|\delta_{J_k} \Delta \bar x_{J_k} + \delta^{-1}_{J_k} \Delta s_{J_k}
+x^{1/2}_{J_k} s^{1/2}_{J_k}\|=\|\delta_{J_{k}}(\Delta \bar x_{J_k}-
\Delta x_{J_k}^\lal)\|\le  2n\lambda\sqrt{\mu}\, .
\end{equation}
Similarly, from \eqref{eq:ll-dual-2}, we see that
\[
\|\delta^{-1}_{J_k} \Delta \bar s_{J_k} + \delta_{J_k} \Delta x_{J_k}
+x^{1/2}_{J_k} s^{1/2}_{J_k}\|=\|\delta^{-1}_{J_{k}}(\Delta \bar s_{J_k}-
\Delta s_{J_k}^\lal)\|\le  2n\lambda\sqrt{\mu}\, .
\]
From the above inequalities, we see that
\[
\|\delta_{J_k} (\Delta \bar x_{J_k} -\Delta x_{J_k})+
\delta^{-1}_{J_k} (\Delta s_{J_k}-\Delta\bar s_{J_k})\|\le 4 n\lambda\sqrt{\mu}\, .
\]
Since $\delta_{J_k} (\Delta \bar x_{J_k} -\Delta x_{J_k})$ and
$\delta^{-1}_{J_k} (\Delta s_{J_k}-\Delta\bar s_{J_k})$ are orthogonal
vectors, we have
\[
\|\delta_{J_k} (\Delta \bar x_{J_k} -\Delta x_{J_k})\|,\,
\|\delta^{-1}_{J_k} (\Delta s_{J_k}-\Delta\bar s_{J_k})\|\le 4n\lambda\sqrt{\mu}\, .
\]
 Together with \eqref{eq:x-bar-x}, this yields
$\|\delta_{J_k} (\Delta x^\lal_{J_k} -\Delta x_{J_k})\|\le
6n\lambda \sqrt{\mu}$.
Combined with \eqref{eq:ll-dual}, we get
\[
\|\delta_{J_k} \Delta x^\lal_{J_k} + \delta^{-1}_{J_k} \Delta s^\lal_{J_k}
+x^{1/2}_{J_k} s^{1/2}_{J_k}\| = \|\delta_{J_k} (\Delta x^\lal_{J_k} -\Delta x_{J_k})\|\le
6n \lambda\sqrt{\mu}\, ,
\]
thus, \eqref{eq:prox-layer} follows.

\paragraph{Part \ref{i:lls-aff}.}  Recall from
Lemma~\ref{lem:affscale}\ref{i:affscale-identity} that
$\sqrt{\mu} \Rx^\as+\sqrt{\mu} \Rs^\as={x^{1/2}s^{1/2}}$. From part \ref{i:prox},
  we can similarly see that
\[
\|\sqrt{\mu} \Rx^\lal +\sqrt{\mu}
\Rs^\lal-{x^{1/2}s^{1/2}}\|\le
6n^{3/2}\lambda\sqrt{\mu}\, .
\]
From these, we get
\[
\| (\Rx^\lal-\Rx^\as)+
(\Rs^\lal-\Rs^\as)\|\le
6n^{3/2}\lambda\, .
\]
The claim follows since $\Rx^\lal-\Rx^\as\in \diag(\delta) W$ and $\Rs^\lal-\Rs^\as\in
\diag(\delta^{-1}) W^\perp $ are orthogonal vectors.
\paragraph{Part \ref{i:norm}.}
Both bounds follow from the previous part and
Lemma~\ref{lem:affscale}\ref{i:lower-bound-in-neighbourhood}, using
the assumption $\ell^\delta({\cal J})\le \beta/(32n^2)$.

\paragraph{Part \ref{i:progress}.}
Let $w^+=w+\alpha \Delta w^\lal$. We need to find
the largest value $\alpha>0$ such
that $w^+\in {\cal N}(2\beta)$. 
To begin, we first show that the normalized
duality gap $\mu(w^+)$ fulfills $\mu(w^+) = (1-\alpha)\mu$ for any $\alpha \in \R$. For this
purpose, we use the decomposition:
\begin{equation}
\label{eq:decomp-w}
(x + \alpha \Delta x^{\lal})(s + \alpha \Delta s^{\lal})  = (1-\alpha) xs +
\alpha (x + \Delta x^{\lal})(s+ \Delta s^{\lal}) - \alpha(1-\alpha) \Delta
x^{\lal} \Delta s^{\lal}.
\end{equation}
Recall from Part~\ref{i:prox} that there exists $\Delta x\in \bigtimes_{k=1}^p
W_{{\cal J},k}$ and $\Delta s\in \bigtimes_{k=1}^p W^\perp_{{\cal J},k}$ as in
\eqref{eq:ll-primal-2} and \eqref{eq:ll-dual-2} such that $\delta \Delta
x^{\lal} + \delta^{-1} \Delta s =
- \delta x$ and $\delta \Delta x + \delta^{-1} \Delta s^{\lal} = -\delta^{-1}
  s$. In particular, $x + \Delta x^{\lal} = -\delta^{-2} \Delta s$ and $s + \Delta
s^{\lal} = -\delta^2 \Delta x$. Noting that $\Delta x^{\lal} \perp \Delta s^\lal$
and $\Delta x \perp \Delta s$, taking the average of the coordinates on both
sides of~\eqref{eq:decomp-w}, we get that
\begin{align}
\mu(w + \alpha \Delta w^\lal) &=
(1-\alpha) \mu(w) + \alpha \langle x + \Delta x^{\lal}, s + \Delta s^\lal\rangle/n -
\alpha(1-\alpha) \langle \Delta x^{\lal}, \Delta s^\lal \rangle/n \nonumber \\
&= (1-\alpha) \mu(w) + \alpha \langle \delta^{-2} \Delta s, \delta^2 \Delta x
\rangle/n
\nonumber \\
&= (1-\alpha) \mu(w), \label{eq:exact-alpha}
\end{align}
as needed.

Let $\epsilon := \eps^{\lal}(w)$. To obtain the desired lower bound on the
step-length, given~\eqref{eq:exact-alpha} it suffices to show that for all $0
\leq \alpha <  1-\frac{3 \sqrt{n} \epsilon}{\beta}$ that
\begin{equation}\label{eq:approx-alpha}
\left\|\frac{(x+\alpha \Delta x^\lal)(s+\alpha \Delta
    s^\lal)}{(1-\alpha)\mu}-e\right\|\le 2\beta\, .
\end{equation}
We will need a bound on the product of the LLS residuals:
\begin{equation}\label{eq:LLS-residuals}
\begin{aligned}
\left\|\Rx^\lal \Rs^\lal-\frac{1}{\mu}\Delta x^\lal \Delta
  s^\lal \right\|&=\left\|\frac{x^{1/2}s^{1/2}}{\sqrt{\mu}}\cdot \frac{\delta
  \Delta x^\lal +\delta^{-1}\Delta s^\lal + x^{1/2}s^{1/2}}{\sqrt{\mu}}\right\|\\
&\le 6(1+2\beta)n^{3/2}\lambda \le \frac{\beta }{4}\, ,
\end{aligned}
\end{equation}
using Proposition~\ref{prop:near-central},
part~\ref{i:prox}, and the assumptions $\lambda\le \beta/(32n^2)$,
$\beta\le 1/4$. Another useful bound will be
\begin{equation}\label{eq:prod-bound-epsilon}
\begin{aligned}
\|\Rx^\lal \Rs^\lal\|^2 &= \sum_{i \in [n]} \left|\Rx^\lal_i\right|^2\left|\Rs^\lal_i\right|^2 \leq \epsilon^2 \sum_{i \in [n]} \max\Big\{\left|\Rx^\lal_i\right|^2,\left|\Rs^\lal_i\right|^2\Big\} \\
& \leq \epsilon^2(\|\Rx^\lal\|^2 + \|\Rs^\lal\|^2)  \leq 2n \epsilon^2\, .
\end{aligned}
\end{equation}
The last inequality uses part~\ref{i:norm}.
With \eqref{eq:decomp-w} we are ready to get the bound in \eqref{eq:approx-alpha}, as 
\begin{align*}
\Big\|\frac{(x + \alpha \Delta x^\lal)(s + \alpha \Delta s^\lal)}{(1-\alpha)\mu} - e\Big\|
& \leq \beta + \Big\|\frac{\alpha}{(1-\alpha)\mu}(x+\Delta x^\lal)(s +
  \Delta s^\lal) - \frac{\alpha}{\mu} \Delta x^\lal  \Delta s^\lal \Big\|
  \\
&= \beta + \Big\|\Big(\frac{\alpha}{1 - \alpha} - \alpha\Big)\Rx^\lal \Rs^\lal + \alpha
  \Big(\Rx^\lal \Rs^\lal - \frac{1}{\mu}\Delta x^\lal \Delta
  s^\lal\Big)\Big\| \, \\
&
\le \beta + \frac{\alpha^2}{1 - \alpha}\|\Rx^\lal \Rs^\lal\| + \alpha
\Big\|\Rx^\lal \Rs^\lal - \frac{1}{\mu}\Delta x^\lal \Delta
s^\lal\Big\| \\
&\leq \beta + \frac{\sqrt{2n}\epsilon}{1 - \alpha} +
  \frac{\beta}{4} \le \frac{5}{4}\beta +
  \frac{\sqrt{2n}\epsilon}{1 - \alpha}\, .
\end{align*}
This value is $\le 2\beta$ whenever ${2\sqrt{n}\epsilon}/({1 - \alpha})\le (3/4)
\beta \Leftarrow \alpha < 1 - \frac{3 \sqrt{n} \epsilon}{\beta}$, as
needed.

\paragraph{Part \ref{i:terminate}.}
From  part \ref{i:progress}, it is immediate that $\epsilon^\lal(w)=0$ implies
$\alpha=1$. If $\alpha=1$, we have that $w+\Delta w^\lal$ is the limit of
(strictly) feasible solutions to~\eqref{LP_primal_dual} and thus is also a
feasible solution. Optimality of $w+\Delta w^\lal$ now follows from Part
\ref{i:progress}, since $\alpha=1$ implies
$\mu(w+\Delta w^\lal)=0$.
The remaining implication is that if $w+\Delta w^\lal$ is optimal, then $\epsilon^\lal(w)=0$. Recall that $\Rx_i^\lal=\delta_i(x_i+\Delta x_i^\lal)/\sqrt{\mu}$ and $\Rs_i^\lal=\delta^{-1}_i(s_i+\Delta s_i^\lal)/\sqrt{\mu}$. The optimality of $w+\Delta w^\lal$ means that for each $i\in [n]$, either $x_i+\Delta x_i^\lal=0$ or $s_i+\Delta s_i^\lal=0$. Therefore, $\epsilon^\lal(w)=0$.
\end{proof}

\section{Proofs of the main lemmas for the potential analysis}\label{sec:main-lemmas}
\lemlowerboundsforcrossoverevents*
\begin{proof}[Proof of Lemma~\ref{lem: lower_bounds_for_crossover_events}]
  \label{proof:lem: lower_bounds_for_crossover_events}
We prove part (i); part (ii) follows analogously using
Lemma~\ref{lem:balanced-dual}.
Let $z$ be a vector fulfilling the statement of Lemma~\ref{lem:force}
for $u=x^*$, $v=x+\Delta x^\lal$, and $I=J_{>q}$. Then $z \in W + x$,
$z_{J_{>q}}=x_{J_{>q}}+\Delta x_{J_{>q}}^\lal$ and by $\ell^\delta(\cal J) \le \gamma$
\[
\left\|\delta_{J_{\le q}} (x^*_{J_{\le q}}-z_{J_{\le q}})\right\|\le \gamma \left\|{\delta_{J_{>q}} \big(x^*_{J_{>q}}-(x_{J_{>q}}+\Delta
x^\lal_{J_{>q}})\big)}\right\|.
\]
Restricting to the components in $J_q$, and dividing by
$\sqrt{\mu}$, we get
\begin{equation}\label{eq:H-lift}
\left\|\frac{\delta_{J_q}(x^*_{J_q}-z_{J_q})}{\sqrt{\mu}}\right\|\le \gamma \left\|\frac{\delta_{J_{>q}}\big(x^*_{J_{>q}}-(x_{J_{>q}}+\Delta
x^\lal_{J_{>q}})\big)}{\sqrt{\mu}}\right\|\le \gamma \left\|\frac{\delta_{J_{>q}}x^*_{J_{>q}}}{\sqrt{\mu}}\right\|+\gamma\|\Rx^\lal_{J_{>q}}\|
\, .
\end{equation}
Since $w\in {\cal N}(\beta)$, from
\Cref{prop:near-central} and \eqref{eq:delta-beta} we see that for $i \in [n]$
\[
\frac{\delta_i}{\sqrt{\mu}}\le \frac{1}{\sqrt{1-2\beta}}\cdot
\frac{\delta_i(w(\mu))}{\sqrt{\mu}}=\frac{1}{\sqrt{1-2\beta}}\cdot
\frac{1}{x_i(\mu)}\, ,
\]
and therefore
\[
\left\|\frac{\delta_{J_{>q}}x^*_{J_{>q}}}{\sqrt{\mu}}\right\|\le \frac{1}{\sqrt{1-2\beta}}
\left\|{x(\mu)^{-1}_{J_{>q}}x^*_{J_{>q}}}\right\|\,
\le
\frac{1}{\sqrt{1-2\beta}}\cdot\left\|{x(\mu)^{-1}_{J_{>q}}x^*_{J_{>q}}}\right\|_1\le
\frac{n}{\sqrt{1-2\beta}}\,,
\]
where the last inequality follows by \Cref{lem: central_path_bounded_l1_norm}.

Using the above bounds with \eqref{eq:H-lift}, along with
$\|\Rx^\lal_{J_{\ge q}}\|\le \|\Rx^\lal\|\le \sqrt{2n}$ from
Lemma~\ref{lem:ll-decompose}\ref{i:norm}, we get
\[
\left\|\frac{\delta_{J_q} z_{J_q}}{\sqrt{\mu}}\right\|\le
\left\|\frac{\delta_{J_q} x^*_{J_q}}{\sqrt{\mu}}\right\|+
\frac{\gamma n}{\sqrt{1-2\beta}}+\gamma \sqrt{2n}\le
\left\|\frac{\delta_{J_q} x^*_{J_q}}{\sqrt{\mu}}\right\|+2\gamma n \, ,
\]
using that $\beta\le 1/8$ and $n\ge 3$.
Note that $z$ is a feasible solution to the least-squares problem
which is optimally solved by $x_{J_q}^\lal$ for layer $J_q$ and so
\[
 \|R x_{J_q}^\lal\|\le\left\|\frac{\delta_{J_q} z_{J_q}}{\sqrt{\mu}}\right\|\, .
\]
It follows that
\[
\left\|\frac{\delta_{J_q} x^*_{J_q}}{\sqrt{\mu}}\right\|\ge\|R
x_{J_q}^\lal\|-2\gamma n\, .
\]
Let us pick $i=\argmax_{t\in J_q}|\delta_t x^*_t|$. Using \Cref{prop:x_i-s_i},
\[
\frac{x^*_i}{x_i}\ge \frac{1}{1+\beta}  \cdot\frac{\delta_i
  x^*_i}{\sqrt{\mu}}\ge \frac{\|R
x_{J_q}^\lal\|-2\gamma n}{(1+\beta)\sqrt{n}}\ge
\frac{2}{3\sqrt{n}}\cdot (\|\Rx_{J_q}^\lal\| - 2\gamma n)\, ,
\]
completing the proof.
\end{proof}

\lempotentialmaster*
\begin{proof}[Proof of \Cref{lem:potential-master}]
  \label{proof:lem:potential-master}
Let us select a value $\mu'$ such that
\[
\log \mu - \log \mu'\ge 5\tau \log\left(\frac{4n\Le^*}{\gamma}\right) +31\log n+44-4\log\beta\, .
\]
The normalized duality gap
decreases to such value within $O(\beta^{-1}\sqrt{n}\tau\cdot\log(\bar\chi^* + n))$
iterations, recalling that $\log(\bar\chi^* + n) = \Theta(\log(\Le^* + n))$. The step-lengths for the affine scaling and LLS steps are stated in \Cref{prop:predictor-corrector} and \Cref{lem:ll-decompose}\ref{i:progress}. Whenever the algorithm chooses an LLS
step,  $\epsilon^\as(w) < 10n^{3/2}\gamma$. Thus, the progress in $\mu$
will be at least as much (in fact, much better) than the
$1-\beta/\sqrt{n}$ guarantee for the affine scaling step in \Cref{prop:predictor-corrector}.

Let  $w'=(x',y',s')$ be the central path element corresponding to
$\mu'$, and let $\delta'=\delta(w')$. From now on we use the shorthand notation
\[
\potentialshorthand := \log \left(\frac{4n\Le^*}{\gamma}\right)\, .
\]
We first show that
\begin{equation}
  \label{eq:mu_prime_bound}
  \potentialshorthand\rho^{\mu'}(i,j)\ge 4\potentialshorthand\tau+18\log n+ 22 \log 2 - 2 \log \beta
\end{equation} for
$\mu'$, and therefore, $\potentialshorthand\Psi^{\mu'}(i,j)\ge \min(2\potentialshorthand n, 4\potentialshorthand\tau+18\log n+ 22 \log 2 - 2 \log \beta) \ge 2\potentialshorthand\tau $ as $\tau \le n$.
Recalling the
definition $\Le_{ij}^\delta=\Le_{ij}\delta_j/\delta_i$, we see that
according to \Cref{prop:x_i-s_i}, \[
\Le_{ij}^\delta\le \frac{\Le_{ij}}{(1-\beta)^2}\cdot\frac{x_is_j}{\mu},
\quad \mbox{and}\quad
\Le_{ij}^{\delta'}={\Le_{ij}}\cdot \frac{x'_is'_j}{\mu'}\, .
\]
Thus,
\begin{align*}
\potentialshorthand \rho^{\mu'}(i,j) &\ge\potentialshorthand \rho^\mu(i,j)+ {\log\mu-\log\mu' +2\log(1-\beta) +
\log x_i'-\log x_i+\log s'_j-\log s_j}\\
&\ge \potentialshorthand\rho^\mu(i,j)+ 5\potentialshorthand\tau + {31\log n + 44
  -4\log \beta +2\log(1-\beta) +
\log x_i'-\log x_i+\log s'_j-\log s_j}.
\end{align*}

Using the near-monotonicity of the central path (\Cref{lem:
  central_path_bounded_l1_norm}), we have $x_i'\ge x^*_i/n$ and
$s_j'\ge s^*_j/n$. Together with our assumptions $x_i^*\ge
\beta x_i/(2^{10}n^{5.5})$ and $s_i^*\ge \beta s_i/(2^{10}n^{5.5})$, we see that
\[
\log x_i'-\log x_i+\log s'_j-\log s_j\ge -13\log n-20\log 2+2\log\beta\, .
\]
Using the assumption $\rho^\mu(i,j)>-\tau$ of the lemma, we can establish \eqref{eq:mu_prime_bound} as $\beta < 1/8$.

\medskip

Next, consider any $\ell\in [n]\setminus\{i,j\}$.
From the triangle inequality \Cref{lem:imbalance_triangle_inequality}\ref{i:triangle} it follows that
$
\Le_{ij}^{\delta'} \le \Le_{i\ell}^{\delta'} \cdot \Le_{\ell j}^{\delta'}\, ,
$
which gives
$\rho^{\mu'}(i,\ell) + \rho^{\mu'}(\ell,j) \ge \rho^{\mu'}(i,j).$
We therefore get
\[\max\{\potentialshorthand\rho^{\mu'}(i,\ell),\potentialshorthand\rho^{\mu'}(\ell,j) \} \ge \frac12 \potentialshorthand \rho^{\mu'}(i,j) \stackrel{\eqref{eq:mu_prime_bound}}{\ge} 2\potentialshorthand\tau+9\log n+11\log 2-\log\beta.\]

We next show that if $\potentialshorthand\rho^{\mu'}(i,\ell)\ge 2\potentialshorthand\tau+9\log n+11\log 2-\log\beta$, then
$\Psi^{\mu'}(i,\ell)\ge 2\tau$. The case $\potentialshorthand\rho^{\mu'}(\ell,j)\ge
2\potentialshorthand\tau+9\log n+11\log 2-\log \beta$ follows analogously.

Consider any $0<\bar \mu<\mu'$
with the corresponding central path point $\bar w=(\bar x,\bar y,\bar
s)$. The proof is complete by showing $\potentialshorthand\rho^{\bar \mu}(i,\ell)\ge
\potentialshorthand\rho^{\mu'}(i,\ell)-9\log n-11\log 2+\log \beta$.
Recall that for central path elements, we have
$\Le^{\delta'}_{ij}=\Le_{ij}x'_i/x'_j$, and $\Le^{\bar
  \delta}_{ij}=\Le_{ij}\bar x_i/\bar x_j$.
Therefore
\[
\potentialshorthand\rho^{\bar\mu}(i,j)=\potentialshorthand\rho^{\mu'}(i,j)+ {\log \bar
x_i-\log x_i'-\log \bar x_j+\log x_j'}\, .
\]
Using \Cref{prop:near-central}, \Cref{lem: central_path_bounded_l1_norm} and the assumption
$x^*_i\ge \beta x_i/(2^{10}n^{5.5})$, we have
$\bar
x_j\le nx_j'$ and

\[\bar x_i\ge \frac{x_i^*}{n}\ge
\frac{\beta x_i}{2^{10}n^{6.5}}\ge \frac{\beta (1-\beta) x'_i}{2^{10}n^{7.5}} \ge \frac{\beta x'_i}{2^{11}n^{7.5}}\, .
\]

Using these bounds, we get
\begin{align*}
\potentialshorthand\rho^{\bar\mu}(i,j) &\ge \potentialshorthand\rho^{\mu'}(i,j) - {9\log n -
                            11\log 2+\log \beta},
\end{align*}
completing the proof.
\end{proof}

It remains to prove \Cref{lem: case_lemma_does_not_crash} and
\Cref{lem: case_layer_crashes},  addressing the more difficult case $\xi_{\cal
  J}^\lal < 4\gamma n$. It is useful to decompose the variables into
two sets. We let
\begin{equation}\label{eq:B-N}
\llspartition B := \{t \in [n] : |\Rs_t^\lal| < 4\gamma n\},\quad \mbox{and}\quad \llspartition  N:=\{t \in [n] :
|\Rx_t^\lal| < 4\gamma n\}\, .
\end{equation}
The assumption $\xi_{\cal  J}^\lal < 4\gamma n$ implies that for every
layer $J_k$, either $J_k\subseteq \llspartition B$ or $J_k\subseteq \llspartition N$. The next two
lemmas describe the relations between $\delta$ and $\delta^+$.

\begin{lemma} \label{lem:near_uniform_shootdown}
Let $w\in {\cal N}(\beta)$ for $\beta\in (0,1/8]$, and assume
$\ell^\delta({\cal J})\le \gamma$ and  $\epsilon^\lal(w)
< 4\gamma n$.
 For the next iterate $w^+ = (x^+, y^+, s^+) \in \overline{\mathcal{N}}(2\beta)$,
we have
\begin{enumerate}[label=(\roman*)]
\item\label{i:stay-same} For $i \in \llspartition B$,
\[
\frac12 \cdot \sqrt{\frac{\mu^+}{\mu}} \le
\frac{\delta^+_i}{\delta_i}\le 2 \cdot
\sqrt{\frac{\mu^+}{\mu}}\,\quad \mbox{and}\quad  \delta_i^{-1}s_i^+\le
\frac{3\mu^+}{\sqrt{\mu}}\, .
\]
\item\label{i:scale-down} For $i \in \llspartition N$,
\[
\frac12\cdot \sqrt{\frac{\mu}{\mu^+}} \le
\frac{\delta^+_i}{\delta_i}\le 2 \cdot
\sqrt{\frac{\mu}{\mu^+}}\, \quad \mbox{and}\quad  \delta_ix_i^+\le
\frac{3\mu^+}{\sqrt{\mu}}\, .
\]
\item \label{i:both-i-j} If
$i,j \in \llspartition B$ or $i,j\in \llspartition N$, then
\begin{align*}
\frac{1}{4} \leq
  \frac{\Le_{ij}^{\delta}}{\Le_{ij}^{\delta^+}}=\frac{\delta^+_i
  \delta_j}{\delta_i \delta^+_j}  \leq 4\, .
\end{align*}
\item\label{i:N-to-B}
If $i\in \llspartition N$ and $j\in \llspartition B$, then
\[
\frac{\Le_{ij}^{\delta}}{\Le_{ij}^{\delta^+}}
\ge 4n^{3.5}\, .
\]
\end{enumerate}
\end{lemma}
\begin{proof}
{\bf Part (i).}  By
\Cref{lem:ll-decompose}\ref{i:prox}, we see that
\begin{align*}
\|\delta_B \Delta x^{\lal}_B\|_\infty
&\leq \|\delta_B \Delta x^{\lal}_B + \delta_B^{-1} \Delta s^{\lal}_B + x^{1/2}_B
s^{1/2}_B\|_\infty + \|\delta^{-1}_B(\Delta s^\lal_B + s_B)\|_\infty
  \\
&=\|\delta_B \Delta x^{\lal}_B + \delta_B^{-1} \Delta s^{\lal}_B + x^{1/2}_B
s^{1/2}_B\|_\infty + \sqrt{\mu}\|\Rs^\lal_B\|_\infty\\
&\leq \sqrt{\mu}\left(6 n \ell^\delta({\cal J}) +4n\gamma\right)\le 10n\gamma\sqrt{\mu}\le
  \sqrt{\mu}/64\, ,
\end{align*}
by the assumption on $\ell^\delta({\cal J}) $ and the definition of $\llspartition B$.

By construction of the LLS step, $|x_i^+-x_i|=\alpha^+|\Delta x_i^\lal|\le |\Delta x_i^\lal|$,
 recalling that $0 \leq \alpha^+ \leq 1$. Using the bound derived above, for $i\in \llspartition B$ we get
\[
\left|\frac{x_i^+}{x_i}-1\right|\le \left|\frac{\Delta x_i^\lal}{x_i}\right|=
\frac{|\delta_i \Delta x_i^\lal|}{\delta_i x_i}\le
\frac{\sqrt{\mu}}{64\delta_ix_i}\le
  \frac{1}{32}\, ,
\]
where the last inequality follows from \Cref{prop:x_i-s_i}.
As
\[
\frac{\delta^+_i}{\delta_i}=\sqrt{\frac{x^+_is^+_i}{x_is_i}}\cdot \frac{x_i}{x^+_i} \quad \text {and} \quad \frac{1 - 2\beta}{1 + \beta} \frac{\sqrt{\mu^+}}{\sqrt{\mu}} \le \sqrt{\frac{x^+_is^+_i}{x_is_i}} \le \frac{1 + 2\beta}{1 - \beta} \frac{\sqrt{\mu^+}}{\sqrt{\mu}}
\]
by \Cref{prop:x_i-s_i} the claimed bounds follow with $\beta\le 1/8$.

To get the upper bound on $\delta^{-1}_is_i^+$, again with \Cref{prop:x_i-s_i}
\[
\delta^{-1}_is_i^+=\frac{\delta^+_i}{\delta_i\delta^+_i}
s_i^+=\frac{\delta^+_i}{\delta_i} \cdot \sqrt{x_i^+s_i^+} \le 2 \sqrt{\frac{\mu^+}{\mu}} \cdot (1+2\beta) \sqrt{\mu^+}
\le \frac{3\mu^+}{\sqrt{\mu}} \, .
\]

\paragraph{Part \ref{i:scale-down}.} Analogously to \ref{i:stay-same}.
\paragraph{Part \ref{i:both-i-j}.} Immediate from parts \ref{i:stay-same} and \ref{i:scale-down}.
\paragraph{Part \ref{i:N-to-B}.} Follows by parts \ref{i:stay-same} and \ref{i:scale-down}, and by the
lower bound on $\sqrt{\mu/\mu^+}$ obtained from \Cref{lem:ll-decompose}\ref{i:progress}
as follows
\[\frac{\Le_{ij}^{\delta}}{\Le_{ij}^{\delta^+}} = \frac{\delta^+_i
\delta_j}{\delta_i \delta^+_j} \ge \frac{\mu}{4\mu^+} = \frac{1}{4(1-\alpha^+)} \ge
\frac{\beta}{12\sqrt{n}\epsilon^\lal(w)}\ge 4n^{3.5}. \qedhere\]
\end{proof}

\lemcaselemmadoesnotcrash*
\begin{proof}[Proof of \Cref{lem: case_lemma_does_not_crash}]
  \label{proof:lem: case_lemma_does_not_crash}
Without loss of generality, let $\xi_{\cal
  J}^\lal=\xi_{J_q}^\lal=\|\Rx_{J_q}^\lal\|$ for a layer $q$ with $J_q\subseteq \llspartition N$. The case
$\xi_{J_q}^\lal=\|\Rs_{J_q}^\lal\|$ and $J_q\subseteq \llspartition B$ can be treated analogously.

By \Cref{lem:ll-decompose}\ref{i:norm}, $\|\Rs_{J_q}^\lal\|\ge
\frac{1}{2}-\frac{3}{4}\beta>\frac{1}4+2n\gamma$, and therefore
\Cref{lem: lower_bounds_for_crossover_events} provides a $j\in J_q$
such that $s_j^*/s_j\ge 1/(6\sqrt{n})$. Using \Cref{lem:
  central_path_bounded_l1_norm} and \Cref{prop:near-central} we find that $s_j^+/s_j \le 2n$ and so $s_j^*/s_j^+ = s_j^*/s_j \cdot s_j/s_j^+ \ge 1/(12 n^{3/2}) > 1/(16 n^{3/2})$.

The final statement  $\rho^{\mu^+}(\ell,\ell')\ge -|J_q|$ for any
$\ell,\ell'\in J_q$ is also straightforward. From
\Cref{lem:near_uniform_shootdown}\ref{i:both-i-j} and the strong connectivity of $J_q$
in $G_{\delta,\gamma/n}$, we obtain that $J_q$ is strongly connected
in $G_{\delta^+,\gamma/(4n)}$. Hence,  $\rho^{\mu^+}(\ell,\ell')\ge -|J_q|$ follows
by \Cref{lem:rho-bound}.

The rest of the proof is dedicated to showing the existence of an $i\in J_q$
such that $x_i^* \ge \beta x_i^+/(16 n^{3/2})$. For this purpose, we will prove following claim.

\begin{restatementclaim}\label{claim:delta_x_star_bound}
$\|\delta_{J_q} x^*_{J_q}\| \geq \frac{\beta \mu^+}{8\sqrt{n\mu}}$.
\end{restatementclaim}

In order to prove \Cref{claim:delta_x_star_bound}, we define
\[
z := (\delta^+)^{-1} L^{\delta^+}_{J_{>q}}\left(\delta^+_{J_{>q}}(x^*_{J_{>q}}-x^+_{J_{>q}})\right)
\quad \text{ and } w := x^*-x^+-z\, ,
\]
as in \Cref{lem:force}. By construction, $w \in W$ and $w_{J_{>q}} =
0$. Thus, $w_{J_q} \in W_{{\cal
J},q}$ as defined in \Cref{sec:linsys}.

Using the triangle inequality, we get
\begin{equation}\label{eq:4-4-triangle}
\|\delta_{J_q}x^*_{J_q}\|\ge \|\delta_{J_q}(x^+_{J_q}+w_{J_q})\| -
\|\delta_{J_q}z_{J_q}\|\, .
\end{equation}
We bound the two terms separately, starting with an upper bound on $\|\delta_{J_q}z_{J_q}\|$.
Since
$\ell^{\delta^+}({\cal J}) \leq 4 \gamma n$, we have with \Cref{lem:force} that
\begin{equation}
\begin{aligned}
\left\|\delta^+_{J_q} z_{J_q}\right\| &\leq \ell^{\delta^+}({\cal J}) \left\|\delta^+_{J_{>q}}\left(x^*_{J_{>q}}-x^+_{J_{>q}}\right)\right\| \\
& \le 4n \gamma \left\|\delta^+_{J_{>q}}\left(x^*_{J_{>q}}-x^+_{J_{>q}}\right)\right\| \\
&= 4n \gamma \left\|\delta_{J_{>q}}^+ x_{J_{>q}}^+\left(\frac{x_{J_{>q}}^*}{x_{J_{>q}}^+} - e\right)\right\| \\
& \le 4n\gamma \left(\|\delta^+ x^+\|_\infty \cdot \left\|\frac{x^*}{x^+}\right\|_1 + \sqrt{n\mu^+}\right) \\
& \le 4n\gamma \left(\frac32 \sqrt{\mu^+} \cdot \frac43 n + \sqrt{n\mu^+}\right) \\
&\le  16n^2\sqrt{\mu^+}  \gamma,
\end{aligned}
\end{equation}
where the penultimate inequality follows by \Cref{prop:x_i-s_i} and \Cref{lem:
central_path_bounded_l1_norm}.
We can use this and
\Cref{lem:near_uniform_shootdown}\ref{i:scale-down} to obtain
\begin{equation}\label{eq:delta-z}
\|\delta_{J_q} z_{J_q}\| \leq
\|\delta_{J_q}/\delta^+_{J_q}\|_\infty\cdot \|\delta^+_{J_q}
z_{J_q}\| \leq \frac{32n^2\gamma\mu^+} {\sqrt{\mu}}\leq
\frac{\beta\mu^+}{32n^3\sqrt{\mu}}\, ,
\end{equation}
using the definition of $\gamma$.

 The first RHS term in \eqref{eq:4-4-triangle} will be bounded as follows.
\begin{restatementclaim}\label{claim:xi-bound}
$\|\delta_{J_q}(x^+_{J_q}+w_{J_q})\| \ge \frac12 \sqrt{\mu}\xi^\lal_{{\cal
    J}}$ .
\end{restatementclaim}

\begin{claimproof}[\Cref{claim:xi-bound}]
We recall the characterization  \eqref{eq:ll-primal}
of the LLS step $\Delta x^\lal \in W$. Namely, there exists
$\Delta s \in W_{{\cal J},1}^\perp \times \cdots \times W_{{\cal
    J},q}^\perp$ that is
the unique solution to $\delta^{-1} \Delta s + \delta \Delta x^\lal = -\delta
x$. From the above, note that
\[
\|\delta^{-1}_{J_q} \Delta s_{J_q}\| = \|\delta_{J_q} (x_{J_q}+\Delta x_{J_q}^\lal)\|=\sqrt{\mu} \|\Rx_{J_q}^\lal\|=
\sqrt{\mu}\xi^\lal_{{\cal J}}\,  .
\]
From the Cauchy-Schwarz inequality,
\begin{equation}\label{eq:4-4-cauchy}
\begin{aligned}
\|\delta^{-1}_{J_q} \Delta s_{J_q}\|\cdot
\|\delta_{J_q}(x^+_{J_q}+w_{J_q})\|&\ge\left|\left\langle
    \delta_{J_q}^{-1} \Delta s_{J_q},
    \delta_{J_q}(x^+_{J_q}+w_{J_q})\right\rangle\right|\\
&=\left|\left\langle
    \delta_{J_q}^{-1} \Delta s_{J_q},
    \delta_{J_q}x^+_{J_q}\right\rangle\right|\, .
\end{aligned}
\end{equation}
Here, we used that $\Delta s_{J_q}\in W^\perp_{{\cal J},q}$ and
$w_{J_q}\in W_{{\cal J},q}$.
Note that
\[
x^+=x+\alpha \Delta x^\lal=x+\Delta x^\lal- (1-\alpha)\Delta
x^\lal=-\delta^{-2}\Delta s- (1-\alpha)\Delta
x^\lal\, .
\]
Therefore,
\[
\begin{aligned}
\left|\left\langle    \delta_{J_q}^{-1} \Delta s_{J_q},
    \delta_{J_q}x^+_{J_q}\right\rangle\right|
&=\left|\left\langle    \delta_{J_q}^{-1} \Delta s_{J_q},
    -\delta_{J_q}^{-1} \Delta s_{J_q} - (1-\alpha) \delta_{J_q} \Delta
    x^\lal_{J_q}\right\rangle\right|\\
&\ge \|\delta_{J_q}^{-1} \Delta s_{J_q}\|^2 - (1-\alpha) \left|\left\langle    \delta_{J_q}^{-1} \Delta s_{J_q},\delta_{J_q} \Delta
    x^\lal_{J_q}\right\rangle\right|\, .
\end{aligned}
\]
By
\Cref{lem:prox-sum-space}, there exists $\Delta
\bar x \in W_{{\cal
J},1} \times \cdots \times W_{{\cal J},p}$ such that $\|\delta_{J_q}(\Delta
x_{J_q}^\lal-\Delta \bar x_{J_q})\| \leq
2n\ell^\delta({\cal J})\sqrt{\mu} $. Therefore, using the
orthogonality of $\Delta s_{J_q}$ and $\Delta \bar x_{J_q}$, we get
that
\[
\left|\left\langle    \delta_{J_q}^{-1} \Delta s_{J_q},\delta_{J_q}
    \Delta  x^\lal_{J_q}\right\rangle\right|=
\left|\left\langle    \delta_{J_q}^{-1} \Delta s_{J_q},\delta_{J_q}
    (\Delta  x^\lal_{J_q}-\Delta \bar
    x^\lal_{J_q})\right\rangle\right|\le 2n\ell^\delta({\cal
  J})\sqrt{\mu} \cdot \|\delta_{J_q}^{-1} \Delta s_{J_q}\|\, .
\]
From the above inequalities, we see that
\[
\|\delta_{J_q}(x^+_{J_q}+w_{J_q})\|\ge \|\delta^{-1}_{J_q} \Delta s_{J_q}\| -2(1-\alpha) n\ell^\delta({\cal
  J})\sqrt{\mu} =\sqrt{\mu}\xi^\lal_{{\cal J}} - 2(1-\alpha) n\ell^\delta({\cal
  J})\sqrt{\mu}\, .
\]
It remains to show $(1-\alpha) n\ell^\delta({\cal
  J})\le \xi^\lal_{{\cal J}}/4$. From
\Cref{lem:ll-decompose}\ref{i:progress}, we obtain
\[
(1-\alpha) n\ell^\delta({\cal  J})\le 3n^{3/2}\ell^\delta( {\cal J}) \xi^\lal_{{\cal
      J}}\beta^{-1},\,
\]
using $\xi^{\lal}_{\cal J} \ge \eps^{\lal}$.
The claim now follows by the assumption $\ell^\delta( {\cal J})\le
  \gamma$, and the choice of $\gamma$.
\end{claimproof}
\begin{claimproof}[\Cref{claim:delta_x_star_bound}]
  Using
  \Cref{lem:ll-decompose}\ref{i:progress},
  \[
  \mu^+\leq \frac{3\sqrt{n}\xi^\lal_{{\cal J}}\mu}{\beta},
  \]
  implying $\|\delta_{J_q}(x^+_{J_q}+w_{J_q})\|\geq \beta\mu^+/(6\sqrt{n\mu})$ by \Cref{claim:xi-bound}.
  Now the claim follows using \eqref{eq:4-4-triangle} and
  \eqref{eq:delta-z}.
\end{claimproof}
By \Cref{lem:near_uniform_shootdown}\ref{i:scale-down}, we see that
 \[
\|\delta_{J_q} x^+_{J_q}\|\leq \sqrt{n} \|\delta_{J_q}
x^+_{J_q}\|_\infty\leq \frac{ 3 \sqrt{n} \mu^+}{\sqrt{\mu}}\, .
\]
Thus, the lemma follows immediately from \Cref{claim:delta_x_star_bound}: for at least one $i\in J_q$, we
must have
\begin{equation*}
  \frac{x_i^*}{x_i} \ge  \frac{\|\delta_{J_q} x^*_{J_q}\|}{\|\delta_{J_q} x^+_{J_q}\|} \ge \frac{\beta}{24n} \ge \frac{\beta}{16n^{3/2}}\,. \qedhere
\end{equation*}

\end{proof}

\lemcaselayercrashes*
\begin{proof}[Proof of \Cref{lem: case_layer_crashes}]
  \label{proof:lem: case_layer_crashes}
Recall the sets $\llspartition B$ and $\llspartition N$ defined in \eqref{eq:B-N}. The key is to
show the existence of an edge
\begin{equation}\label{eq:i-j-prime}
 (i',j')\in E_{\delta^+,\gamma/(4n)}\quad \mbox{such that} \quad
 i'\in J_q\subseteq \llspartition B, \quad j'\in J_r\subseteq \llspartition N,\quad r<q\, .
\end{equation}
Before proving the existence of such $i'$ and $j'$, we show how the
rest of the statements follow. Note that $x^+ \le (1-\beta)^{-1}(1+2 \cdot 2\beta) nx \le \frac74 nx$ by \Cref{lem: central_path_bounded_l1_norm} and \Cref{prop:near-central}. Further, we have $\|\Rx^\lal_{J_q}\| - 2\gamma n \ge \frac12 - \frac34 \beta - 2\gamma n \ge \frac{2}{5}$ by \Cref{lem:ll-decompose} \ref{i:norm}. The existence of $i \in J_q$ such that $x_i^*\ge x^+_i/(8n^{3/2})$ now follows immediately from \Cref{lem: lower_bounds_for_crossover_events}, as there is an $i \in J_q$ such that 
\begin{equation}
  x_i^* \geq\frac{2x_i}{3\sqrt{n}}\cdot(\|\Rx_{J_q}^\lal\| - 2\gamma n) \ge \frac{2}{3\sqrt{n}} \frac{4x_i^+}{7n}\frac{2}{5} \ge \frac{x_i^+}{8n^{3/2}}.
\end{equation}
With analogous argumentation it can be shown that there exists $j \in J_r$ such that $s_j^*\ge s^+_j/(8n^{3/2})$.
The other statements are that
$\rho^{\mu^+}(i,j)\ge -|J_q\cup J_r|$, and for each
$\ell,\ell'\in J_q\cup J_r$, $\ell \neq \ell'$, $\Psi^\mu(\ell,\ell')\le |J_q\cup
J_r|$. According to \Cref{lem:rho-bound}, the latter is true (even with the
stronger bound $\max\{|J_q|,|J_r|\}$) whenever
$\ell,\ell'\in J_q$, or $\ell,\ell'\in J_r$, or if $\ell\in J_q$ and
$\ell'\in J_r$.
 It is left to show the lower bound on
$\rho^{\mu^+}(i,j)$ and $\Psi^\mu(\ell,\ell')\le |J_q\cup
J_r|$ for $\ell'\in J_q$ and $\ell\in J_r$.

From \Cref{lem:near_uniform_shootdown}\ref{i:both-i-j}, we have that if $\ell,\ell'\in
J_q\subseteq \llspartition B$ or $\ell,\ell'\in J_r\subseteq \llspartition N$, then
$ \Le_{\ell\ell'}^{\delta}/4\le {\Le_{\ell\ell'}^{\delta^+}}$.
Hence, the strong connectivity of $J_r$ and $J_q$ in $G_{\delta,\gamma}$
implies the strong connectivity of these sets in $G_{\delta^+,
\gamma/(4n)}$.
Together with the edge $(i',j')$, we see that every $\ell'\in J_q$ can reach every $\ell\in J_r$ on a
directed path of length $\le|J_q\cup J_r|-1$ in
$G_{\delta^+,\gamma/(4n)}$. Applying \Cref{lem:rho-bound} for this
setting, we obtain $\Psi^\mu(\ell,\ell')\le
\rho^{\mu^+}(\ell,\ell')\le |J_q\cup J_r|$ for all such pairs, and
also $\rho^{\mu^+}(i,j)\ge -|J_q\cup J_r|$.

\medskip

The rest of the proof is dedicated to showing the existence of $i'$
and $j'$ as in \eqref{eq:i-j-prime}.
We let $k \in [p]$ such that
$\ell^{\delta^+}(J_{\geq k}) = \ell^{\delta^+}(\mathcal J) >
4n\gamma$. To simplify the notation, we let $I=J_{\ge k}$.

When constructing $\cal J$ in \textsc{Layering}($\delta,\hat\kappa$), the
subroutine \textsc{Verify-Lift}($\diag(\delta)W,I,\gamma$) was called for the
set $I=J_{\geq k}$, with the answer `pass'.
 Besides $\ell^\delta(I)\le \gamma$, this guaranteed the stronger
 property that $\max_{ji}|B_{ji}|\le \gamma$ for the matrix $B$
 implementing the lift (see \Cref{rem:B-i-j}).

Let us recall how this matrix $B$ was obtained. The subroutine starts
by finding a minimal $I'\subset I$ such that $\dim(\pi_{I'}(W)) = \dim(\pi_I(W))$.
Recall that $\pi_{I'}(W) = \R^{I'}$ and  $L_I^\delta(p) = L_{I'}^\delta(p_{I'})$
for every $p \in \pi_I(\diag(\delta)W)$.

Consider the optimal lifting $L_I^\delta:\pi_I(\diag(\delta)W)\to \diag(\delta)W$.
We defined $B \in \R^{([n] \setminus I) \times I'}$ as  the matrix sending any $q \in \pi_{I'}(\diag(\delta)W)$
to the corresponding vector $[L_{I'}^\delta(q)]_{[n]\setminus I}$. The column $B_i$
can be computed as $[L_{I'}^\delta(e^i)]_{[n]\setminus I}$ for $e^i \in \R^{I'}$.

We consider the transformation
\[
{\bar B}:=\diag(\delta^+\delta^{-1})B\diag\big((\delta^+_{I'})^{-1}\delta_{I'}\big).
\]
This maps $\pi_{I'}(\diag(\delta^+)W)\to
\pi_{[n]\setminus I} (\diag(\delta^+)W)$.

Let $z \in\pi_I(\diag(\delta^+) W)$ be the singular vector corresponding to the
maximum singular value of $L_I^{\delta^+}$, namely,
 $\|[L_{I}^{\delta^+}(z)]_{[n]\setminus I}\| > 4n\gamma\| z\|$.
Let us normalize $z$ such that $\|z_{I'}\|=1$. Thus,
\[
\norm{[L_{I'}^{\delta^+}(z_{I'})]_{[n]\setminus I}}> 4n\gamma\, .
\]
Let us now apply ${\bar B}$ to $z_{I'}\in \pi_{I'}(\diag(\delta^+)W)$. Since
$L_I^{\delta^+}$ is the minimum-norm lift operator, we see
that
\[
\norm{{\bar B} z_{I'}}\ge \norm{[L_{I'}^{\delta^+}(z_{I'})]_{n\setminus I}}>
4n\gamma\, .
\]
We can upper bound  the operator norm by the Frobenius norm
$\|{\bar B}\| \leq \|{\bar B}\|_F = \sqrt{\sum_{ji} {{\bar B}_{ji}}^2} \leq n\max_{ji}
|{\bar B}_{ji}|$, and therefore
\[
\max_{ji} |{\bar B}_{ji}|> 4\gamma\, .
\]
Let us fix $i'\in I'$ and $j'\in [n]\setminus I$  as the indices giving
the maximum value of ${\bar B}$. Note that
${\bar B}_{j'i'}=B_{j'i'}\delta^+_{j'}\delta_{i'}/(\delta^+_{i'}\delta_{j'})$.

Let us now use Lemma~\ref{lem:identify} for the pair $i',j'$, the matrix $B$ and the
subspace $\diag(\delta)W$. Noting that
$B_{j'i'}=[L_{I'}^\delta(e^{i'})]_{j'}$, we obtain $\Le_{i'j'}^\delta\ge |B_{j'i'}|$.
Now,
\begin{align}\label{eq: lem44_bound_imbalance}
\Le_{i'j'}^{\delta^+} = \Le_{i'j'}^\delta \cdot \frac{\delta^+_{j'}\delta_{i'}}{\delta^+_{i'}\delta_{j'}}\ge
|B_{j'i'}|\cdot \frac{\delta^+_{j'}\delta_{i'}}{\delta^+_{i'}\delta_{j'}}= |{\bar B}_{j'i'}| >
4\gamma\, .
\end{align}

The next claim finishes the proof.
\begin{claim}
For $i'$ and $j'$ selected as above, \eqref{eq:i-j-prime} holds.
\end{claim}
\begin{proof}
$(i',j') \in E_{\delta^+, \gamma/(4n)}$ holds by \eqref{eq: lem44_bound_imbalance}.
From the above, we have
\[
|B_{j'i'}| > 4\gamma\cdot \frac{\delta^+_{i'}
  \delta_{j'}}{\delta_{i'} \delta^+_{j'}}\, .
\]
According to \Cref{rem:B-i-j}, $|B_{j'i'}|\le \gamma$ follows since
\textsc{Verify-Lift}($\diag(\delta)W,I,\gamma$) returned with `pass'.
We thus have
\[
 \frac{\delta^+_{i'}
  \delta_{j'}}{\delta_{i'} \delta^+_{j'}} < \frac{1}{4}.
\]
\Cref{lem:near_uniform_shootdown} excludes the
scenarios $i',j'\in \llspartition N$, $i',j'\in \llspartition B$, and $i'\in \llspartition N$, $j'\in \llspartition B$,
leaving  $i'\in \llspartition B$ and $j'\in \llspartition N$ as  the
only possibility. Therefore, $i'\in J_q\subseteq \llspartition B$ and $j'\in
J_r\subseteq \llspartition N$. We have $r<q$ since $i\in I=J_{\ge k}$ and $j\in
[n]\setminus I=J_{<k}$.
\end{proof}
\phantom{\qedhere}
\end{proof}

\section{Initialization}\label{sec:initialization}
Our main algorithm (\Cref{alg:overall} in Section~\ref{sec:
  the_algorithm}), requires an initial solution $w^0=(x^0,y^0,s^0)\in
{\cal N}(\beta)$. In this section, we remove this assumption by adapting the initialization
method of \cite{Vavasis1996} to our setting.

We use the ``big-$M$ method'', a standard initialization approach for
path-following interior point methods that introduces an auxiliary
system whose optimal solutions map back to the optimal solutions of the original system.
The primal-dual system we consider is
\begin{align}
\label{initialization_LP} \tag{Init-LP}
\begin{aligned}
\min \; c^\top x + &Me^\top \ubarx & \max \; y^\top b + 2 &Me^\top z \\
Ax - A\ubarx& = b & A^\top y + z + s &= c\\
x + \barx &= 2Me & z + \bars &= 0 \\
x,\barx,\ubarx& \geq 0 &  -A^\top y + \ubars &= Me \\
& & s,\bars, \ubars &\geq 0.
\end{aligned}
\end{align}
The constraint matrix used in this system is
\begin{align*}
\hat{A} =
\begin{pmatrix}
A & -A & 0 \\
I & 0  & I \\
\end{pmatrix}
\end{align*}
The next lemma asserts that the $\bar \chi$ condition number of $\hat
A$ is not much bigger than that of $A$ of the original system
\eqref{LP_primal_dual}.
\begin{lemma}[{\cite[Lemma 23]{Vavasis1996}}] \label{vy: extended_chibar}
$\bar{\chi}_{\hat{A}} \leq 3\sqrt{2}(\chibar + 1).$
\end{lemma}
We extend this bound for $\bar\chi^*$.
\begin{lemma}
${\bar\chi}^*_{\hat A} \leq 3\sqrt{2}(\bar\chi_A^*+1)$.
\end{lemma}
\begin{proof}
Let $D \in \mathbf D_n$ and let $\hat D \in \mathbf D_{3n}$ the matrix consisting of three copies of $D$, i.e.
\begin{align*}
\hat{D} &=
\begin{pmatrix}
D & 0 & 0 \\
0 & D & 0 \\
0 & 0 & D
\end{pmatrix}\, .
\intertext{Then}
\hat A \hat D &=
\begin{pmatrix}
AD & - AD & 0 \\
D & 0 & D
\end{pmatrix}\,
\end{align*}
Row-scaling does not change $\bar\chi$ as the kernel of the matrix
remains unchanged. Thus, we can rescale the last $n$ rows of $\hat A \hat D$, to the identity matrix, i.e. multiplying by $(I, D^{-1})$ from the left hand side.
We observe that
\begin{align*}
\bar\chi_{\hat{A}\hat{D}} = \bar\chi \left(
\begin{pmatrix}
AD & -AD & 0 \\
I & 0 & I
\end{pmatrix} \right)
\leq 3\sqrt{2} (\bar\chi_{AD} + 1)
\end{align*}
where the inequality follows from \Cref{vy: extended_chibar}.
The lemma now readily follows as
\begin{align*}
{\bar\chi}^*_{\hat A} &= \inf\{\bar\chi_{\hat A \hat D} : D \in \mathbf D_{3n} \}
\leq \inf\{3\sqrt{2}(\bar\chi_{AD} + 1): D \in \mathbf D_n\} = 3\sqrt{2}(\bar\chi_A^*+1). \qedhere
\end{align*}
\end{proof}
We show next that the optimal solutions of the original system are preserved for sufficiently large $M$. We let $d$ be the min-norm solution to $Ax=b$, i.e., $d = A^\top(AA^\top)^{-1}b$.
\begin{proposition} \label{prop: opt_solutions_coincide}
Assume both primal and dual of \eqref{LP_primal_dual} are feasible,
and $M > \max\{(\bar{\chi}_{A}+1)\|c\|, \bar{\chi}_{A}\|d\|\}$.
Every optimal solution $(x,y,s)$ to \eqref{LP_primal_dual}, can be
extended to an optimal solution $(x,\ubar x,\bar x,
y,z,s,\ubar s, \bar s)$ to \eqref{initialization_LP}; and conversely,
from every optimal solution $(x,\ubar x,\bar x,
y,z,s,\ubar s, \bar s)$ to \eqref{initialization_LP}, we obtain an
optimal solution $(x,y,s)$  by deleting the auxiliary variables.
\end{proposition}
\begin{proof}
If system \eqref{LP_primal_dual} is feasible, it admits a basic
optimal solution $(x^*,y^*,s^*)$ with basis $B$ such that $A_Bx_B^* = b, x^* \geq 0, A_B^\top y^* = c$ and $A^\top y^* \leq c.$ Using \Cref{prop:chibar}\ref{i:submatrix_characterisation} we see that
\begin{align}
\|x_B^*\| &= \|A_B^{-1}b\| = \|A_B^{-1}Ad\| \leq \bar{\chi}_{A}\|d\| < M \, , \label{ineq: primal_prop} \\
\intertext{and using that $\|A\| = \|A^\top\|$ we observe}
\|A^\top y^*\| &= \|A^\top A_B^{-\top}c\| \leq \|A^\top A_B^{-\top}\|\|c\| = \|A_B^{-1}A\|\|c\| \leq \chibar \|c\| <  M. \label{ineq: dual_prop}
\end{align}
We can extend this solution to a solution of system
\eqref{initialization_LP} via setting $\barx^* = 2Me - x^*, \ubarx^*
=0, z^* = \bars^* = 0$ and $\ubars^* = Me + A^\top y^*$. Observe that
$\barx^* > 0$ and $\ubars^* > 0$ by \eqref{ineq: primal_prop} and
\eqref{ineq: dual_prop}. Furthermore observe that by complementary
slackness this extended solution for \eqref{initialization_LP} is an
optimal solution. The property that $\ubars^* > 0$ immediately tells
us that $\ubarx$ vanishes for all optimal solutions of
\eqref{initialization_LP} and thus all optimal solutions of
\eqref{LP_primal_dual} coincide with the optimal solutions of
\eqref{initialization_LP}, with the auxiliary variables removed.
\end{proof}

The next lemma is from \cite[Lemma 4.4]{Monteiro2003}. Recall that
$w=(x,y,s)\in {\cal N}(\beta)$ if $\|xs/\mu(w)-e\|\le \beta$.

\begin{lemma}\label{lem:approx-mu}
Let $w=(x,y,s)\in {\cal P}^{++}\times {\cal D}^{++}$, and let
$\nu>0$. Assume that $\|xs/\nu -e\|\le \tau$. Then
$(1-\tau/\sqrt{n})\nu\le \mu(w)\le (1+\tau/\sqrt{n})\nu$ and $w\in {\cal N}(\tau/(1-\tau))$.
\end{lemma}

The new system has the advantage that we can easily initialize the system with a feasible solution in close proximity to central path:
\begin{proposition}
We can initialize system \eqref{initialization_LP} close to the central path with initial solution $w^0 = (x^0, y^0, s^0) \in \mathcal{N}(1/8)$ and parameter $\mu(w^0) \approx M^2$ if $M > 15\max\{(\bar\chi_A + 1)\|c\|, \bar\chi_A \|d\|\}$.
\end{proposition}
\begin{proof}
The initialization follows along the lines of \cite[Section
10]{Vavasis1996}. We let $d$ as above,
and set
\begin{align*}
\barx^0 &= Me, x^0 = Me, \ubarx^0 = Me -d \\
y^0 &= 0, z^0 = -Me \\
\bars^0 &= Me, s^0 = Me + c, \ubars^0 = Me.
\end{align*}
This is a feasible primal-dual solution to system \eqref{initialization_LP} with parameter
\begin{align*}
\mu^0 &= (3n)^{-1}(\pr{x^0}{s^0} + \pr{\ubarx^0}{\ubars^0} + \pr{\barx^0}{\bars^0 }) = (3n)^{-1} (3nM^2 + Mc^\top e - Md^\top e)\approx M^2\, .
\end{align*}
We see that
\begin{align*}
\left\|\frac{1}{M^2}\begin{pmatrix}\barx^0 \bars^0 \\ x^0s^0 \\ \ubarx^0 \ubars^0 \end{pmatrix} - e\right\|^2 &=  M^{-2}\|c\|^2 + M^{-2}\|d\|^2 \leq \frac{1}{9^2\bar\chi_A^2} \leq \frac{1}{9^2}.
\end{align*}
With \Cref{lem:approx-mu} we conclude that $w^0 = (x^0, y^0, s^0) \in \mathcal{N}\left(\frac{1/9}{1-1/9}\right) = \mathcal{N}(1/8)$.
\end{proof}

\paragraph{Detecting infeasibility} To use the extended system
\eqref{initialization_LP}, we still need to assume that both the
primal and dual programs in \eqref{LP_primal_dual} are feasible. For
arbitrary instances, we first need to check if this is the case, or
conclude that the primal or the dual (or both) are infeasible.

This can be done by employing a two-phase method. The first phase decides feasibility by running \eqref{initialization_LP} with data $(A, b, 0)$ and $M > \bar\chi_A \|d\|_1$. The objective value of the optimal primal-dual pair is 0 if and only if \eqref{LP_primal_dual} has a feasible solution. If the optimal primal/dual solution $(x^*, \ubarx^*, \barx^*, y^*, z^*, s^*, \ubars^*, \bar s^*)$ has positive objective value, we can extract an infeasibility certificate in the following way. 

We can w.l.o.g.\ assume that $x^*$ is supported on some basis $B$ of $A$.
Note that the objective function of the primal is equivalent to $\|\ubar x\|_1$. Therefore, clearly $\|\ubarx^*\|_1 \le -\sum_{i : d_i < 0} d_i \le \|d\|_1$ and so $\|\ubarx^*\| \le \|d\|_1$. Due to the constraint $Ax^* - A\ubar x^* = b = Ad$ we get that 
\begin{equation}
  \|x^*\| = \|B^{-1}A(d + \ubar x^*)\| \le \|B^{-1}A\|(\|d\| + \|\ubar x^*\|) \le 2\bar\chi_A \|d\|_1.
\end{equation}
Therefore, if $M > \bar\chi_A \|d\|_1$, then $\bar x^* = 2Me - \|x^*\| > 0$ so by strong duality, $\bar s^* = 0$. From the dual, we conclude that $z^* = 0$, and therefore $A^\top y^* \leq A^\top y^* + s^* + z^* = c = 0$.
On the other hand, by assumption the objective value of the dual is positive, and so ${(y^*)}^\top b \geq {(y^*)}^\top b + 2M e^\top z^* > 0$. Hence, $y^*$ is the desired certificate.

Feasibility of the dual of \eqref{LP_primal_dual} can be decided by running \eqref{initialization_LP} on data $(A,0,c)$ and $M > (\bar\chi_A + 1)\|c\|$ with the same argumentation: Either the objective value of the dual is 0 and therefore the dual optimal solution $(y^*,z^*, \ubars^*, s^*, \bar s^*)$ corresponds to a feasible dual solution of \eqref{LP_primal_dual} or the objective value is negative and we extract a dual infeasibility certificate in the following way:
For the optimal corresponding primal solution $(x^*, \ubar x^*, \bar x^*)$ we have by assumption $c^\top x^* \leq c^\top x^* + Me^\top \ubarx^* < 0$. Furthermore, w.l.o.g.\ the support of $s^*$ is contained in a basis which allows us to conclude that $\ubars^* > 0$ and therefore $\ubarx^* = 0$. So we have $Ax^* = 0 + A\ubar x^* = 0$, which together with $c^\top x^* < 0$ yields the certificate of dual infeasibility.

\paragraph*{Finding the right value of $M$}
While \Cref{alg:overall} does not require any estimate on
$\bar\chi^*$ or $\bar\chi$, the initialization needs to set $M \ge
\max\{(\bar{\chi}_{A}+1)\|c\|, \bar{\chi}_{A}\|d\|\}$ as in
\Cref{prop: opt_solutions_coincide}.

A straightforward guessing approach (attributed to J. Renegar in
\cite{Vavasis1996}) starts with a constant guess, say
$\bar\chi_A=100$, constructs the extended system, and runs the algorithm. In
case the optimal solution to the extended system does not map to an
optimal solution of \eqref{LP_primal_dual}, we restart with
$\bar\chi_A=100^2$ and try again; we continue squaring the guess until
an optimal solution is found.

This would still require a series of
$\log\log\bar\chi_A$ guesses, and thus, result in a dependence on
$\bar\chi_A$ in the running time.
However, if we initially rescale our system using the near-optimal
rescaling Theorem~\ref{thm:bar-chi-star}, then we can turn the dependence
from $\bar\chi_A$ to $\bar\chi^*_A$. The overall
iteration complexity remains $O(n^{2.5}\log n\log(\bar\chi^*_A+n))$, since the
running time for the final guess on $\bar\chi^*_A$ dominates the total
running time of all previous computations due to the repeated squaring.

An alternative approach, that does not rescale the system, is to use \Cref{thm:bar-chi-star} to approximate $\bar\chi_A$. In this case we repeatedly square a guess of $\bar\chi_A^*$ instead of $\bar\chi_A$ which takes $\mathcal{O}(\log \log \bar\chi_A^*)$ iterations until our guess corresponds to a valid upper bound for $\bar\chi_A$.

Note that either guessing technique can handle bad guesses gracefully. For the first phase, if neither a feasible solution to \eqref{LP_primal_dual} is returned nor a Farkas' certificate can be extracted, we have proof that the guess was too low by the above paragraph.
Similarly, in phase two, when feasibility was decided in the affirmative for primal and dual, an optimal solution to \eqref{initialization_LP} that corresponds to an infeasible solution to \eqref{LP_primal_dual} serves as a certificate that another squaring of the guess is necessary.

\subsection*{Acknowledgement} The authors are grateful to the anonymous reviewers for their comments that helped to improve the presentation.
\bibliographystyle{myalpha}
\bibliography{mybib}{} 

\end{document}